\documentclass[bj,preprint]{imsart}

\def\journal@name{}

\let\counterwithin\relax
\usepackage{chngcntr}

\usepackage{times}
\usepackage{bm}
\usepackage[round]{natbib}

\usepackage[usenames,dvipsnames]{color}

\usepackage[latin1]{inputenc}
\usepackage{color,graphics}
\usepackage{dsfont}
\usepackage{amsmath,amssymb}
\usepackage{graphicx,subfigure}
\usepackage[plain,noend]{algorithm2e}

\usepackage{chngcntr}

\usepackage{appendix}

\usepackage{tikz}
\usetikzlibrary{arrows}
\usetikzlibrary{fit}
\usetikzlibrary{calc}
\usetikzlibrary{positioning}

\makeatother
\newtheorem{assumption}{Assumption}
\newtheorem{theorem}{Theorem}

\newtheorem{corollary}[theorem]{Corollary}

\newtheorem{definition}[theorem]{Definition}
\newtheorem{example}[theorem]{Example}

\newtheorem{lemma}[theorem]{Lemma}

\newtheorem{proposition}[theorem]{Proposition}
\newtheorem{remark}[theorem]{Remark}

\newenvironment{proof}[1][Proof]{\noindent\textbf{#1.} }{\ \rule{0.5em}{0.5em}}
\newenvironment{sketchproof}[1][Sketch of the proof]{\noindent\textbf{#1.} }{\ \rule{0.5em}{0.5em}}

\newcommand{\R}{\mathbb{R}}
\newcommand{\Nedges}{N_{\alpha}^{(e)}}

\newcommand{\var}{\mathrm{var}}
\newcommand{\1}[1]{\mathds{1}_{#1}}
\def\R{\mathbb R}
\renewcommand{\Pr}{\mathrm{pr}}
\graphicspath{{./figures/}}

\makeatletter
\renewcommand{\algocf@captiontext}[2]{#1\algocf@typo. \AlCapFnt{}#2} 
\def\@algocf@capt@plain{top}
\renewcommand{\algocf@makecaption}[2]{%
  \addtolength{\hsize}{\algomargin}%
  \sbox\@tempboxa{\algocf@captiontext{#1}{#2}}%
  \ifdim\wd\@tempboxa >\hsize
    \hskip .5\algomargin%
    \parbox[t]{\hsize}{\algocf@captiontext{#1}{#2}}
  \else%
    \global\@minipagefalse%
    \hbox to\hsize{\box\@tempboxa}
  \fi%
  \addtolength{\hsize}{-\algomargin}%
}
\makeatother

\usepackage{amsmath}
\usepackage{xr-hyper}
\usepackage{hyperref}

\begin{document}

\begin{frontmatter}
\begin{aug}
\title{On sparsity, power-law and clustering properties of graphex processes}
\date{\today}
\author{Fran\c cois Caron\ead[label=e1]{caron@stats.ox.ac.uk}},
\author{Francesca Panero\ead[label=e2]{f.panero@lse.ac.uk}}
\and
\author{Judith Rousseau\ead[label=e3]{judith.rousseau@stats.ox.ac.uk}}

\affiliation{University of Oxford}

\address{Department of Statistics, University of Oxford. }
\end{aug}

\begin{abstract}
This paper investigates properties of the class of graphs based on exchangeable point processes. We provide asymptotic expressions for the number of edges, number of nodes and degree distributions, identifying four regimes: (i) a dense regime, (ii) a sparse almost dense regime, (iii) a sparse regime with power-law behaviour, and (iv) an almost extremely sparse regime. We show that, under mild assumptions, both the global and local clustering coefficients converge to constants which may or may not be the same. We also derive a central limit theorem for subgraph counts and for the number of nodes.  Finally, we propose a class of models within this framework where one can separately control the latent structure and the global sparsity/power-law properties of the graph.
\end{abstract}



\end{frontmatter}

\section{Introduction}

The ubiquitous availability of large, structured network data in various scientific areas ranging from biology to social sciences has been a driving force in the development of statistical network models~\citep{Kolaczyk2009,Newman2010}. Vertex-exchangeable random graphs, also known as $W$-random graphs or graphon models~\citep{Hoover1979,Aldous1981,Lovasz2006,Diaconis2008} offer in particular a flexible and tractable class of random graph models. It includes many models, such as the stochastic block-model~\citep{Nowicki2001}, as special cases. Various parametric and nonparametric model-based approaches~\citep{Palla2010,Lloyd2012,Latouche2016}, or nonparametric estimation procedures~\citep{Wolfe2013,Chatterjee2015,Gao2015} have been developed within this framework. Although very flexible, it is known that vertex-exchangeable random graphs are dense~\citep{Lovasz2006,Orbanz2015}, that is the number of edges scales quadratically with the number of nodes; this property is considered unrealistic for many real-world networks.
\begin{figure}[t]
\begin{center}
\includegraphics[width=.45\textwidth]{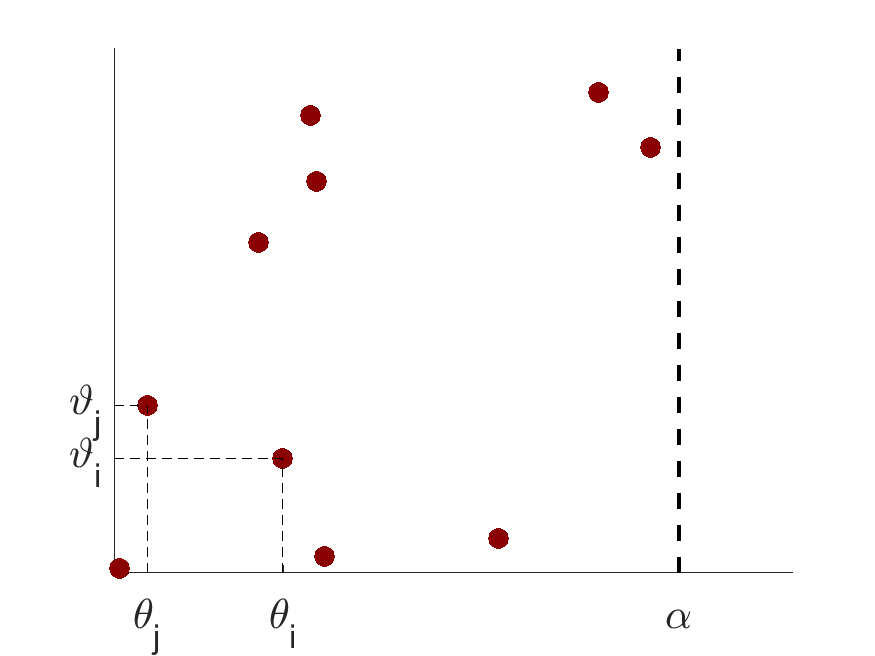}
\includegraphics[width=.45\textwidth]{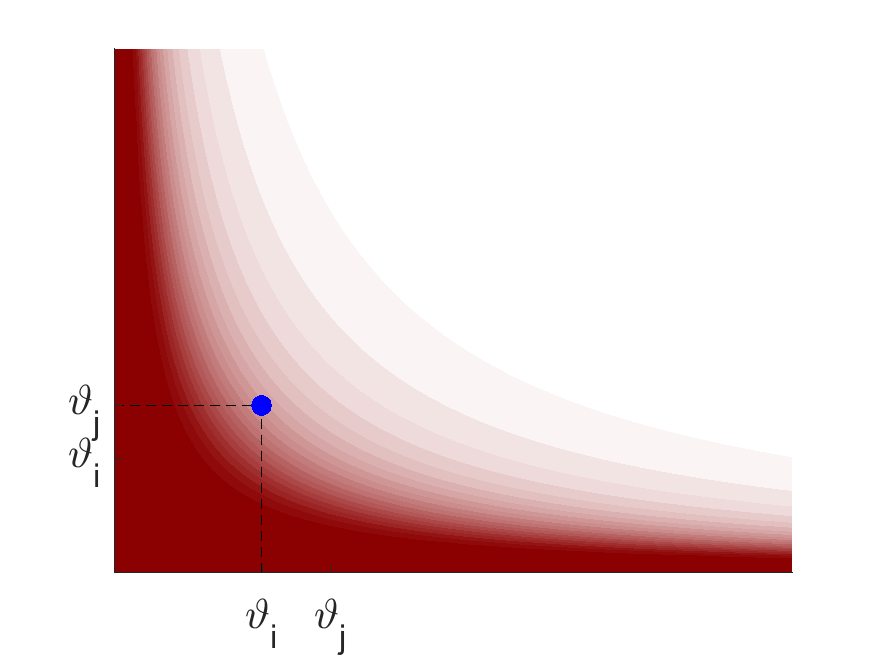}
\end{center}
\caption{Illustration of the graph model based on exchangeable point processes. (left) A unit-rate Poisson process
$(\theta_{i},\vartheta_{i})$, $i\in\mathbb{N}$ on $(0,\alpha]\times
\mathbb{R}_{+}$. (right) For each pair $i\leq j$, set $Z_{ij}=Z_{ji}=1$ with probability
$W(\vartheta_{i},\vartheta_{j})$. Here, $W$ is indicated by the red shading
(darker shading indicates higher value). Similar to Figure 5 in \citep{Caron2017}.}%
\label{fig:Kallenberg}
\end{figure}
To achieve sparsity, rescaled graphon models have been proposed in the literature~\citep{BollobasRiordan2009,BickelChen2009,BickelChenLevina2011,Wolfe2013}. While these models can capture sparsity, they are not projective; additionally, standard rescaled graphon models cannot simultaneously capture sparsity and a clustering coefficient bounded away from 0 (see Section \ref{sec:discussion}).

These limitations are overcome by another line of works initiated by~\cite{Caron2017}, \cite{Veitch2015} and \cite{Borgs2018}. They showed that,  by modeling the graph as an exchangeable point process, the classical vertex-exchangeable/graphon framework can be naturally extended to the sparse regime, while preserving its flexibility and tractability. In such a representation, introduced by \cite{Caron2017}, nodes are embedded at some location $\theta_i\in\mathbb R_+$, and the set of edges is represented by a point process on the plane
\begin{equation}
\sum_{i,j} Z_{ij} \delta_{(\theta_i,\theta_j)}
\label{eq:pointprocess}
\end{equation}
where $Z_{ij}=Z_{ji}$ is a binary variable indicating if there is an edge between node $\theta_i$ and node $\theta_j$. Finite-size graphs are obtained by restricting the point process \eqref{eq:pointprocess} to points $(\theta_i,\theta_j)$ such that $\theta_i,\theta_j\leq\alpha$, with $\alpha$ a positive parameter controlling the size of the graph. Focusing on a particular construction as a case study, \cite{Caron2017} showed that one can obtain sparse and exchangeable graphs within this framework; they also pointed out that exchangeable random measures admit a representation theorem due to~\cite{Kallenberg1990}, giving a general construction for such graph models. \cite{Herlau2016}, \cite{Todeschini2016} developed sparse graph models with (overlapping) community structure within this framework.  \cite{Veitch2015} and \cite{Borgs2018} showed how such construction naturally generalizes the dense exchangeable graphon framework to the sparse regime, and analysed some of the properties of the associated class of random graphs, called \textit{graphex processes}\footnote{\cite{Veitch2015} introduced the term \textit{graphex}. In the same paper, they referred to the class of random graphs as \textit{Kallenberg exchangeable graphs}, but the term \text{graphex processes} is now more commonly used.}; further properties were derived by \cite{Janson2016,Janson2017a}, \cite{Veitch2016} and \cite{Borgs2019a}.  Following the notations of \cite{Veitch2015}, and ignoring additional terms corresponding to stars and isolated edges, the graph is  then parameterised by a symmetric measurable function $W:\mathbb R^2_+\rightarrow[0,1]$, where for each $i\leq j$,
\begin{equation} \label{eq:Zij} 
Z_{ij}\mid (\theta_k,\vartheta_k)_{k=1,2,\ldots}\sim \text{Bernoulli}\{W(\vartheta_i,\vartheta_j)\},
\end{equation}
where $(\theta_k,\vartheta_k)_{k=1,2,\ldots}$ is a unit-rate Poisson process on $\mathbb R^2_+$. See Figure~\ref{fig:Kallenberg} for an illustration of the model construction.
The function $W$ is a natural generalisation of the graphon for dense exchangeable graphs~\citep{Veitch2015,Borgs2018} and we refer to it as the graphon function.

This paper investigates asymptotic properties of the general class of graphs based on exchangeable point processes defined by Equations \eqref{eq:pointprocess} and \eqref{eq:Zij}. Our findings can be summarised as follows.
\begin{itemize}
\item[(i)] We relate the sparsity and power-law properties of the graph to the tail behaviour of the marginal of the graphon function $W$, identifying four regimes: a) a dense regime, b) a sparse (almost dense) regime without power-law behaviour, c) a sparse regime with power-law behaviour, and d) an almost extremely sparse regime. In the sparse, power-law regime, the power-law exponent is in the range $(1,2)$.
\item[(ii)] We derive the asymptotic properties of the global and local clustering coefficients, two standard measures of the transitivity of the graph.
\item [(iii)] We give a central limit theorem for subgraph counts and for the number of nodes in the graph.
\item[(iv)] We introduce a parametrisation that allows to model separately the global sparsity structure and other local properties such as community structure. Such a framework enables us to sparsify any dense graphon model, and to characterise its sparsity properties.
\item[(v)] We show that the results apply to a wide range of sparse and dense graphex processes, including the models studied by \cite{Caron2017}, \cite{Herlau2016} and \cite{Todeschini2016}.
\end{itemize}

\begin{figure}[t!]
\centering
\subfigure[Power-law Degree distribution]{\includegraphics[width=.32\textwidth]{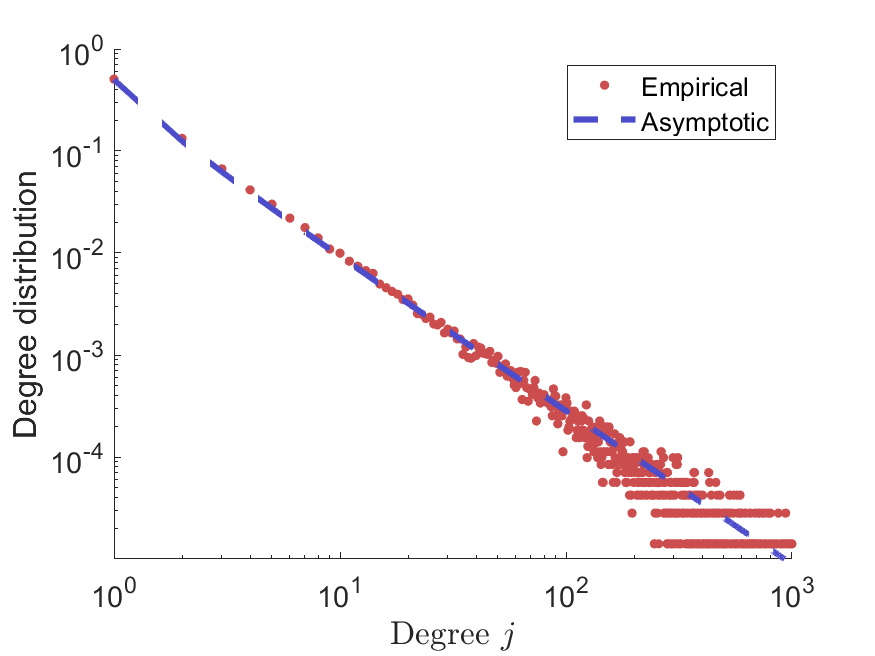}}
\subfigure[Average local and global clustering coefficients]{\includegraphics[width=.32\textwidth]{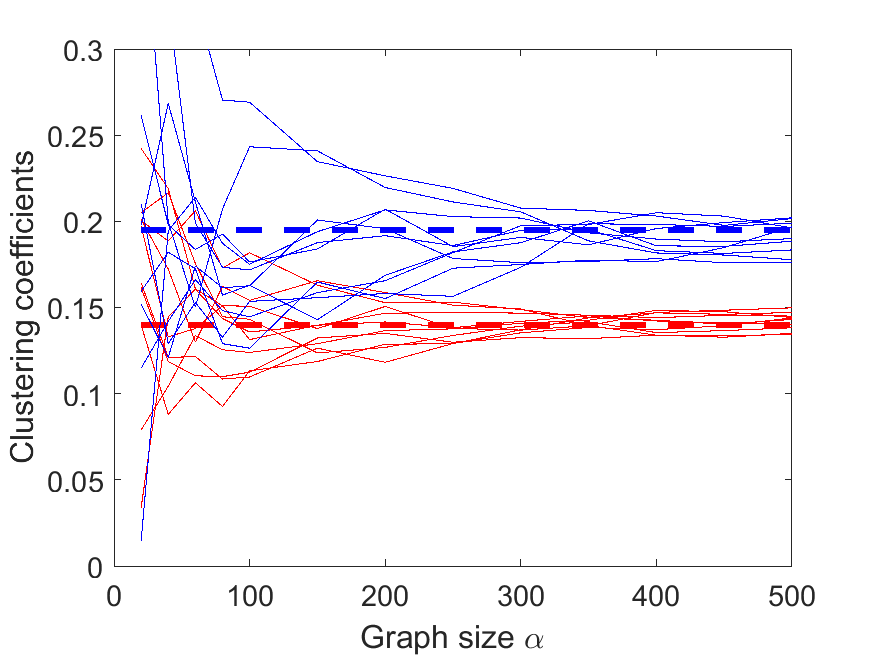}}
\subfigure[Local clustering coefficient per degree]{\includegraphics[width=.32\textwidth]{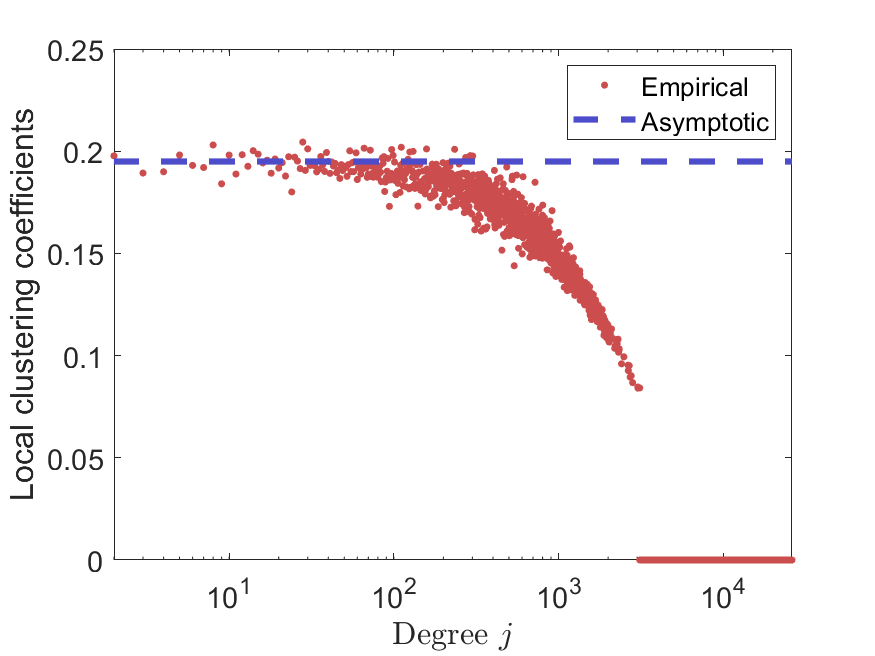}}
\caption{Illustration of some of the asymptotic results developed in this paper, applied to the generalised graphon model defined by Equations \eqref{eq:graphonCF} and \eqref{eq:GGP} with $\sigma_0=0.2$ and $\tau_0=2$. (a) Empirical degree distribution for a graph of size $\alpha=1000$ (red) and asymptotic degree distribution (dashed blue, see Corollary \ref{th:sparsity1}). (b) Average local (blue) and global (red) clustering coefficients for 10 graphs of growing sizes. Limit values are represented by dashed lines (see Propositions \ref{prop:globalclustering} and \ref{prop:localclustering}). (c) Local clustering coefficient for nodes of a given degree $j$, for a graph of size $\alpha=1000$. The limit value is represented by a dashed line (see Proposition \ref{prop:localclustering}).}
\label{fig:illustration}
\end{figure}

Some of the asymptotic results are illustrated in Figure~\ref{fig:illustration} for a specific graphex process in the sparse, power-law regime.

The article is organised as follows. In Section~\ref{sec:main} we give the notations and the main Assumptions. In Section~\ref{sec:asympstats}, we derive the asymptotic results for the number of nodes, degree distribution and clustering coefficients. In Section~\ref{sec:clt}, we derive central limit theorems for subgraphs and for the number of nodes. Section~\ref{sec:discussion} discusses related work. In Section \ref{sec:examples} we provide specific examples of sparse and dense graphs and show how to apply the results of the previous section to those models. In Section~\ref{sec:localglobal} we describe a generic construction for graphs with local/global structure and adapt some results of Section \ref{sec:asympstats} to this setting. Most of the proofs are given in the main text, with some longer proofs in the Appendix, together with some technical lemma and background material. Other more technical proofs are given in a Supplementary Material \citep{CaronRousseau2017:supplement}.\smallskip

Throughout the document, we use the notations $X_\alpha \sim Y_\alpha$ and $X_\alpha=o(Y_\alpha)$ respectively for $X_\alpha/Y_\alpha\rightarrow 1$ and $X_\alpha/Y_\alpha\rightarrow 0$. Both notations $X_\alpha\lesssim Y_\alpha$ and $X_\alpha=O(Y_\alpha)$ are used for $\lim\sup X_\alpha/Y_\alpha<\infty$. The notation $X_\alpha\asymp Y_\alpha$ means both $X_\alpha\lesssim Y_\alpha$ and $Y_\alpha\lesssim X_\alpha$ hold. All unspecified limits are when $\alpha$ tends to infinity. When $X_\alpha$ and/or $Y_\alpha$ are random quantities, the asymptotic relation is meant to hold almost surely.

\section{Notations and Assumptions}\label{sec:main}
\subsection{Notations}
\label{sec:notations}

Let $M=\sum_{i}\delta_{(\theta_{i},\vartheta_{i})}$ be a unit-rate Poisson
random measure on $(0,+\infty)^{2}$ and $W:[0,+\infty)^{2}\rightarrow
\lbrack0,1]$ a symmetric measurable function such that $\lim_{x\to\infty} W(x,x)$ and $\lim_{x\to 0} W(x,x)$ both exist \footnote{By \eqref{eq:integrabilityW}, this implies $\lim_{x\to\infty} W(x,x)=0$.} and
\begin{equation}
0<\overline W=\int_{\mathbb{R}_{+}^{2}}W(x,y)dxdy<\infty, \quad \int_{0}^{\infty}W(x,x)dx<\infty. \label{eq:integrabilityW}
\end{equation}

Let $(U_{ij})_{i,j\in\mathbb{N}^{2}}$ be a
symmetric array of independent random variables, with $U_{ij}\sim U(0,1)$ if $i\leq j$ and $U_{ij}=U_{ji}$ for $i>j$. Let $Z_{ij}=\1{U_{ij}\leq W(\vartheta_{i},\vartheta_{j})}$ be a binary random variable indicating if there is a link between $i$ and $j$, where $\1{A}$ denotes the indicator function.

Restrictions of the point process $\sum_{ij} Z_{ij}\delta_{(\theta_i,\theta_j)}$ to squares $[0,\alpha]^2$ then define a growing family of random graphs $(\mathcal G_\alpha)_{\alpha\geq 0}$, called a graphex process, where $\mathcal G_\alpha=(\mathcal V_\alpha,\mathcal E_\alpha)$ denotes a graph of size $\alpha\geq 0$ with vertex set $\mathcal V_\alpha$ and edge set $\mathcal E_\alpha$, defined by
\begin{align}
\mathcal V_\alpha &=\left \{ \theta_i \mid \theta_i\leq \alpha\text{ and }\exists \theta_k\leq \alpha\text{ s.t. }Z_{ik}=1\right \} \\
\mathcal E_\alpha &=\left \{ \{\theta_i,\theta_j\} \mid \theta_i,\theta_j\leq \alpha\text{ and }Z_{ij}=1\right \}.
\end{align}
The connection between the point process and graphex process is illustrated in Figure~\ref{fig:pointgraph}. The conditions \eqref{eq:integrabilityW}  are sufficient (though not necessary) conditions for $|\mathcal E_\alpha|$ (hence $|\mathcal V_\alpha|$) to be almost surely finite, and the graphex process  well defined~\cite[Theorem 4.9]{Veitch2015}. Note crucially that the graphs $\mathcal G_\alpha$ have no isolated vertices (that is, no vertices of degree 0), and that the number of nodes $|\mathcal V_\alpha|$ and edges $|\mathcal E_\alpha|$ are both random variables.\medskip

\begin{figure}
\centering
\begin{minipage}{0.45\textwidth}
\subfigure[Point process $\sum_{ij} Z_{ij}\delta_{(\theta_i,\theta_j)}$]{\scalebox{1.15}{\begin{tikzpicture}[node distance=1.4cm, auto,>=latex',scale=.8]
\draw[->,ultra thick] node[left] {0} (0,0) -- (6,0) node[right] {};
\draw[->,ultra thick] (0,0) -- (0,-6) node[right] {};
\draw[fill=blue,opacity=0.1] (0,0) -- (0,-6) -- (6,-6) -- (6,0)  ;
\draw[dashed] (1,0) -- (1,-1) node[right] {};
\draw[dashed] (3,0) -- (3,-3) node[right] {};
\draw[dashed] (3.5,0) -- (3.5,-3.5) node[right] {};
\draw[dashed] (5,0) -- (5,-5) node[right] {};
\draw[dashed] (0,-1) -- (1,-1) node[right] {};
\draw[dashed] (0,-3) -- (3,-3) node[right] {};
\draw[dashed] (0,-3.5) -- (3.5,-3.5) node[right] {};
\draw[dashed] (0,-5) -- (5,-5) node[right] {};
\node at (1,0.5) {1};
\node at (3,0.5) {3};
\node at (3.5,0.5) {3.5};
\node at (5,0.5) {5};
\node at (6,.5) {$\alpha$};
\node[draw,circle,inner sep=2.5pt,fill=blue,shading=ball] at (1,-1) {};
\node[draw,circle,inner sep=2.5pt,fill=blue,shading=ball] at (1,-5) {};
\node[draw,circle,inner sep=2.5pt,fill=blue,shading=ball] at (3,-1.5) {};
\node[draw,circle,inner sep=2.5pt,fill=blue,shading=ball] at (3.5,-1) {};
\node[draw,circle,inner sep=2.5pt,fill=blue,shading=ball] at (5,-1) {};
\node[draw,circle,inner sep=2.5pt,fill=blue,shading=ball] at (1.5,-3) {};
\node[draw,circle,inner sep=2.5pt,fill=blue,shading=ball] at (1,-3.5) {};
\node[draw,circle,inner sep=2.5pt,fill=blue,shading=ball] at (5,-3.5) {};
\node[draw,circle,inner sep=2.5pt,fill=blue,shading=ball] at (3.5,-5) {};
\node[draw,circle,inner sep=2.5pt,fill=blue,shading=ball] at (5,-0.5) {};
\node[draw,circle,inner sep=2.5pt,fill=blue,shading=ball] at (0.5,-5) {};
\node[draw,circle,inner sep=2.5pt,fill=blue,shading=ball] at (5,-5) {};

\end{tikzpicture}}}
\end{minipage}
\begin{minipage}{0.45\textwidth}
\centering
\subfigure[$\alpha\in[1,3)$]{\fbox{\scalebox{.7}{\begin {tikzpicture}[-latex ,auto ,node distance =1.5 cm and 1.5cm ,on grid ,
semithick ,
state/.style ={ circle ,top color =white , bottom color = blue!20 ,
draw,blue , text=blue , minimum width =1 cm},every loop/.style={}]
\node[state] (t1) {$1$};
\node[] (t3) [left=of t1] {};
\node[] (t15) [above=of t3] {};

\draw[-] (t1) edge [loop above] (t1);

\end{tikzpicture} }}}~~~~~~~~~~~
\subfigure[$\alpha\in[3,3.5)$]{\fbox{\scalebox{.7}{\begin {tikzpicture}[-latex ,auto ,node distance =1.5 cm and 1.5cm ,on grid ,
semithick ,
state/.style ={ circle ,top color =white , bottom color = blue!20 ,
draw,blue , text=blue , minimum width =1 cm},every loop/.style={}]
\node[state] (t1) {$1$};
\node[state] (t3) [left=of t1] {$3$};
\node[state] (t15) [above=of t3] {$1.5$};

\draw[-] (t1) edge [loop above] (t1);
\draw[-] (t3) -- (t15);

\end{tikzpicture}}}}\\
\subfigure[$\alpha\in[3.5,5)$]{\fbox{\scalebox{.7}{\begin {tikzpicture}[-latex ,auto ,node distance =1.5 cm and 1.5cm ,on grid ,
semithick ,
state/.style ={ circle ,top color =white , bottom color = blue!20 ,
draw,blue , text=blue , minimum width =1 cm},every loop/.style={}]
\node[state] (t1) {$1$};
\node[state] (t3) [left=of t1] {$3$};
\node[state] (t15) [above=of t3] {$1.5$};
\node[state] (t35) [below=of t1] {$3.5$};

\draw[-] (t1) edge [loop above] (t1);
\draw[-] (t3) -- (t15);
\draw[-] (t1) -- (t35);

\end{tikzpicture}}}}
\subfigure[$\alpha=5$]{\fbox{\scalebox{.7}{\begin {tikzpicture}[-latex ,auto ,node distance =1.5 cm and 1.5cm ,on grid ,
semithick ,
state/.style ={ circle ,top color =white , bottom color = blue!20 ,
draw,blue , text=blue , minimum width =1 cm},every loop/.style={}]
\node[state] (t1) {$1$};
\node[state] (t3) [left=of t1] {$3$};
\node[state] (t15) [above=of t3] {$1.5$};
\node[state] (t35) [below=of t1] {$3.5$};
\node[state] (t5) [right=of t1] {$5$};
\node[state] (t05) [above=of t5] {$0.5$};

\draw[-] (t1) edge [loop above] (t1);
\draw[-] (t3) -- (t15);
\draw[-] (t1) -- (t35);
\draw[-] (t5) -- (t1);
\draw[-] (t5) -- (t35);
\draw[-] (t5) -- (t05);
\draw[-] (t5) edge [loop right] (t5);

\end{tikzpicture}}}}
\end{minipage}
\caption{Illustration of the connection between the point process on the plane and the graphex process. (a) Point process $\sum_{ij} Z_{ij}\delta_{(\theta_i,\theta_j)}$ on the plane. (b-e) Associated graphs $\mathcal G_\alpha$ for (b) $\alpha\in [1,3)$, (c) $\alpha\in[3,3.5)$, (d) $\alpha\in [3.5,5)$ and (e) $\alpha=5$. Note that the graph is empty for $\alpha<1$.}
\label{fig:pointgraph}
\end{figure}

We now define a number of summary statistics of the graph $\mathcal G_\alpha$.
For $i\geq 1$, let
 $$D_{\alpha,i}=\sum_{k}Z_{ik}\1{\theta_{k}\leq\alpha
}.$$ If $\theta_i\in \mathcal V_\alpha$, then $D_{\alpha,i}\geq 1$ corresponds to the degree of the node $\theta_i$ in the graph $\mathcal G_\alpha$ of size $\alpha$; otherwise $D_{\alpha,i}=0$. Let $N_{\alpha
}=|\mathcal V_\alpha|$ and  $N_{\alpha,j}$ be the number of  nodes and the number of nodes of degree $j,\,j\ge 1$ respectively,
\begin{equation}
N_{\alpha}=\sum_{i}\1{\theta_{i}\leq\alpha}\1{D_{\alpha,i}\geq 1}, \quad N_{\alpha,j}=\sum_{i}\1{\theta_{i}\leq\alpha}\1{D_{\alpha,i}=j} \label{eq:nbnodes}
\end{equation}
and $N^{(e)}_{\alpha}=|\mathcal E_\alpha|$ the number of edges
\begin{equation}
N^{(e)}_{\alpha}=\frac{1}{2}\sum_{i\neq j}Z_{ij} \1{\theta_{i}\leq\alpha}\1{\theta_{j}%
\leq\alpha}+\sum_{i}%
Z_{ii}\1{\theta_{i}\leq\alpha}.\label{eq:nbedges}
\end{equation}
For $i\geq 1$, let
\begin{equation}
T_{\alpha, i}=\frac{1}{2}\sum_{j,k\mid j\neq k\neq i}Z_{ij}Z_{jk}Z_{ik}\1{\theta
_{i}\leq\alpha}\1{\theta_{j}\leq\alpha}\1{\theta_{k}\leq\alpha}.
\end{equation}
If $\theta_i\in \mathcal V_\alpha$, $T_{\alpha, i}$ corresponds to the number of triangles containing node $\theta_i$ in the graph $\mathcal G_\alpha$, otherwise $T_{\alpha, i}=0$. Let
\begin{equation}
T_{\alpha}=\frac{1}{3}\sum_{i}T_{\alpha,i}=\frac{1}{6}\sum_{i\neq j\neq k}Z_{ij}Z_{jk}Z_{ik}\1{\theta
_{i}\leq\alpha}\1{\theta_{j}\leq\alpha}\1{\theta_{k}\leq\alpha}%
\end{equation}
denote the total number of triangles and
\begin{align}
A_{\alpha}   =\sum_{i}\frac{D_{\alpha,i}(D_{\alpha,i}-1)}{2} =\frac{1}{2}\sum_{i\neq j\neq k}Z_{ij}Z_{jk}\1{\theta_{i}\leq\alpha
}\1{\theta_{j}\leq\alpha}\1{\theta_{k}\leq\alpha}
\end{align}
the total number of adjacent edges in the graph $\mathcal G_\alpha$. The global clustering coefficient, also known as the transitivity coefficient, is defined as
\begin{equation}
C_{\alpha}^{(g)}=\frac{3T_{\alpha}}{A_{\alpha}}\label{eq:globalclusteringdef}
\end{equation}
if $A_{\alpha}\geq 1$ and 0 otherwise. The global clustering coefficient counts the proportion of closed connected triplets over all the connected triplets, or equivalently the fraction of pairs of nodes connected to the same node that are themselves connected, and is a standard measure of the transitivity of a network~\citep[Section 7.9]{Newman2010}. Another measure of the transitivity of the graph is the local clustering coefficient. For any degree $j\geq 2$, define
\begin{equation}
C_{\alpha,j}^{(\ell)}=\frac{2}{j(j-1)N_{\alpha,j}}\sum_{i}T_{\alpha,i}\1{D_{\alpha,i}=j}
\end{equation}
if $N_{\alpha,j}\geq 1$ and 0 otherwise. $C_{\alpha,j}^{(\ell)}$ corresponds to the proportion of pairs of neighbours of nodes of degree $j$ that are connected. The average local clustering coefficient is obtained by
\begin{equation}
\overline{C}_{\alpha}^{(\ell)}=\frac{1}{N_{\alpha}-N_{\alpha,1}}\sum
_{j\geq2}N_{\alpha,j}C_{\alpha,j}^{(\ell)}
\end{equation}
if $N_{\alpha}-N_{\alpha,1}\geq 1$ and $\overline{C}_{\alpha}^{(\ell)}=0$ otherwise.

\subsection{Assumptions}
\label{sec:assumptions}

 We will make use of the following three assumptions. Assumption \ref{assumpt:1} characterises the behaviour of the small degree nodes. Assumption~\ref{assumpt:2} is a technical assumption to obtain the almost sure results. Assumption \ref{assumpt:3} characterises the behaviour of large degree nodes.\smallskip

A central quantity of interest in the analysis of the asymptotic properties of graphex processes is the marginal generalised graphon function $\mu:(0,\infty)\to\mathbb R_+$, defined for $x>0$ by
\begin{equation}
\mu(x)=\int_0^\infty W(x,y)dy\label{eq:marginalgraphon}
.
\end{equation}
The integrability of the generalised graphon $W$ implies that $\mu$ is integrable. Ignoring loops (self-edges), the expected number of connections of a node with parameter $\vartheta$ is proportional to $\mu(\vartheta)$. Therefore, assuming $\mu$ is monotone decreasing, its behaviour at infinity controls the small degree nodes, while its behaviour at 0 controls the large degree nodes.

For mathematical convenience, it will be easier to work with the generalised inverse $\mu^{-1}$ of $\mu$. The behaviour at 0 of $\mu^{-1}$ then controls the small degree nodes, while the behaviour of $\mu^{-1}$ at infinity controls large degree nodes.

The following assumption characterises the behaviour of $\mu$ at infinity or, equivalently, of $\mu^{-1}$ at 0. We require $\mu^{-1}$ to behave approximately as a power function $x^{-\sigma}$ around 0, for some $\sigma\in[0,1]$. This behaviour, known as regular variation, has been extensively studied  (see, e.g., \cite{Bingham1987}) and we provide some background on it in Appendix~\ref{sec:regularvariation}.

\begin{assumption}\label{assumpt:1}
Assume $\mu$ is non-increasing,  with
generalised inverse $\mu^{-1}(x)=\inf\{ y>0 \mid \mu(y)\leq  x\}$, such that
\begin{equation}
\mu^{-1}(x)\sim\ell(1/x)x^{-\sigma}\text{ as }x\rightarrow0
\end{equation}
where $\sigma\in\lbrack0,1]$ and $\ell$ is a slowly varying function at
infinity: for all $c>0$, $\lim_{t\rightarrow\infty}\ell(ct)/\ell(t)=1.$
\end{assumption}
Examples of slowly varying functions $\ell$ include functions converging to a strictly positive constant, or powers of logarithms. Note that Assumption 1 implies that, for $\sigma\in(0,1)$, $\mu(t)\sim \overline \ell(t)t^{-1/\sigma}\text{ as }t\rightarrow\infty$ for some slowly varying function $\overline \ell$.  We can differentiate four cases, as it will be formally derived in Corollary \ref{th:sparsity1}.
\begin{itemize}
\item[(i)]  Dense case: $\sigma=0$ and $\lim_{t\rightarrow\infty}\ell(t)<\infty$. In this case, $\lim_{x\rightarrow 0}\mu^{-1}(x)<\infty$, hence $\mu$ has bounded support. The other three cases are all sparse cases.
\item[(ii)] Almost dense case:  $\sigma=0$ and $\lim_{t\rightarrow\infty}\ell(t)=\infty$. In this case $\mu$ has full support and super-polynomially decaying tails.
\item[(iii)] Sparse case with  power law:  $\sigma\in(0,1)$. In this case $\mu$ has full support and polynomially decaying tails (up to a slowly varying function).
\item[(iv)] Very sparse case:  $\sigma=1$. In this case $\mu$ has full support and very light tails. In order for $\mu^{-1}$ (and hence $W$) to be integrable, we need $\ell$ to go to zero sufficiently fast.
\end{itemize}

Now define, for $x,y>0$
\begin{equation}\label{eq:nu}
\nu(x,y)=\int_0^\infty W(x,z)W(y,z)dz.
\end{equation}
The expected number of common neighbours of nodes with parameters $(\vartheta_1,\vartheta_2)$ is proportional to $\nu(\vartheta_1,\vartheta_2)$.

The following assumption is a technical assumption needed in order to obtain the almost sure results on the number of nodes and degrees. \cite{Veitch2015} made a similar assumption to obtain results in probability, see the discussion section for further details.
\begin{assumption}\label{assumpt:2}
Assume that there exists $C_1,a >0$ and $x_0\geq 0$ such that for all $x,y>x_0$
\begin{equation}\label{eq:conda}
 \nu(x,y)\leq C_1 \mu(x)^a\mu(y)^a, \quad \mu(x_0)>0, \quad \left \{
\begin{array}{ll}
  a>\max\left(\frac{1}{2},\sigma\right) & \text{if }\sigma\in[0,1) \\
  a=1 & \text{if }\sigma=1.
\end{array}\right .
\end{equation}
\end{assumption}

\begin{remark}
Assumption \ref{assumpt:2} is trivially satisfied when the function $W$ is separable
$
W(x,y)= \mu(x)\mu(y)/\overline W.
$
Assumptions \ref{assumpt:1} and \ref{assumpt:2} are also satisfied if
\begin{equation}
\quad W(x,y)= 1-e^{-f(x)f(y)/\overline f}
\label{eq:bernoullipoissonlink}
\end{equation}
for some positive, non-increasing, measurable function $f$ with $\overline f=\int_0^\infty f(x)dx<\infty$ and generalised inverse $f^{-1}$ verifying
$
f^{-1}(x)\sim\ell(1/x)x^{-\sigma}\text{ as }x$ tends to 0. In this case, $\mu$ is monotone non-increasing. We have
\begin{align*}
\mu\{f^{-1}(x)\}&=\int_0^\infty \{1-e^{-x f(y)/\bar f}\}dy=x\int_0^\infty  e^{-xu/\bar f} f^{-1}(u)/\bar fdu
\sim x
\end{align*}
as $x$ tends to $0$ by dominated convergence.  \sloppy Hence $f\{\mu^{-1}(x)\}\sim x$ as $x$ tends to 0 and $f^{-1}[f\{\mu^{-1}(x)\}]\sim \ell(1/x)x^{-\sigma}$. Assumption \ref{assumpt:2} follows from the inequality $W(x,y)\leq f(x)f(y)/\overline f$. Other examples are considered in Section \ref{sec:examples}.
\end{remark}

\noindent The following assumption is used to characterise both the asymptotic behaviour of small and large degree nodes.
\begin{assumption}\label{assumpt:3}
Assume $\mu^{-1}(t)=\int_t^\infty f(x)dx$ where $f$ is continuous on $(0, \infty)$ and
\begin{align*}
(a) : \quad f(x) &\sim \tau x^{-\tau-1}\ell_2(x)\text{ as }x\to\infty\\
(b) : \quad  f(x) &\sim x^{-\tilde{\sigma}-1}\tilde{\ell}_2(1/x)\text{ as }x\to 0
\end{align*}
where $\tau>0,\tilde{\sigma} \leq 1$ and $\ell_{2}, \tilde{\ell}_2$  are slowly varying functions.
\end{assumption}
\sloppy Note that Assumption \ref{assumpt:3} implies that $\mu^{-1}(x)\sim x^{-\tau}\ell_2(x)\text{ as }x\to\infty$, and $\mu(t)\sim \overline \ell_2(t)t^{-1/\tau}\text{ as }t\rightarrow 0$ for some slowly varying function $\overline \ell_2$. Assumption \ref{assumpt:3} also implies  Assumption \ref{assumpt:1} with $\sigma = \max( \tilde \sigma , 0)$, $\ell(x)=\frac{1}{\sigma} \tilde{\ell}_2(x)$ if $\tilde \sigma\neq 0$, and $\ell(x)=o(\tilde{\ell}_2(x))$ if $\tilde \sigma= 0$.
\medskip

Finally, we state an assumption on $\nu(x, y)$, the quantity proportional to the expected number of common neighbours of two nodes with parameters $x$ and $y$, defined in Equation \eqref{eq:nu}. This technical assumption is used to prove a result on the asymptotic behaviour of the variance of the number of nodes (Proposition \ref{prop:variancenbnodes}) and the central limit theorem for sparse graphs enunciated in Section \ref{subsec:CLT_sparse}.
\begin{assumption}\label{assumpt:4}
Assume that there exists $0<C_{0}\leq C_1$ and
$x_{0}\geq0$ such that for all $x,y>x_{0}$%
\[
C_{0}\mu(x)\mu(y)\leq\nu(x,y)\leq C_1\mu(x)\mu(y).
\]
\end{assumption}
Assumption \ref{assumpt:4} holds when $W$ is separable, as well as in the model of \cite{Caron2017} under some moment conditions (see Section \ref{ex:caronfox}). Obviously, Assumption \ref{assumpt:4} implies that Assumption \ref{assumpt:2} is satisfied with $a=1$.

\section{Asymptotic behaviour of various statistics of the graph}\label{sec:asympstats}

\subsection{Asymptotic behaviour of the number of edges, number of nodes and degree distribution} \label{sec:theorems}
In this section we characterise the almost sure and expected behaviour of the number of nodes $N_\alpha$, number of edges $\Nedges$ and number of nodes with $j$ edges $N_{\alpha,j}$. These results allow us to provide precise statements about the sparsity of the graph and the asymptotic power-law properties of its degree distribution. 

We first recall existing results on the asymptotic growth of the number of edges. The growth of the mean number of edges has been shown by \cite{Veitch2015} and the almost sure convergence follows from~\citep[Proposition 56]{Borgs2018}.
\begin{proposition}
[Number of edges \citep{Veitch2015,Borgs2018}]\label{th:numberedges}
As $\alpha $ goes to infinity, almost surely
\begin{equation}
\Nedges \sim E(\Nedges)\sim \alpha^{2}\overline W/2.
\end{equation}
\end{proposition}
The following two theorems provide a description of the asymptotic behaviour of the terms $N_\alpha, N_{\alpha,j}$ in expectation and almost surely.

\begin{theorem}
\label{th:meannumbernodes} For $\sigma\in[0,1]$, let $\ell_\sigma$ be slowly varying functions defined as
\begin{align}\label{eq:defellsigma}
\ell_1(t)=\int_t^\infty y^{-1} \ell(y)dy~~~\text{ and }~~~
\ell_\sigma(t) = \ell(t) \Gamma(1-\sigma)\text{ for }\sigma \in[0,1).
\end{align}
Under Assumption~\ref{assumpt:1}, for all $\sigma \in [0,1]$,
\begin{equation}
  E(N_{\alpha})\sim\alpha^{1+\sigma}\ell_\sigma(\alpha).
\end{equation}
If $\sigma = 0$  then for $j\geq 1$  $$ E(N_{\alpha,j})=o\{\alpha\ell(\alpha)\}.$$ If $\sigma\in(0,1)$ then for $j\geq 1$
$$
    E(N_{\alpha,j})\sim\frac{\sigma\Gamma(j-\sigma)}{j!}\alpha^{1+\sigma
}\ell(\alpha) $$
Finally, if $\sigma= 1$,
$$
E(N_{\alpha,j})\sim \left \{
\begin{array}{ll}
  \alpha^{2}\ell_1(\alpha) & j=1 \\
  \frac{\alpha^{2}}{j(j-1)}\ell(\alpha) & j\geq 2
\end{array}\right .
$$
\end{theorem}

Theorem \ref{th:meannumbernodes} follows rather directly from asymptotic properties of regularly varying functions~\citep{Gnedin2007}, recalled in Lemma \ref{lemma:abelianvariations}  and \ref{lemma:abelianvariationsv2}  in the Appendix. Details of the proof are given in Appendix~\ref{sec:proofmeannodesedges}.  Note that $\ell(\alpha)=o(\ell_1(\alpha))$; hence, for $\sigma=1$, $E(N_{\alpha,j})=o\{E(N_{\alpha,1})\}$ for all $j\geq 2$. \medskip

\cite{Veitch2015} have shown that, under Assumption 2 with $a=1$, we have, in probability,
\[
  N_{\alpha}\sim E(N_{\alpha}),\quad \sum_{k\geq j} N_{\alpha,k}\sim E\left (\sum_{k\geq j} N_{\alpha,k}\right)~~\text{for }j\geq 1.
\]
The next theorem shows that the asymptotic equivalence holds almost surely under Assumptions \ref{assumpt:1} and  \ref{assumpt:2}. Additionally, combining these results with Theorem~\ref{th:meannumbernodes} allows us to characterise the almost sure asymptotic behaviour of the number of nodes and number of nodes of a given degree. The proof of Theorem \ref{th:asnumbernodes}  is given in Section \ref{sec:app:asnumbernodes}.

\begin{theorem}
\label{th:asnumbernodes} Under Assumptions~\ref{assumpt:1} and~\ref{assumpt:2}, we have almost surely as $\alpha$ tends to infinity
\begin{equation}
  N_{\alpha}\sim E(N_{\alpha}),\quad \sum_{k\geq j} N_{\alpha,k}\sim E\left (\sum_{k\geq j} N_{\alpha,k}\right)~~\text{for }j\geq 1.
  \label{eq:nodesexpectationalmostsure}
\end{equation}
Combining this with Theorem \ref{th:meannumbernodes}, we obtain that, for all $\sigma\in[0,1]$,
\[
  N_{\alpha}\sim\alpha^{1+\sigma}\ell_\sigma(\alpha).
\]
Moreover, for $j\geq 1$, if $\sigma = 0$ then    $N_{\alpha,j}=o\{\alpha\ell(\alpha)\} $, while if $0< \sigma < 1$
$$
    N_{\alpha,j}\sim\frac{\sigma\Gamma(j-\sigma)}{j!}\alpha^{1+\sigma
}\ell(\alpha). $$
If $\sigma = 1$, $N_{\alpha,1}\sim \alpha^{2}\ell_1(\alpha)$ and for all $j \geq 2$  we also have,
    $N_{\alpha,j}=o\{\alpha^{2}\ell_1(\alpha)\}.$
\end{theorem}

The following result is a corollary of Theorem~\ref{th:asnumbernodes} which shows how the parameter $\sigma$ relates to the sparsity and power-law properties of the graphs.
We denote $\ell^\#$ the de Bruijn conjugate (see definition \ref{def:debruijn} in the Appendix) of the slowly varying function $\ell$.

\begin{corollary}[Sparsity and power-law degree distribution]
\label{th:sparsity1}Assume Assumptions~\ref{assumpt:1} and~\ref{assumpt:2}.
For $\sigma\in[0,1]$, almost surely as  $\alpha$ tends to infinity,
\[
\Nedges\sim \frac{\overline W }{2} N_\alpha^{2/(1+\sigma)}\ell_\sigma^*(N_\alpha)
, \quad \ell_\sigma^*(y)=\left [\left \{\ell_\sigma^{1/(1+\sigma)}(y^{1/1+\sigma})\right \}^\#\right ]^2.
\]
$\ell_\sigma^*(y)$ is slow varying and the graph is dense if $\sigma=0$  and $\lim_{t\rightarrow\infty} \ell(t)=C<\infty$, as
$\Nedges/N_{\alpha}^2\rightarrow C^2\overline W/2$ almost surely. Otherwise, if $\sigma>0$ or $\sigma=0$ and $\lim_t \ell(t)=\infty$, the graph is sparse, as
$\Nedges/N_{\alpha}^2\rightarrow 0$. Additionally, for $\sigma\in [0,1)$, for any $j=1,2,\ldots$,
\begin{align}
\frac{N_{\alpha,j}}{N_\alpha}\rightarrow \frac{\sigma \Gamma(j-\sigma)}{j!\Gamma(1-\sigma)}\label{eq:powerlawdegree}
\end{align}
almost surely. If $\sigma>0$, this corresponds to a degree distribution with a power-law behaviour as, for $j$ large
$$
\frac{\sigma \Gamma(j-\sigma)}{j!\Gamma(1-\sigma)}\sim \frac{\sigma}{\Gamma(1-\sigma)j^{1+\sigma}}.
$$
For $\sigma=1$,
$N_{\alpha,1}/N_\alpha\rightarrow 1$ and $N_{\alpha,j}/N_\alpha\rightarrow 0\text{ for $j\geq 2$}
$, hence the nodes of degree 1 dominate in the graph.
\end{corollary}

\begin{remark}
If $\sigma=0$ and $\lim_{t\rightarrow\infty} \ell(t)=\infty$, the graph is almost dense, that is
$
\Nedges/N_\alpha^{2}\rightarrow 0\text{ and }\Nedges/N_\alpha^{2-\epsilon}\rightarrow \infty
$
for any $\epsilon>0$. If $\sigma=1$, the graph is almost extremely sparse~\citep{BollobasRiordan2009}, as
$
\Nedges/N_\alpha\rightarrow \infty\text{ and }\Nedges/N_\alpha^{1+\epsilon}\rightarrow 0
$
for any $\epsilon>0$.
\end{remark}

The above results are important in terms of modelling aspects, since they allow a precise description of the degrees and number of edges as a function of the number of nodes. They can also be used to conduct inference on the parameters of the statistical network model, since the behaviour of most estimators will depend heavily on the behaviour of $N_\alpha, N_{\alpha}^{(e)} $ and possibly $N_{\alpha,j}$. For instance the following naive estimator\footnote{Following an earlier version of the present paper, \cite{Naulet2017} proposed an alternative estimator for $\sigma$, with better statistical properties.} of $\sigma$
\begin{equation}
\hat \sigma = \frac{ 2 \log N_\alpha}{ \log N_\alpha^{(e)} } - 1\label{eq:estimatorsigma}
\end{equation}
is almost surely consistent. Indeed under Assumptions \ref{assumpt:1} and \ref{assumpt:2},
using Theorems \ref{th:numberedges} and \ref{th:asnumbernodes}, we have almost surely $N_\alpha^2\sim\alpha^{2+2\sigma}\ell_\sigma(\alpha)^2$ and $N_\alpha^{(e)}\sim \alpha^2\overline W/2$. Hence
$$
\log\frac{N_\alpha^2}{N_\alpha^{(e)}}\sim 2\sigma\log(\alpha)+\log\{\ell_\sigma(\alpha)^2 2/\overline W\}
$$
and the result follows as $\log\ell_\sigma(\alpha)/\log\alpha\rightarrow 0$.\medskip

All the above results concern the behaviour of small degree nodes, where the degree $j$ is fixed as the size of the graph goes to infinity. It is also of interest to look at the number of nodes of degree $j$ as both $\alpha$ and  $j$ tend to $\infty$. We show in the next proposition that this is controlled by the behaviour of the function $f$, introduced in Assumption~\ref{assumpt:3}, at 0 or $\infty$.

\begin{proposition}[Power-law for high degree nodes]\label{prop:largedegrees}
Assume that Assumption \ref{assumpt:3}  holds. Then when $j \rightarrow \infty $ and $\log \alpha = o(j)$ and $j/\alpha \rightarrow c_0 \in [0, \infty]$, then
\begin{equation*}
E(N_{\alpha,j}) \sim f(j/\alpha).
 \end{equation*}

\end{proposition}

Note that Proposition \ref{prop:largedegrees} implies that  when $j/\alpha\to\infty$,
$$ E(N_{\alpha,j})\sim\frac{\tau\alpha^{1+\tau}\ell_{2}(j/\alpha)}{j^{1+\tau}}$$
which  corresponds to a power-law behaviour with exponent $1+\tau$ . If $j/\alpha\to 0$ then
$$ E(N_{\alpha,j})\sim\frac{\alpha^{1+\tilde{\sigma}}\tilde{\ell}_{2}(\alpha/j)}{j^{1+\tilde{\sigma}}}.
$$
This is similar to the asymptotic results for $j$ fixed, stated in Theorem \ref{th:meannumbernodes}, noting that $\frac{\Gamma(j-\sigma)}{j!}\sim j^{-1-\sigma}$ as $j\to\infty$. Finally, if  $j/\alpha \rightarrow c_0 \in (0,\infty)$, then $ E(N_{\alpha,j}) \sim f(c_0) \in (0,\infty)$.

\begin{proof}
Under Assumption \ref{assumpt:3}, we have $\mu^{-1}(t)=\int_t^\infty f(x)dx$ with
\[
f(x)\sim \tau x^{-\tau-1}\ell_2(x)\text{ as }x\to\infty, \quad f(x)\sim \tau x^{-\sigma-1}\tilde{\ell}_2(x)\text{ as }x\to 0
\]
From \cite[Theorem 5.5]{Veitch2015} we have, assuming that $W(x,x)=0$ for the sake of simplicity,
\begin{align*}
E(N_{\alpha,j})&=\alpha\int_0^\infty e^{-\alpha\mu
(\vartheta)}\frac{(\alpha\mu(\vartheta))^{j}}{j!}d\vartheta  \\
&=\alpha\int_0^\infty e^{-\alpha x}\frac{(\alpha x)^{j}}{j!}f(x)dx  \\
 &= E[ f((j+1) X_j /\alpha) ],
\end{align*}
where $X_j$ is a gamma random variable with rate $j+1$ and inverse scale $j+1$.
We split the above expectation into $X_j < 1/2, X_j \in [1/2, 3/2], X_j>3/2 $. The idea is that the third and the first expectations are small because $X_j$ concentrates fast to 1, while the middle expectation ($X_j \in [1/2, 3/2]$ ) uses the fact that $f((j+1)X_j/\alpha) \approx f((j+1)/\alpha) $. More precisely, using Stirling's approximation, for all $\epsilon>0$, there exists $c>0$
\begin{align*}
E[ f((j+1) X_j /\alpha) \1{X_j<1/2}] & = \frac{ (j+1)^{j+1} }{\Gamma(j+1)}  \int_0^{1/2} f((j+1)x/\alpha) x^j e^{-(j+1)x}dx \\
&\lesssim \sqrt{j} \int_0^{1/2}\left( 1 + \left(\frac{ (j+1)x}{\alpha}\right)^{-1-\tilde{\sigma} -\epsilon } \right) e^{-j(x-\log x -1) } dx  \\
& \lesssim e^{- c j}  \left( 1 + \left(\frac{ j }{ \alpha }\right)^{-1-\tilde{\sigma } -\epsilon } \right) = o(1)
\end{align*}
since $\alpha/j = o(e^{cj }) $ for any $c>0$.
The expectation over $X_j >3/2$ is treated similarly. We now study the expectation over $[1/2, 3/2]$.
We have that if $j/\alpha \rightarrow \infty $, then uniformly in $x \in [1/2, 3/2]$, under Assumption \ref{assumpt:3},
$$ \left | \frac{ f( (j+1) x/\alpha)  }{  f( (j+1) /\alpha) }  - x^{-1-\tau} \right| = o(1) $$
and similarly when   $j/\alpha \rightarrow 0$, with $\tau $ replaced by $\tilde \sigma$; if  $j/\alpha \rightarrow c_0 \in (0,\infty)$, then uniformly in $x \in [1/2, 3/2]$,  $$ \left | \frac{ f( (j+1) x/\alpha)  }{  f( (j+1) /\alpha) }  - \frac{f(c_0 x)}{f(c_0)} \right| = o(1). $$
Moreover since $X_j $ converges almost surely to 1, we finally obtain that
$$E\left[ \frac{ f((j+1) X_j /\alpha) }{ f((j+1)  /\alpha) } \1{X_j\in [1/2,3/2]} \right] \to 1 $$
which terminates the proof.
\end{proof}

\subsection{Proof of Theorem \ref{th:asnumbernodes}}

\label{sec:app:asnumbernodes}

The proof follows similarly to that of \cite[Theorem 6.1]{Veitch2015}, by bounding the variance. \cite{Veitch2015} showed that $\var(N_{\alpha})=o(E(N_{\alpha})^2)$ and $\var(N_{\alpha,j})=o(E(N_{\alpha,j})^2)$ and use this result to prove that \eqref{eq:nodesexpectationalmostsure} holds in probability; we need a slightly tighter bound on the variances to obtain the almost sure convergence. This is stated in the next two Propositions.
\begin{proposition}
\label{prop:variancenbnodes}
Let $N_\alpha$ be the number of nodes. We have
\begin{align}
\var(N_{\alpha})  &  =E(N_{\alpha})+2\alpha^{2}\int_{\mathbb{R}_{+}
}\mu(x)\{1-W(x,x)\}e^{-\alpha\mu(x)}dx\nonumber\\
&  +\alpha^{2}\int_{\mathbb{R}_{+}^{2}}\{1-W(x,y)\}\{1-W(x,x)\}\{1-W(y,y)\}\nonumber\\
&~~~~~~~~~~\left\{ e^{\alpha \nu(x,y)} -1 +W(x,y) \right\} e^{-\alpha\mu(x)-\alpha\mu(y)}dxdy.\label{eq:variancenbnodes1}
\end{align}
Under Assumptions \ref{assumpt:1} and \ref{assumpt:2}, with $\sigma\in[0,1]$, slowly varying function $\ell$ and positive scalar $a$ satisfying \eqref{eq:conda},  we have
\begin{equation}
\var(N_{\alpha})=O\{\alpha^{3+2\sigma-2a}\ell_\sigma(\alpha)^2\}.\label{eq:variancenbnodes2}
\end{equation}
where the slowly varying functions $\ell_\sigma$ are defined in Equation~\eqref{eq:defellsigma}. Additionally, under Assumptions \ref{assumpt:1} and \ref{assumpt:4}, we have, for any $\sigma\in[0,1]$ and any slowly varying function $\ell$
\begin{equation}
\var(N_{\alpha})\asymp\alpha^{1+2\sigma}\ell_\sigma^{2}(\alpha).\label{eq:variancenbnodes3}
\end{equation}
\end{proposition}
\begin{sketchproof}
We give here the ideas behind the proof, deferring its completion to Section~\ref{sec:proofvarnnodes} of the Supplementary Material~\citep{CaronRousseau2017:supplement}.  Equation \eqref{eq:variancenbnodes1} is immediately obtained using the Slivnyak-Mecke and Campbell theorems. Applying the inequality $e^{x}-1\le xe^x$ and the Lemmas \ref{lemma:abelianvariations} and~\ref{lemma:boundintnu} to the right-hand side of Equation \eqref{eq:variancenbnodes1},  the upper bound of equation \eqref{eq:variancenbnodes2} follows.  Finally, if Assumption \ref{assumpt:4} hold, then Assumption \ref{assumpt:2} holds as well with $a=1$. Together with Assumption \ref{assumpt:1} we can therefore specialise the upper bound of Equation \eqref{eq:variancenbnodes2} to the case $a=1: O(\alpha^{1+2\sigma}\ell^2_\sigma(\alpha))$.  The lower bound with the same order is found using the inequality $e^x-1\ge x$ and Lemmas~\ref{lemma:abelianvariations} and \ref{lemma:abelianvariationsv2}.
 \end{sketchproof}
 \medskip

Proposition \ref{prop:variancenbnodes} and Theorem~\ref{th:meannumbernodes}  imply in particular that, under Assumptions~\ref{assumpt:1} and~\ref{assumpt:2},
$$
\var(N_{\alpha})=O\{E(N_\alpha)^2 \alpha^{-\kappa}\}
$$
for some $\kappa>0$.  $N_\alpha$ is a positive, monotone increasing stochastic process. Using Lemma~\ref{lemma:almostsure} in the Appendix, we obtain that $N_\alpha\sim E(N_\alpha)$ almost surely as $\alpha$ tends to $\infty$.

\begin{proposition}
\label{prop:variancenbnodesj}
Let $N_{\alpha,j}$ be the number of nodes of degree $j$. Then, under  Assumptions \ref{assumpt:1} and \ref{assumpt:2}, with $\sigma\in[0,1]$, slowly varying function $\ell$ and positive scalar $a$ satisfying \eqref{eq:conda},  we have
$$
\var(N_{\alpha,j})=O\{\alpha^{3+2\sigma-2a}\ell_\sigma(\alpha)^2\}.
$$
where the slowly varying functions $\ell_\sigma$ are defined in Equation~\eqref{eq:defellsigma}. In the case $\sigma=0$ and $a=1$, we have the stronger result
$$
\var(N_{\alpha,j})=o\{\alpha\ell(\alpha)^2\}.
$$
\end{proposition}
\begin{sketchproof} While the complete proof of Proposition \ref{prop:variancenbnodesj} is given in Section~\ref{sec:proofvarnnodesj} in the Supplementary Material~\citep{CaronRousseau2017:supplement}, we explain here its main passages.  We start by evaluating the expectation of $N_{\alpha, j}^2$ and $N_{\alpha, j}$ conditional on the unit-rate Poisson random measure $M=\sum_{i}\delta_{(\theta_{i},\vartheta_{i})}$:
\begin{align*}
& E(N_{\alpha,j}^{2}\mid M)-E(N_{\alpha,j}\mid M)\nonumber\\
&=  \sum_{i_{1}\neq i_{2}}\1{\theta_{i_{1}}\leq\alpha}\1{\theta_{i_{2}}\leq\alpha}\Pr\left\{  \sum
_{k}\1{\theta_{k}\leq\alpha}Z_{i_{1}k}=j\text{ and } \sum_{k}\1{\theta_{k}\leq\alpha}Z_{i_{2}%
,k}=j\mid M\right\}\label{eq:bigexpectation}\\
&=\sum_{b\in\{0,1\}^3} \sum_{j_1=0}^j \sum_{i_{1}\neq i_{2}}\1{\theta_{i_{1}}\leq\alpha}\1{\theta_{i_{2}}\leq\alpha}\\
&\quad\times\Pr\left\{  \sum
_{k}\1{\theta_{k}\leq\alpha}Z_{i_{1}k}=j\text{ and } \sum_{k}\1{\theta_{k}\leq\alpha}Z_{i_{2}%
,k}=j\text{ and } \sum
_{k}\1{\theta_{k}\leq\alpha}Z_{i_{1}k}Z_{i_{2}k}=j-j_1\right .\\
&\quad\quad\quad\quad \left .\text{ and }Z_{i_1i_1}=b_{11},Z_{i_1i_2}=b_{12},Z_{i_2i_2}=b_{22} \mid M\right \}
\end{align*}
where $b=(b_{11},b_{12},b_{22})\in\{0,1\}^3$.  We then use the Slivnyak-Mecke theorem to obtain $E(N_{\alpha,j}^{2})-E(N_{\alpha,j})$, which can be bounded by a sum of terms of the form
\begin{equation}
\alpha^2\int_{\R^2} [\alpha\mu(x)]^{k_1}[\alpha\mu(y)]^{k_2} (\alpha \nu(x,y))^{r} e^{ - \alpha \mu(x) -\alpha \mu(y) +\alpha \nu(x,y)}dxdy\label{eq:bound_Naj}
\end{equation}
for $k_1,k_2,r\in\{0,\ldots,j\}$. 
For terms with $r\geq1$, we use Lemma \ref{lemma:Ir} (enunciated and proved, using Lemmas \ref{lemma:abelianvariations} and  \ref{lemma:abelianvariationsv3}, in Section~\ref{sec:proofvarnnodesj} of the Supplementary Material). The Lemma states that, under Assumptions \ref{assumpt:1} and \ref{assumpt:2}, the integral in \eqref{eq:bound_Naj} is in $O( \alpha^{r-2ar+2\sigma } \ell_\sigma^2(\alpha)))=O( \alpha^{1-2a+2\sigma } \ell_\sigma^2(\alpha)))$ for any $r\ge 1, k_1,k_2\ge 0$.
For terms with $r=0$ in \eqref{eq:bound_Naj}, we use the inequality $e^{x}\leq 1+xe^x$, Cauchy-Schwarz inequality and Lemma \ref{lemma:abelianvariationsv3} to show that these terms are in $O\{\alpha^{3+2\sigma-2a}\ell_\sigma^2(\alpha)\}$, which completes the proof.
\end{sketchproof}
\\

Define $\widetilde{N}_{\alpha,j}    =\sum_{k\geq j}N_{\alpha,k}$,
the number of nodes of degree at least $j$. Note that $\widetilde{N}%
_{\alpha,j}$ is a positive, monotone increasing stochastic process in $\alpha$, with
$ \widetilde{N}_{\alpha,j}  = N_\alpha  - \sum_{k = 1}^{j-1} N_{\alpha,k}$.
We then have that, using Cauchy-Schwarz and Jensen's inequalities
$$ E(\widetilde{N}_{\alpha,j}) = E(N_\alpha) - \sum_{k = 1}^{j-1} E( N_{\alpha,k}), \quad  \var(\widetilde{N}_{\alpha,j}) \leq j \left\{ \var(N_\alpha) + \sum_{k = 1}^{j-1} \var( N_{\alpha,k})\right\}.$$
Consider first the case $\sigma\in[0,1)$. Since Theorem~\ref{th:meannumbernodes} implies, for $j\geq 2$,
$
\alpha^{1+\sigma}\ell(\alpha) \lesssim E(\widetilde{N}_{\alpha,j})
$
as $\alpha$ goes to infinity, using Propositions \ref{prop:variancenbnodes} and \ref{prop:variancenbnodesj}, we obtain
$
\var(\widetilde N_{\alpha,j})  =O\{\alpha^{-\tau} E(\widetilde N_{\alpha,j})^2  \}
$ for some $\tau>0$. Combined with  Lemma~\ref{lemma:almostsure}, it leads to $\widetilde N_{\alpha,j}\sim E(\widetilde N_{\alpha,j})$
almost surely as $\alpha$ goes to infinity.

The almost sure results for $N_{\alpha,j}$ then follow from the fact that, for all $j\geq 2$, $E(\widetilde N_{\alpha,j})\asymp E(N_\alpha)$ if $\sigma\in(0,1)$, $E(\widetilde N_{\alpha,j})\sim E(N_\alpha)$ if $\sigma=0$ and  $E(\widetilde N_{\alpha,j})=o\{E(N_\alpha)\}$ if $\sigma=1$.

\subsection{Asymptotic behaviour of the clustering coefficients}

The following Proposition is a direct corollary of \citep[Proposition 56]{Borgs2018} who showed the almost sure convergence of subgraph counts in graphex processes.

\begin{proposition}[Global clustering coefficient \citep{Borgs2018}]\label{prop:globalclustering}
Assume $\int_{0}^{\infty}\mu(x)^{2}dx<\infty$. Recall that $T_\alpha$ and $A_\alpha$ are respectively the number of triangles and number of adjacent edges in the graph of size $\alpha$.  We have%
\begin{align*}
T_{\alpha}  &  \sim E(T_{\alpha})=\frac{\alpha^{3}}{6}\int_{\mathbb{R}_{+}%
^{3}}W(x,y)W(x,z)W(y,z)dxdydz,\\
A_{\alpha}  &  \sim E(A_{\alpha})=\frac{\alpha^{3}}{2}\int_{0}^{\infty}%
\mu(x)^{2}dx
\end{align*}
almost surely as $\alpha\rightarrow\infty$. Therefore, if $\int_{0}^{\infty
}\mu(x)^{2}dx>0$, the global clustering coefficient defined in Equation \eqref{eq:globalclusteringdef} converges to a constant%
\[
C_{\alpha}^{(g)}\rightarrow\frac{\int_{\mathbb{R}_{+}^{3}}%
W(x,y)W(x,z)W(y,z)dxdydz}{\int_{0}^{\infty}\mu(x)^{2}dx}\text{ almost surely
as }\alpha\rightarrow\infty.
\]
\end{proposition}

Note that if $\mu$ is monotone decreasing, as $\overline{W}<\infty$, we necessarily have $\int_{a}^{\infty}\mu(x)^{2}dx<\infty$ for
any $a>0$. Hence the condition $\int_{0}^{\infty}\mu(x)^{2}dx<\infty$ in Proposition \ref{prop:globalclustering} requires
additional assumptions on the behaviour of $\mu$ at 0 (or equivalently the
behaviour of $\mu^{-1}$ at $\infty$), which drives the behaviour of large degree nodes. If the graph is dense, $\mu$ is bounded and thus $\int_{0}^{\infty}\mu(x)^{2}dx<\infty$.

\begin{proposition}[Local clustering coefficient]\label{prop:localclustering}
Assume Assumptions \ref{assumpt:1} and \ref{assumpt:2} hold with $\sigma\in(0,1)$. Assume additionally that%
\begin{equation} \label{cond:localcluster}
\lim_{x\rightarrow\infty}\frac{\int_{\mathbb{R}_{+}^{2}}%
W(x,y)W(x,z)W(y,z)dydz}{\mu(x)^{2}}\rightarrow b
\end{equation}
for some $b\in[0,1]$. Then the local clustering coefficients converge in probabiltiy  as $\alpha\rightarrow\infty$: 
\begin{align*}
C_{\alpha,j}^{(\ell)}  &  \rightarrow b \quad \forall j\geq2.
\end{align*}
If $b>0$, the above result holds almost surely, and the average local clustering coefficient satisfies
\begin{align*}
\lim_{\alpha \rightarrow \infty} \overline{C}_{\alpha}^{(\ell)}  &  \rightarrow b, \quad \text{almost surely} .
\end{align*}

\end{proposition}

In general,
\[
\lim_{x\rightarrow\infty}\frac{1}{\mu(x)^{2}}\int W(x,y)W(x,z)W(y,z)dydz\neq
\frac{\int W(x,y)W(x,z)W(y,z)dxdydz}{\int\mu(x)^{2}dx}%
\]
and the global clustering and local clustering coefficients converge to different limits. A
notable exception is the separable case where $W(x,y)=\mu(x)\mu(y)/\overline{W}$, since in this case
$$ \int W(x,y)W(x,z)W(y,z)dydz = \overline{W}^{-3}\mu(x)^2\left(\int \mu(y)^2dy\right)^2 , \quad b = \frac{ \left(\int \mu(y)^2dy\right)^2 }{\overline{W}^{3} } $$
and
$$ \int W(x,y)W(x,z)W(y,z)dydzdx = \overline{W}^{-3}\left(\int \mu(y)^2dy\right)^3 .$$

\textit{Sketch of the proof.} Full details are given in Appendix \ref{sec:proofclustering}, and we only give here a sketch of the proof, which is similar to that of Theorem~\ref{th:asnumbernodes}.
We have
$$C_{\alpha,j}^{(\ell)}  = \frac{ 2 R_{\alpha,j} }{ j(j-1) N_{\alpha,j} }, \quad \text{where} \quad  R_{\alpha,j} = \sum_i T_{\alpha,i} \1{D_{\alpha,i}=j} $$
$R_{\alpha,j}$ corresponds to the number of triangles having a node of degree $j$ as a vertex, where triangles
having $k\leq 3$ degree-$j$ nodes as vertices are counted $k$ times.

We obtain an asymptotic expression for $E(R_{\alpha,j})$, and show that $\var( R_{\alpha,j} ) = O(\alpha^{1 -2a} [E(R_{\alpha,j})]^2 )$. We then prove that $R_{\alpha,j}/E(R_{\alpha,j}) $ goes to 1 almost surely. The latter is obtained by proving that $ R_{\alpha,j}$ is nearly monotonic increasing by constructing an increasing  sequence $\alpha_n $ going to infinity such that $E(R_{\alpha_n,j})/E(R_{\alpha_{n+1},j})$ goes to 1 and such that for all $\alpha \in (\alpha_n, \alpha_{n+1})$
$$ R_{\alpha_n,j} - \tilde R_{n,j} \leq R_{\alpha,j}\leq R_{\alpha_{n+1},j} + \tilde R_{n,j}, \quad \tilde R_{n,j}= o(E(R_{\alpha_n,j} )).$$
Roughly speaking $\tilde R_{n,j}$ corresponds to the  sum of the number of triangles from $i$, over the set $i$ such that $D_{n,i}\leq j$ and $i$ has at least one connection with some $i'$ such that $\theta_{i'} \in (\alpha_n, \alpha_{n+1})$. The result for the local clustering coefficient then follows from Toeplitz's lemma (see e.g. \cite[p. 250]{Loeve1977}).

\section{Central limit theorems}\label{sec:clt}

We now present central limit theorems (CLT) for subgraph counts (number of edges, triangles, etc.) and for the number of nodes $N_\alpha$. Subgraph counts can be expressed as $U$-statistics of Poisson random measures (up to an asymptotically negligible term). A CLT then follows rather directly from CLT on $U$-statistics of Poisson random measures~\citep{Reitzner2013}.

Obtaining a CLT for quantities like $N_\alpha$ is more challenging, since these cannot be reduced to $U$-statistics. We prove in this Section the CLT for $N_\alpha$ and we separate the dense and sparse cases because the techniques of the respective proofs are very different. The proof of the sparse case requires additional assumptions and is much more involved.  We believe that the same technique of proof can be used for other quantities of interest, such as the number $N_{\alpha,j}$ of nodes of degree $j$ with more tedious computations.

\subsection{CLT for subgraph counts}

\subsubsection{Statement of the result}
Let $F$ be a given subgraph, which has neither isolated vertices nor loops.  Denote $|F|$ the number of  nodes, $\{1, \cdots, |F|\}$ the set of vertices and $e(F)$ the set of edges.
Let $N^{(F)}_{\alpha}$ be the number of subgraphs $F$ in the graph $\mathcal{G}_\alpha$:

$$ N^{(F)}_\alpha =k_{(F)}\sum_{(v_1,\cdots, v_{|F|})}^{\neq} \prod_{(i,j) \in e(F)} Z_{v_{i},v_{j}} \1{\theta_{v_{i}}\le\alpha}\1{\theta_{v_{j}}\le\alpha},$$
where $k_{(F)}$ is a constant accounting for the multiple counts of $F$, that we can omit in the rest of the discussion since it does not depend on $\alpha$.
Note that this statistics covers the number of edges (excluding loops) if $|F|=2$ and the number of triangles if $|F|=3$ and $e(F)=\{(1,2),(1,3),(2,3)\}$. It is known in the graph literature as the number of injective adjacency maps from the vertex set of $F$ to the vertex set of $\mathcal{G}_\alpha$, see \cite[Section 2.5]{Borgs2018}.

\begin{proposition}\label{prop:CLT_subgraphs}
Let $F$ be a subgraph without loops nor isolated vertices. Assume that $\int_0^\infty \mu(x)^{2|F|-2}dx<\infty$. Then
\begin{equation}\label{th:CLT:countsk}
\frac{ N^{(F)}_{\alpha} - E\left (N^{(F)}_{\alpha}\right) }{ \sqrt{\var\left (N^{(F)}_{\alpha}\right) }} \to \mathcal N(0,1),
\end{equation}
as $\alpha$ goes to infinity, where
\begin{align}
E(N^{(F)}_{\alpha})& = k_{(F)}\alpha^{|F|}  \int_{\mathbb R_+^{|F|}} \prod_{(i,j)\in e(F)} W(x_i,x_j)dx_1\cdots dx_{|F|}<\infty
\end{align}
and
$$\var(N^{(F)}_{\alpha})\sim c_F \alpha^{2|F|-1} $$
for some positive constant $c_F$ that depends only on $F$.
\end{proposition}

\begin{remark}
If the graph is dense, $\mu$ is a bounded function with bounded support and therefore $\int_0^\infty \mu(x)^{p}dx<\infty$ for any $p$. In the sparse case, if $\mu$ is monotone, we necessarily have $\int_a^\infty \mu(x)^p dx<\infty$ for any $p>1$. The condition $\int_0^\infty \mu(x)^{2|F|-2}dx<\infty$ therefore requires additional assumptions on the behaviour of $\mu$ at 0, which drives the behaviour of large degree nodes.
\end{remark}

\subsubsection{Proof}

Recall that $M=\sum_{i}\delta_{(\theta_{i},\vartheta_{i})}$. The main idea of the proof is to use the decomposition
\begin{equation}\label{eq:decompositionCLTNF}
    N^{(F)}_{\alpha} - E(N^{(F)}_{\alpha})  = E(N^{(F)}_{\alpha}|M) - E(N^{(F)}_{\alpha}) + N^{(F)}_{\alpha} - E(N^{(F)}_{\alpha}|M),
    \end{equation}
and to show that $E(N^{(F)}_{\alpha}|M)$ is a geometric $U$-Statistic of a Poisson process, for which CLT have been derived by \cite{Reitzner2013}.

In this section, denote $K=|F|\geq 2$ the number of nodes of the subgraph $F$. The subgraph counts are
\begin{align*}
N^{(F)}_\alpha &=k_{(F)}\sum_{(v_1,\cdots, v_K)}^{\neq} \left ( \prod_{k=1}^K \1{\theta_{v_{k}}\leq \alpha} \right ) \frac{1}{|\mathbb S_K|}\sum_{\pi\in \mathbb S_K} \prod_{(i,j) \in e(F)} Z_{v_{\pi_i},v_{\pi_j}}
\end{align*}
where $\mathbb S_K$ denotes the set of permutations of $\{1,\ldots, K\}$.

Using the extended Slivnyak-Mecke theorem, we have
\begin{align}\label{eq:expectationNF}
E(N^{(F)}_{\alpha})& = k_{(F)}\alpha^K  \int_{\mathbb R_+^K} \prod_{(i,j)\in e(F)} W(x_i,x_j)dx_1\cdots dx_K.
\end{align}
As $\int_0^\infty \mu(x)^{K-1}dx<\infty$, Lemma 62 in \citep{Borgs2018} implies that $E(N^{(F)}_{\alpha})<\infty$. For any $K\geq 2$, define the symmetric function
\begin{equation*}
f(x_1,\ldots,x_K)=\frac{1}{|\mathbb S_K|} \sum_{\pi\in \mathbb S_K} \prod_{(i,j)\in e(F)} W(x_{\pi_i},x_{\pi_j});
\end{equation*}
additionally, using condition~\eqref{eq:integrabilityW} and $\int_0^\infty \mu(x)^{K-1}dx<\infty$, it satisfies $0 <\int_{\mathbb R_+^{K}} f(x_1,\ldots,x_K)dx_1\ldots dx_K<\infty.$

We state the following useful lemma.
\begin{lemma}\label{lemma:functionfsubgraphs}
The function $f$ satisfies for all $x_K\geq 0$
$$
g(x_K):=\int_{\mathbb R_+^{K-1}} f(x_1,\ldots,x_{K-1},x_K)dx_1\ldots dx_{K-1}\leq C_0\max(\mu(x_K),\mu(x_K)^{K-1})
$$
for some constant $C_0$.
\end{lemma}
\begin{proof}
Let $\pi\in\mathbb S_K$ and $r_K\in \{1,\ldots,K\}$ be such that $\pi_{r_K}=K$. Denote $S\subseteq \{1,\ldots,K-1\}$ the set of indices $i$ such that $(i,r_K)\in e(F)$ and $i$ has no other connections in $F$. Then
\begin{align*}
\int_{\mathbb R_+^{K-1}}&  \prod_{(i,j)\in e(F)} W(x_{\pi_i},x_{\pi_j}) dx_1\ldots dx_{K-1}\leq C_1 \int_{\mathbb R_+^{|S|}} \left [\prod_{i\in S} W(x_{\pi_i},x_K)dx_i\right ]\\
&= C_1 \mu(x_K)^{|S|} \leq C_1 \max(\mu(x_K),\mu(x_K)^{K-1})
\end{align*}
for some constant $C_1$.
\end{proof}

It follows from Lemma~\ref{lemma:functionfsubgraphs} and from the fact that $\int_0^\infty\mu(x)dx<\infty$ that, if $\int_0^\infty \mu(x)^{2K-2} dx<\infty$, then
\begin{align*}
\int_0^\infty \left (\int_{\mathbb R_+^{K-1}} f(x_1,\ldots,x_{K-1},y)dx_1\ldots dx_{K-1} \right )^2dy<\infty.
\end{align*}

We are now ready to derive the asymptotic expression for the variance of $N^{(F)}_\alpha$. Using the extended Slivnyak-Mecke theorem again,
\bgroup
\allowdisplaybreaks
\begin{align*}
&E((N^{(F)}_\alpha)^2)=E(E((N^{(F)}_\alpha)^2 \mid M))\\
&=k_{(F)}^2 E\left(\sum_{(v_1,\cdots, v_K,v'_1,\ldots,v_K')}^{\neq} f(\vartheta_{v_1},\ldots,\vartheta_{v_K})f(\vartheta_{v'_1},\ldots,\vartheta_{v'_K})\prod_{k=1}^{K}\1{\theta_{v_k}\leq \alpha}\1{\theta_{v'_k}\leq \alpha}  \right )\\
&+ k_{(F)}^2 K^2 E\left(\sum_{\substack{(v_1,\cdots, v_K, \\v'_1,\ldots,v_{K-1}')} }^{\neq} f(\vartheta_{v_1},\ldots,\vartheta_{v_K})f(\vartheta_{v'_1},\ldots,\vartheta_{v'_{K-1}},\vartheta_{v_{K}} )\1{\theta_{v_{K}}\leq\alpha}\prod_{k=1}^{K-1}\1{\theta_{v_k}\leq \alpha}\1{\theta_{v'_k}\leq \alpha}  \right )\\
&+O(\alpha^{2K-2})\\
&= k_{(F)}^2 K^2 \alpha^{2K-1}\int_{\mathbb R_+^{2K-1}} f(x_1,\ldots,x_K)f(x'_1,\ldots,x'_{K-1},x_{K})dx_1,\ldots dx_K dx'_1\ldots dx'_{K-1} \\
&+E(N^{(F)}_\alpha)^2 +O(\alpha^{2K-2}).
\end{align*}
\egroup
It follows that
\begin{align*}
\var(N^{(F)}_\alpha)\sim k_{(F)}^2 K^2  \alpha^{2K-1} \sigma_F^2
\end{align*}
as $\alpha$ tends to infinity, where
$$
\sigma_F^2=\int_0^\infty \left(\int_{\mathbb R_+^{K-1}} f(x_1,\ldots,x_{K-1},y)dx_1\ldots dx_{K-1} \right )^2 dy<\infty.
$$

We now prove the CLT.
The first term of the right-handside of Equation~\eqref{eq:decompositionCLTNF} takes the form
\begin{equation}\label{exp_N^F_cond}
E(N^{(F)}_{\alpha}|M)  = k_{(F)}  \sum_{(v_1,\cdots, v_K)}^{\neq} f(\vartheta_{v_1},\ldots,\vartheta_{v_K})\prod_{i=1}^K \1{\theta_{v_i}\leq \alpha}.
\end{equation}
By the superposition property of Poisson random measures, we have
$$
E(N^{(F)}_{\alpha}|M)  \overset{d}{=}k_{(F)}  \sum_{(v_1,\cdots, v_K)}^{\neq} f(\widetilde\vartheta_{v_1},\ldots,\widetilde\vartheta_{v_K})\prod_{i=1}^K \1{\widetilde\theta_{v_i}\leq 1}
$$

where the right-handside is a geometric $U$-statistic \cite[Definition 5.1]{Reitzner2013} of the Poisson point process $\{(\widetilde \theta_i,\widetilde \vartheta_i)_{i\geq 1}\}$ with mean measure $\alpha d\widetilde \theta d\widetilde\vartheta$ on $[0,1]\times\mathbb R_+$. Theorem 5.2 in \cite{Reitzner2013} therefore implies that
\begin{align}\label{CLTconditional}
\frac{E(N_{\alpha}^{(F)}\mid M)-E(N_\alpha^{(F)})}{\sqrt{\var(E(N_{\alpha}^{(F)} \mid M))}}\to\mathcal N(0,1)
\end{align}
where $\var(E(N_{\alpha}^{(F)} \mid M))\sim \var(N_{\alpha}^{(F)})\sim k_{(F)}^2 |F|^2  \alpha^{2|F|-1} \sigma_F^2$.
One can show similarly (proof omitted) that $\var(N_{\alpha}^{(F)}-E(N_{\alpha}^{(F)}\mid M))=o(\alpha^{2|F|-1})$. It follows from Equations~\eqref{eq:decompositionCLTNF}, \eqref{CLTconditional} and Chebyshev inequality that
$$
\frac{N_{\alpha}^{(F)}-E(N_{\alpha}^{(F)})}{\sqrt{\var(N_{\alpha}^{(F)})}}\to \mathcal N(0,1)
$$
as $\alpha$ tends to infinity.

\subsection{CLT for $N_\alpha$ (dense case)}

\subsubsection{Statement of the result}

In the dense case, $\mu$ has a bounded support. If it is monotone decreasing, then Assumption 1 is satisfied with $\sigma=0$ and $\ell(t)=\sup\{x>0\mid\mu(x)>0\}$ is constant. In this case a central limit theorem (CLT) applies, as described in the following theorem.

\begin{theorem}[Dense case]
\label{th:cltdense1} Assume
 that Assumption 1 holds with $\sigma=0$ and $\ell
(t)= C\in(0,\infty)$ where $C=\sup\{x>0\mid\mu(x)>0\}$ (dense case). Also assume Assumption 2 holds with $a=1$. Then
\begin{equation}
\frac{N_{\alpha}-E(N_\alpha)}{\sqrt{\alpha C} } \rightarrow\mathcal{N}(0,1).
\end{equation}
Moreover, $E(N_\alpha) = \alpha C -m_{\alpha,0}$ where
\begin{equation}\label{eq:malpha0}
m_{\alpha,0}=\alpha\int_{0}^{C}e^{-\alpha\mu(x)}(1-W(x,x)) dx=o(\alpha).
\end{equation}
\end{theorem}

$m_{\alpha,0}$ can be interpreted as the expected number of degree 0 nodes, and is finite in the dense case. $m_{\alpha,0}$ can either diverge or converge to a constant as $\alpha$ tends to infinity, as shown in the following examples.

\begin{example}
Consider $\mu(x)=\1{x\in\lbrack0,1]}$,
$\mu(x)=(1-x)^{2}\1{x\in\lbrack0,1]}$ and $\mu(x)=(1-x)^{3}\1{x\in\lbrack
0,1]}$. We respectively have $m_{\alpha,0}\rightarrow0$, $m_{\alpha,0}%
\sim\frac{\sqrt{\pi}}{2}\alpha^{1/2}$ and $m_{\alpha,0}\sim\Gamma
(4/3)\alpha^{2/3}$.
\end{example}

The above CLT for $N_\alpha$ can be generalised to $\widetilde N_{\alpha,j}=\sum_{k\geq j} N_{\alpha,k}$, the number of nodes of degree at least $j$.

\begin{theorem}
\label{th:cltdense2} Assume
 that Assumption 1 holds with $\sigma=0$ and $\ell
(t)= C\in(0,\infty)$ where $C=\sup\{x>0\mid\mu(x)>0\}$ (dense case). Also assume Assumption 2 holds with $a=1$. Then, for any $j\geq 1$
\begin{equation}
\frac{\widetilde N_{\alpha,j}-E(\widetilde N_{\alpha,j})}{\sqrt{\alpha C} } \rightarrow\mathcal{N}(0,1).
\end{equation}
Moreover, $E(\widetilde N_{\alpha,1})=E(N_\alpha)=\alpha C-m_{\alpha,0}$ and for $j\geq 2$, $E(\widetilde N_{\alpha,j}) = \alpha C -m_{\alpha,0}-\sum_{k=1}^{j-1}E(N_{\alpha,j})$ where $m_{\alpha,0}$ is defined in Equation~\eqref{eq:malpha0} and $E(N_{\alpha,j})$ is defined in Equation \eqref{eq:ExpNalphaj}. Note that $m_{\alpha,0}=o(\alpha)$ and for any $j\geq 1$, $E(N_{\alpha,j})=o(\alpha)$.
\end{theorem}

\subsubsection{Proof}

For a point ($\theta,\vartheta$) such that $\vartheta>C$, its degree is
necessarily equal to zero, as $\mu(\vartheta)=0$. Write%
\[
N_{\alpha}=Q_{\alpha}-N_{\alpha,0}, \quad \text{where} \quad Q_{\alpha}=\sum_{i}\1{\theta_{i}\leq\alpha}\1{\vartheta_{i}\leq C};
\]
$Q_{\alpha}$ is the total number of nodes $i$ with $\theta_i\leq \alpha$ that could have a
connection (hence such that $\mu(\vartheta_i)>0$), and
\[
N_{\alpha,0}=\sum_{i}\1{\theta_{i}\leq\alpha}\1{\vartheta_{i}\leq
C}\1{D_{\alpha,i}=0}%
\]
is the set of nodes $i$ with degree 0, but for which
$\theta_i\leq\alpha,\mu(\vartheta_i)>0$. In the dense regime, both $Q_{\alpha}$ and $N_{\alpha,0}$ are
almost surely finite. $(Q_{\alpha})_{\alpha\geq 0}$ is a homogeneous Poisson process with rate $C$. By
the law of large numbers, $Q_{\alpha}\sim\alpha C\sim N_{\alpha}$ almost surely as
$\alpha$ tends to infinity.
Using Campbell's theorem, the Slivnyak-Mecke formula, and monotone convergence, we have
$
E(N_{\alpha,0})=\alpha\int_{0}^{C}(1-W(x,x))e^{-\alpha\mu(x)}dx=o(\alpha).
$
We also have that
\begin{align*}
E(N_{\alpha,0}^{2})-E(N_{\alpha,0})=\alpha^{2}\int_{0}^{C}%
\int_{0}^{C}&(1-W(x,x))(1-W(y,y))(1-W(x,y))\\
& \times e^{-\alpha\mu(x)-\alpha\mu
(y)+\alpha\nu(x,y)}dxdy.
\end{align*}
Hence, using the inequality $e^{x}-1\leq xe^{x}$, we obtain
\begin{align*}
\var(N_{\alpha,0})  &  =\alpha^{2}\int_{0}^{C}%
\int_{0}^{C}(1-W(x,x))(1-W(y,y))(1-W(x,y))e^{-\alpha\mu(x)-\alpha\mu
(y)+\alpha\nu(x,y)}dxdy\\
& \quad -\alpha^{2}\left(  \int_{0}^{C}(1-W(x,x))e^{-\alpha\mu(x)}dx\right)  ^{2}+E(N_{\alpha,0})\\
&  \leq E(N_{\alpha,0})+\alpha^{3}\int_{0}%
^{C}\int_{0}^{C}\nu(x,y)e^{-\alpha\mu(x)-\alpha\mu(y)+\alpha\nu(x,y)}dxdy.
\end{align*}

Using Lemma \ref{lemma:boundintnu} in the Appendix and Assumption \ref{assumpt:2} with $a=1$,
$$\int_{0}^{C}\int%
_{0}^{C}\nu(x,y)e^{-\alpha/2\mu(x)-\alpha/2\mu(y)}dxdy=o(\alpha^{-2}).$$
It follows that $\var(N_{\alpha,0})=o(\alpha)$. This implies, using Chebyshev's inequality, the CLT for Poisson processes and  Slutsky's theorem that
\begin{equation*}
\frac{ N_\alpha - E(N_\alpha) }{ \sqrt{\alpha C} } =\frac{  Q_\alpha - \alpha C }{ \sqrt{\alpha C} } - \frac{  N_{\alpha,0} - E(N_{\alpha,0}) }{ \sqrt{\alpha C} } \rightarrow \mathcal N(0,1).
\end{equation*}

This concludes the proof of Theorem \ref{th:cltdense1}. The proof of Theorem \ref{th:cltdense2} follows similarly. Note that the case $j=1$ in Theorem \ref{th:cltdense2} corresponds to Theorem \ref{th:cltdense1}.
For any $j\geq 2$,
$
\widetilde N_{\alpha,j}=Q_{\alpha}-N_{\alpha,0}-\sum_{k=1}^{j-1} N_{\alpha,k}.
$
We have, using Cauchy-Schwarz inequality and Proposition \ref{prop:variancenbnodesj},
$$
\var\left (N_{\alpha,0}+\sum_{k=1}^{j-1} N_{\alpha,k}\right)\leq j\left (\var(N_{\alpha,0})+\sum_{k=1}^{j-1} \var(N_{\alpha,k})\right )=o(\alpha)
$$
This implies
\begin{equation*}
\frac{ \widetilde N_{\alpha,j} - E(\widetilde N_{\alpha,j}) }{ \sqrt{\alpha C} } =\frac{  Q_\alpha - \alpha C }{ \sqrt{\alpha C} } -\frac{N_{\alpha,0}+\sum_{k=1}^{j-1} N_{\alpha,k}-E(N_{\alpha,0}+\sum_{k=1}^{j-1} N_{\alpha,k})}{\sqrt{\alpha C}} \rightarrow \mathcal N(0,1).
\end{equation*}

\subsection{CLT for $N_\alpha$ (sparse case)}\label{subsec:CLT_sparse}

\subsubsection{Statement of the result}
We now assume that we are in the sparse regime, that is $\mu$ has unbounded support. We make the following
additional assumption in order to prove the asymptotic normality. This holds when $W$ is separable, as well as in the model of \cite{Caron2017} under some moment conditions (see Section \ref{ex:caronfox}).

\begin{assumption}\label{assumpt:5}
Assume that for any $j\leq 6$, and any $(x_1,\ldots,x_j)\in\mathbb R_+^j$
$$\int_0^\infty \prod_{i=1}^j W(x_i,y)dy \leq \prod_{i=1}^j L(x_i)\mu(x_i)$$
where $L$ is a locally integrable, slowly varying function converging to a (strictly positive) constant, and such that
$$
\int_0^\infty L(x)\mu(x)dx<\infty.
$$
\end{assumption}

We now state the central limit theorem for $N_\alpha$ under the sparse regime. Recall that in this case, when Assumption 1 holds, we either have $\sigma=0$ and $\ell(t)\to\infty$ or $\sigma\in(0,1]$.
\begin{theorem}[Sparse case]\label{th:cltsparse}
Assume that $\mu$ has an unbounded support (sparse regime). Under Assumptions \ref{assumpt:1}, \ref{assumpt:4} and \ref{assumpt:5}, we have
\[
\frac{N_{\alpha}-E(N_{\alpha})}{\sqrt{\var(N_{\alpha})}}%
\rightarrow\mathcal{N}(0,1).
\]

\end{theorem}

\begin{remark}
\label{prop:variance} As detailed in Proposition~\ref{prop:variancenbnodes}, under Assumptions \ref{assumpt:1} and \ref{assumpt:4}, we have, for any $\sigma\in[0,1]$ and any slowly varying function $\ell$,
$
\var(N_{\alpha})\asymp\alpha^{1+2\sigma}\ell_\sigma^{2}(\alpha)
$
where the slowly varying function $\ell_\sigma$ is defined in Equation~\eqref{eq:defellsigma}.
\end{remark}

\subsubsection{Proof}
The proof uses the recent results of \citet{Last2016} on normal approximations of non-linear functions of a Poisson random measure. We have the decomposition%
\begin{align*}
N_{\alpha}-E(N_{\alpha})  &  =(N_{\alpha}-E(N_{\alpha}\mid
M))+(E(N_{\alpha}\mid M)-E(N_{\alpha}))\\
&  =(N_{\alpha}-E(N_{\alpha}\mid M))+(M(h_{\alpha})-E(N_{\alpha}))+f_{\alpha}(M)
\end{align*}

where
\[
f_{\alpha}(M)=\sum_{i}\1{\theta_{i}\leq\alpha}\left[  (1-W(\vartheta
_{i},\vartheta_{i}))e^{-\alpha\mu(\vartheta_{i})}-e^{-M(g_{\alpha,\vartheta_{i}}%
)}\right]
\]
is a nonlinear functional of the Poisson random measure $M$, and
\[
M(h_{\alpha})=\sum_{i}\1{\theta_{i}\leq\alpha}\left[  1-(1-W(\vartheta
_{i},\vartheta_{i}))e^{-\alpha\mu(\vartheta_{i})}\right]
\]
is a linear functional of $M$ with $h_{\alpha}(\theta,\vartheta)=\1{\theta
\leq\alpha}\left[  1-(1-W(\vartheta,\vartheta))e^{-\alpha\mu(\vartheta)}\right]  $. Theorem \ref{th:cltsparse} is a direct consequence of the following three propositions  and of Slutsky's theorem.

\begin{proposition}
\label{prop:normpart1}Under Assumptions \ref{assumpt:1} and \ref{assumpt:4}, we have
\[
N_{\alpha}-E(N_{\alpha}\mid M)=\left \{
\begin{array}{ll}
  O(\alpha^{1/2+\sigma/2}\ell^{1/2}_\sigma(\alpha)) & \text{if }\sigma\in[0,1)\\
  o(\alpha^{1/2}\ell^{1/2}(\alpha)) & \text{if }\sigma=0
\end{array}\right.\text{ in probability}
\]
hence
\[
\frac{N_{\alpha}-E(N_{\alpha}\mid M)}{\sqrt{\var(N_{\alpha})}%
}\rightarrow0\text{ in probability}.
\]

\end{proposition}

\begin{proposition}
\label{prop:normpart2}Under Assumptions \ref{assumpt:1} and \ref{assumpt:4}, we have
\[
M(h_{\alpha})-E(N_{\alpha})=O(\alpha^{1/2+\sigma/2}\ell
^{1/2}(\alpha))\text{ in probability}
\]
hence, if $\mu$ has an unbounded support,
\[
\frac{M(h_{\alpha})-E(N_{\alpha})}{\sqrt{\var(N_{\alpha})}}%
\rightarrow0\text{ in probability}.
\]

\end{proposition}

The above two propositions are proved in Section \ref{sec:proofsecondCLT} of the Supplementary Material~\citep{CaronRousseau2017:supplement}. \\

\begin{proposition}
\label{prop:normpart3}Assume $\mu$ has an unbounded support. Under Assumptions \ref{assumpt:1}, \ref{assumpt:4} and \ref{assumpt:5}, we have
\[
\frac{f_{\alpha}(M)}{\sqrt{\var(N_{\alpha})}}\rightarrow\mathcal{N}(0,1).
\]

\end{proposition}

\begin{sketchproof} To prove Proposition \ref{prop:normpart3} we resort to \cite[Theorem 1.1]{Last2016} on the normal approximation of non-linear functionals of Poisson random measures.  Define
\begin{equation}
F_\alpha=\frac{f_{\alpha}(M)}{\sqrt{v_{\alpha}}}
\end{equation} where $v_{\alpha}= \var(f_{\alpha}(M))\sim \var(N_{\alpha} )\asymp\alpha^{1+2\sigma}\ell_\sigma^{2}(\alpha)$. Note that $E(F_\alpha)=0$ and $\var(F_\alpha)=1$. Consider the difference operator $D_{z}F_\alpha$ defined by%
\[
D_{z}F_\alpha=\frac{1}{\sqrt{v_\alpha}}( f_\alpha(M+\delta_{z})-f_\alpha(M)),
\]
and also%
\begin{align*}
D_{z_1,z_2}^{2}F_\alpha  &  =D_{z_2}(D_{z_1}F_\alpha)=D_{z_2}\left(\frac{1}{\sqrt{v_\alpha}} (f_\alpha(M+\delta_{z_1})-f_\alpha(M))\right )\\
&  =\frac{1}{\sqrt{v_\alpha}}\left( f_\alpha(M+\delta_{z_1}+\delta_{z_2})- f_\alpha(M+\delta_{z_1})- f_\alpha(M+\delta_{z_2}
)+ f_\alpha(M)\right ).
\end{align*}
Define
\begin{align*}
\gamma_{\alpha,1}  &  :=2\left(  \int_{\mathbb{R}_{+}^{6}}\sqrt{\mathbb{E}(D_{z_{1}%
}F_\alpha)^{2}(D_{z_{2}}F_\alpha)^{2}}\sqrt{\mathbb{E}(D_{z_{1},z_{3}}^{2}F_\alpha)^{2}%
(D_{z_{2},z_{3}}^{2}F_\alpha)^{2}}dz_{1}dz_{2}dz_{3}\right)  ^{1/2}\\
\gamma_{\alpha,2}  &  :=\left(  \int_{\mathbb{R}_{+}^{6}}\mathbb{E}\left [(D_{z_{1},z_{3}%
}^{2}F_\alpha)^{2}(D_{z_{2},z_{3}}^{2}F_\alpha)^{2}\right ]dz_{1}dz_{2}dz_{3}\right)
^{1/2}\\
\gamma_{\alpha,3}  &  :=\int_{\mathbb{R}_{+}^{2}}\mathbb{E}|D_{z}F_\alpha|^{3}dz
\end{align*}
In Section \ref{sec:proofCLTsparse} of the Supplementary Material \citep{CaronRousseau2017:supplement} we prove that, under Assumptions \ref{assumpt:1}, \ref{assumpt:4} and \ref{assumpt:5}, $\gamma_{\alpha,1},\gamma_{\alpha,2},\gamma_{\alpha,3}\rightarrow0$. The proof is rather lengthy, and makes repeated use of H\"older's inequality and of properties of integrals involving regularly varying functions (in particular Lemma \ref{lemma:abelianvariations2}). An application of \cite[Theorem 1.1]{Last2016} then implies that $F_\alpha\to\mathcal N(0,1)$.
\end{sketchproof}

\section{Related work and Discussion}
\label{sec:discussion}

\cite{Veitch2015} proved that Equation~\eqref{eq:nodesexpectationalmostsure} holds in probability, under slightly different assumptions: they assume that Assumption \ref{assumpt:2} holds with $a=1$ and that $\mu$ is differentiable, with some conditions on the derivative, but do not make any assumption on the existence of $\sigma$ or $\ell$. We note that for all the examples considered in Section~\ref{sec:examples}, Assumptions \ref{assumpt:1} and \ref{assumpt:2} are always satisfied, but Assumption \ref{assumpt:2} does not hold with $a=1$ for the non-separable graphon function \eqref{eq:nonseparablegraphon}. Additionally, the differentiability condition does not hold for some standard graphon models such as the stochastic blockmodel.
\cite{Borgs2018} proved, amongst other results, the almost sure convergence of the subgraph counts in graphex models (Theorem 56).
For the subclass of graphon models defined by Equation~\eqref{eq:graphonCF}, \cite{Caron2017} provided a lower bound on the growth in the number of nodes, and therefore an upper bound on the sparsity rate, using assumptions of regular variation similar to Assumption \ref{assumpt:1}. Applying the results derived in this Section, we show in Section \ref{sec:CF} that the bound is tight, and we derive additional asymptotic properties for this particular class.

As mentioned in the introduction, another class of (non projective) models that can produce sparse graphs are sparse graphons~\citep{BollobasRiordan2009,BickelChen2009,BickelChenLevina2011,Wolfe2013}. In particular, a number of authors considered the following sparse graphon model, where two nodes $i$ and $j$ in a graph of size $n$ connect with probability $\rho_n W(U_i,U_j)$ where $W:[0,1]^2\to [0,1]$ is the graphon function, measurable and symmetric and $\rho_n\to 0$. Although such model can capture sparsity, it has rather different properties compared to those of graphex models. For example, the global clustering coefficient for this sparse graphon model converges to 0, while the clustering coefficient converges to a positive constant, as shown in Proposition \ref{prop:globalclustering}.

Also  graphex processes include as a special case dense vertex-exchangeable random graphs~\citep{Hoover1979,Aldous1981,Lovasz2006,Diaconis2008}, that is models based on a graphon on $[0,1]$. They also include as a special case the class of graphon models over more general probability spaces~\citep{Bollobas2007}; see \cite[p.21]{Borgs2018} for more details.
Some other classes of graphs, such as geometric graphs arising from Poisson processes in different spaces~\citep{Penrose2003}, cannot be cast in this framework.

\section{Examples of sparse and dense models}
\label{sec:examples}

We provide here some examples of the four different cases: dense, almost dense, sparse and almost extremely sparse. We also show that the results of the previous section apply to the particular model studied by \cite{Caron2017}.

\subsection{Dense graph}

Let us consider the graphon function
$$
W(x,y)=(1-x)(1-y)~\1{x\leq1}~\1{y\leq1}$$ which has bounded support. The corresponding marginal graphon function
$
\mu(x)= \1{x\leq1} (1-x)/2%
$
has inverse $\mu^{-1}(x)=\ell(1/x)$ where $\ell(1/x)=(1-2x)\1{x\leq1/2}$  is slowly varying since $\ell(1/x)\rightarrow 1$. Assumptions 1 and 2 are satisfied, hence by Theorem \ref{th:asnumbernodes} and Corollary~\ref{th:sparsity1}
\begin{equation*}
N_\alpha\sim \alpha,
\quad \Nedges\sim \alpha^2/8, \quad
\Nedges\sim N_\alpha^2/8,\quad
\frac{N_{\alpha,j}}{N_\alpha}\rightarrow 0 \quad j \geq 1
\end{equation*}
almost surely as $\alpha\rightarrow\infty$. The function $W$ is separable and $C_\alpha^{(g)}\to 4/9$.

\subsection{Sparse, almost dense graph without power-law}

Consider the graphon function, considered by \cite{Veitch2015},
$$
W(x,y)=e^{-x-y}
$$
which has full support.
The corresponding function $\mu(x)=e^{-x}$ has inverse $\mu^{-1}(x)=\ell(1/x)=\log(1/x)\1{0<x<1}$,
which is a slowly varying function. We have $\ell_0^*(x)=1/\log(x)^2$. Assumptions 1 and 2 are satisfied and
\begin{align*}
N_\alpha\sim \alpha\log (\alpha),~~~\Nedges\sim \alpha^2/2,~~~
\Nedges\sim \frac{N_\alpha^2}{2\log(N_\alpha)^2},~~~
\frac{N_{\alpha,j}}{N_\alpha}\rightarrow 0\text{ for all }j=1,2,\ldots
\end{align*}
The function $W$ is separable, and  $C_\alpha^{(g)}\to 1/4$.

\subsection{Sparse graphs with power-law}
\label{sec:example:powerlaw}

We consider two examples here, a separable and a non-separable one. Interestingly, while both examples have similar power-law behaviours regarding the degree distribution, the clustering properties are very different. In the first example, the local clustering coefficient converges to a strictly positive constant, while in the second example, it converges to 0.

\paragraph{Separable example.}  First, consider the function%
\[
W(x,y)=(x+1)^{-1/\sigma}(y+1)^{-1/\sigma}%
\]
with $\sigma\in(0,1)$. We have
$\mu(x)=\sigma (x+1)^{-1/\sigma}/(1-\sigma)$, $
\mu^{-1}(x)=x^{-\sigma}(1/\sigma-1)^{-\sigma}-1$, $\ell(t)\sim (1/\sigma-1)^{-\sigma}$ and
$\ell_\sigma^*(t)\sim \left \{(1/\sigma-1)^{-\sigma}\Gamma(1-\sigma) \right \}^{-2/(1+\sigma)}$.
Assumptions 1 and 2 are satisfied. We have
$
N_\alpha\sim \alpha^{1+\sigma}\Gamma(1-\sigma)(1/\sigma-1)^{-\sigma}$, $
\Nedges\sim \alpha^2 \sigma^2/\{2(1-\sigma)^2\}$ and
\begin{align*}
\Nedges&\sim  \frac{\sigma^2\left \{\Gamma(1-\sigma)(\frac{1}{\sigma}-1)^{-\sigma} \right \}^{-\frac{2}{1+\sigma}}}{2(1-\sigma)^2}  N_\alpha^{2/(1+\sigma)}, \quad
\frac{N_{\alpha,j}}{N_\alpha}\rightarrow  \frac{\sigma \Gamma(j-\sigma)}{j!\Gamma(1-\sigma)}, \quad j\geq 1.
\end{align*}
The function is separable, and  we obtain, for $\sigma\in(0,1)$
\begin{align*}
\lim_{\alpha\to\infty} C_\alpha^{(g)}=\left (\frac{1-\sigma}{2-\sigma}\right )^2\text{ and }\lim_{\alpha\to\infty} C_{\alpha,j}^{(\ell)}=\left (\frac{1-\sigma}{2-\sigma}\right )^2\text{ almost surely}.
\end{align*}

\paragraph{Non-separable example.} Consider now the non-separable function
\begin{equation}
W(x,y)=(x+y+1)^{-1/\sigma-1}\label{eq:nonseparablegraphon}
\end{equation}
where $\sigma\in(0,1)$. We have
$
\mu(x)=\sigma (x+1)^{-1/\sigma}$,
$\mu^{-1}(x)=\sigma^{\sigma} x^{-\sigma}-1$, $
\ell(t)\sim \sigma^{\sigma}$ and
$\ell_\sigma^*(t)\sim \left \{\sigma^\sigma\Gamma(1-\sigma) \right \}^{-2/(1+\sigma)}$. Assumptions 1 and 2 are satisfied as for all $(x,y)\in\mathbb R_+^2$
\begin{align*}
W(x,y)&\leq (x+1)^{-1/(2\sigma)-1/2}(y+1)^{-1/(2\sigma)-1/2}=\sigma^{-1-\sigma} \mu(x)^{\frac{1+\sigma}{2}}\mu(y)^{\frac{1+\sigma}{2}}.
\end{align*} We have
$
N_\alpha\sim \alpha^{1+\sigma}\Gamma(1-\sigma)\sigma^\sigma$, $\Nedges\sim \alpha^2 \sigma^2/\{2(1-\sigma)\}$ and
\begin{align*}
\Nedges&\sim  \frac{\sigma^2\left [\Gamma(1-\sigma)\sigma^{\sigma} \right ]^{-\frac{2}{1+\sigma}}}{2(1-\sigma)}  N_\alpha^{2/(1+\sigma)}, \quad
\frac{N_{\alpha,j}}{N_\alpha}\rightarrow  \frac{\sigma \Gamma(j-\sigma)}{j!\Gamma(1-\sigma)} , \quad j \geq 1.
\end{align*}

We have $\int \mu(x)^2dx=\frac{\sigma^3}{2-\sigma}$. There is no analytical expression for $\int W(x,y)W(y,z)W(x,z)dxdydz$, but this quantity can be evaluated numerically, and is non-zero, so the global clustering coefficient converges almost surely to a non-zero constant for any $\sigma\in(0,1)$. For the local clustering coefficient, we have $\mu(x)^2\sim \sigma^2x^{-2/\sigma}$ as $x\to\infty$ and
\begin{align*}
\int W(x,y)W(x,z)W(y,z)dydz\leq x^{-2/\sigma -2}\int (y+z+1)dydz =o(\mu(x)^2).
\end{align*}
Hence the local clustering coefficients $C_{\alpha,j}^{(\ell)}$ converge in probability to 0 for all $j$.

\subsection{Almost extremely sparse graph}

Consider the function
$$
W(x,y)=\frac{1}{(x+1)(1+\log(1+x))^2}\frac{1}{(y+1)(1+\log(1+y))^2}.
$$
We have $\overline W=1$ and $\mu(x)=(x+1)^{-1}(1+\log(1+x))^{-2}$ and, using properties of inverses of regularly varying functions,
$
\mu^{-1}(x)\sim x^{-1}\ell(1/x)
$
as $x\rightarrow 0$, where $\ell(t)=\log(t)^{-2}$ is a slowly varying function. We have, for $t>1$,
$
\ell_1(t)=\int_t^\infty x^{-1}\ell(x)dx=1/\log(t)
$
and
$
\ell_1^*(t)\sim \log(t)/2.
$
Assumptions 1 and 2 are satisfied, and almost surely
\begin{align*}
\Nedges\sim \alpha^2/2,~~~
N_\alpha\sim \frac{\alpha^2}{\log(\alpha)},~~~
\Nedges\sim \frac{1}{4} N_\alpha \log(N_\alpha),\\
\frac{N_{\alpha,1}}{N_\alpha}\rightarrow 1,~~~
\frac{N_{\alpha,j}}{N_\alpha}\rightarrow 0\text{ for all }j\geq 2.
\end{align*}
$\int\mu(x)^2dx=\frac{1}{6}(2+e\text{Ei}(-1))\simeq 0.24$ where $\text{Ei}$ is the exponential integral, hence $C_\alpha^{(g)}\to 0.0576$ almost surely.

\subsection{Model of \cite{Caron2017}}
\label{ex:caronfox}
\label{sec:CF}

\cite{Caron2017} studied a particular subclass of non-separable graphon models. This class is very flexible and allows to span the whole range of sparsity and power-law behaviours described in Section~\ref{sec:asympstats}. As shown by \cite{Caron2017}, efficient Monte Carlo algorithms can be developed for estimating the parameters of this class of models. Additionally, \cite[Corollary 1.3]{Borgs2019} recently showed that this class is the limit of some sparse configuration models, providing further motivation for the study of their mathematical properties.

Let $\rho$ be a L\'evy measure on $(0,+\infty)$ and $\overline{\rho}(x)=\int_{x}^{\infty}\rho(dw)$ the corresponding tail L\'evy intensity with generalised inverse $\overline{\rho}^{-1}(x)=\inf\{u>0|\overline{\rho}(u)<x\}$. \cite{Caron2017} introduced the model defined by
\begin{align}\label{eq:graphonCF}
W(x,y)=\left\{
\begin{array}
[c]{ll}%
1-e^{-2\overline{\rho}^{-1}(x)\overline{\rho}^{-1}(y)} & x\neq y\\
1-e^{-\{\overline{\rho}^{-1}(x)\}^{2}} & x=y
\end{array}
\right. .
\end{align}
$w=\overline \rho^{-1}(x)$ can be interpreted as the sociability of a node with parameter $x$. The larger this value, the more likely it is to connect to other nodes.
The tail L\'evy intensity $\overline \rho$ is a monotone decreasing function; its behaviour at 0 will control the low degree nodes while its behaviour at infinity will control the behaviour of high degree nodes.

The following proposition formalises this and shows how the results of Sections~\ref{sec:asympstats} and \ref{sec:clt} apply to this model. Its proof is given in Section~\ref{sec:proofCaronFoxexample}.
\begin{proposition}\label{prop:CF2017}
Consider the graphon function $W$ defined by Equation~\eqref{eq:graphonCF} with L\'evy measure $\rho$ and tail L\'evy intensity $\overline\rho$.
Assume $m=\int_0^\infty w\rho(dw)<\infty$ and
\begin{equation}
\overline{\rho}(x)\sim x^{-\sigma}\widetilde\ell(1/x)\text{ as }x\to 0
\label{eq:asymptrho_zero}
\end{equation} for some $\sigma\in [0,1]$ and some slowly varying function $\widetilde\ell$. Then
Equation~\eqref{eq:integrabilityW} and Assumptions 1 and 2 hold, with $a=1$ and $\ell(x)=(2m)^\sigma \widetilde \ell(x).$ Proposition \ref{th:numberedges}, Theorems~\ref{th:meannumbernodes}, \ref{th:asnumbernodes} and Corollary \ref{th:sparsity1}  therefore hold. 
If $\int_0^\infty \psi(2w)^2\rho(dw)<\infty$, where $\psi(t)=\int(1-e^{-wt})\rho(dw)$ is the Laplace exponent, then the global clustering coefficient converges almost surely
$$
\lim_{\alpha\to\infty}C_\alpha^{(g)}= \frac{\int_{\mathbb R_+^3} (1-e^{-2xy})(1-e^{-2xz})(1-e^{-2yz})\rho(dx)\rho(dy)\rho(dz)}{\int_0^\infty \psi(2w)^2\rho(dw)}
$$
and when $\sigma\in(0,1)$, Proposition \ref{prop:localclustering} holds and for any $j\geq 2$
\begin{align*}
\lim_{\alpha\to\infty}C_{\alpha, j}^{(\ell)}=\lim_{\alpha\to\infty}\overline C_{\alpha}^{(\ell)}= &\ 1-\frac{\int_{\mathbb R_+^2} yze^{-2yz}\rho(dy)\rho(dz)}{m^2},
\end{align*}
almost surely.
For a given subgraph $F$, the CLT for the number of such subgraphs (Proposition~\ref{prop:CLT_subgraphs}) holds if $\int \psi(2\overline\rho^{-1}(x))^{2|F|-2}dx<\infty$.
Under Assumption~\ref{assumpt:1}, this condition always holds if $\sigma=0$; for $\sigma\in(0,1]$, it holds if $\overline \rho(x)=O(x^{-(2|F|-2)\sigma-\epsilon})$ as $x\to\infty$ for some $\epsilon>0$. In this case, we have
\begin{equation}
\frac{N_{\alpha}^{(F)}-E(N_\alpha^{(F)})}{\sqrt{\var(N_\alpha^{(F)})} } \rightarrow\mathcal{N}(0,1).
\end{equation}

Moreover, if $\int w^6\rho(dw)<\infty$, then Assumptions \ref{assumpt:4} and \ref{assumpt:5} also hold. It follows that Theorems~\ref{th:cltdense1}, \ref{th:cltdense2} and \ref{th:cltsparse} apply and, for any $\sigma\in[0,1]$ and any $\ell$,
\begin{equation}
\frac{N_{\alpha}-E(N_\alpha)}{\sqrt{\var(N_\alpha)} } \rightarrow\mathcal{N}(0,1).
\end{equation}

Finally, assume $\sigma\in(0,1)$ and $\widetilde\ell(t)=c>0$. If additionally
\begin{equation}
\overline \rho(x)\sim  c_0 x^{-\sigma\tau}\text{ as }x\to\infty
\label{eq:asymptrho_inf}
\end{equation}
for some $\tau>0,c_0>0$ then Assumption 3 is also satisfied with $\tau>0$, $\ell_2(x)=\frac{c_0}{2^{\sigma\tau}c^\tau\Gamma(1-\sigma)^\tau} $ and Proposition \ref{prop:largedegrees} applies; that is,
for
fixed $\alpha$
\[
E(N_{\alpha,j})\sim\frac{\alpha^{1+\tau}\tau\ell_{2}(j)}{j^{1+\tau}}\text{ as
}j\rightarrow\infty.
\]
\end{proposition}

\noindent We consider below two specific choices of mean measures $\rho$. Both measures have similar properties for large graph size $\alpha$, but different properties for large degrees $j$.

\paragraph{Generalised Gamma measure.} Let $\rho$ be the generalised gamma measure
\begin{equation}
\rho(dw)=1/\Gamma(1-\sigma_0)w^{-1-\sigma_0}e^{-\tau_0 w}dw\label{eq:GGP}
\end{equation}
with $\tau_0>0$ and $\sigma_0\in(-\infty,1)$. The tail L\'evy intensity satisfies
\begin{align*}
\overline\rho(x)\sim \left \{
\begin{array}{ll}
  \frac{1}{\Gamma(1-\sigma_0)\sigma_0}x^{-\sigma_0} & \sigma_0>0 \\
  \log(1/x) & \sigma_0= 0 \\
  -\frac{\tau_0^{\sigma_0}}{\sigma_0}& \sigma_0<0
\end{array}\right .
\end{align*}
as $x\to 0$. Then for $\sigma_0\in(0,1)$ (sparse with power-law)
\begin{align*}
\Nedges\asymp  N_\alpha^{2/(1+\sigma_0)}, \quad
\frac{N_{\alpha,j}}{N_\alpha}\rightarrow  \frac{\sigma_0 \Gamma(j-\sigma_0)}{j!\Gamma(1-\sigma_0)}, \quad j \geq 1.
\end{align*}
For $\sigma_0=0$ (sparse, almost dense),
$
\Nedges\asymp  N^2_\alpha/\log(N_\alpha)^2$ and $N_{\alpha,j}/N_\alpha\rightarrow  0,  j \geq 1;
$
 for $\sigma_0<0$ (dense)
$
\Nedges\asymp  N^2_\alpha$ and $N_{\alpha,j}/N_\alpha\rightarrow  0, \quad j \geq 1
$ almost surely as $\alpha$ tends to infinity. The constants in the asymptotic results are omitted for simplicity of exposure but can be obtained as well from the results of Section~\ref{sec:asympstats}. $\int w^p\rho(dw)<\infty$ for all $p\geq 1$, hence the global clustering coefficient converges, and the CLT applies for the number of subgraphs and the number of nodes. Note that Equation~\eqref{eq:asymptrho_inf} is not satisfied, as the L\'evy measure has exponentially decaying tails, and Proposition \ref{prop:largedegrees} does not apply.  The asymptotic properties of this model are illustrated in Figure~\ref{fig:illustration} for $\sigma_0=0.2$ and $\tau_0=2$ (sparse, power-law regime).

\paragraph{Generalised gamma Pareto measure.} Consider the generalised gamma Pareto measure, introduced by \cite{Ayed2019,Ayed2020}
$$
\rho(dw)= \frac{1}{\Gamma(1-\sigma)}w^{-1-\sigma\tau}\gamma(\sigma(\tau-1),\beta w)dw
$$
where $\gamma(s,x)=\int_0^x u^{s-1}e^{-u}du$ is the lower incomplete gamma function, $c>0$, $\tau>1$, $\sigma\in(0, 1)$.  The tail L\'evy intensity satisfies
\begin{align*}
\overline\rho(x)&\sim  cx^{-\sigma}\text{ as }x\to 0\\
\overline\rho(x)&\sim c_0x^{-\sigma\tau}\text{ as }x\to \infty
\end{align*}
where $c=\frac{\beta^{\sigma(\tau-1)}}{\sigma^2(\tau-1)\Gamma(1-\sigma)}$ and $c_0=\frac{\Gamma(\sigma(\tau-1))}{\sigma\tau\Gamma(1-\sigma)}$. It is both regularly varying at 0 and infinity and satisfies \eqref{eq:asymptrho_zero} and \eqref{eq:asymptrho_inf}. We therefore have, almost surely,
\begin{align*}
\Nedges\asymp  N_\alpha^{2/(1+\sigma)}, \quad
\frac{N_{\alpha,j}}{N_\alpha}\rightarrow  \frac{\sigma_0 \Gamma(j-\sigma)}{j!\Gamma(1-\sigma)}, \quad j \geq 1.
\end{align*}
Proposition \ref{prop:largedegrees} applies and, for large degree nodes,
 \[
E(N_{\alpha,j})\sim\frac{\tau\alpha^{1+\tau}c_0}{2^{\sigma\tau }c^\tau\Gamma(1-\sigma)^\tau} \frac{1}{j^{1+\tau}}\text{ as
}j\rightarrow\infty.
\]
The global clustering coefficient converges if $\tau>2$, and the CLT applies for the number of subgraphs $F$ if $\tau>2|F|-2$, and for the number of nodes if $\sigma\tau>6$.

\subsection{Proof of Proposition \ref{prop:CF2017}}\label{sec:proofCaronFoxexample}
 The marginal graphon function is given by
$
\mu(x)=\psi(2\overline \rho^{-1}(x))
$
where $\psi(t)=\int_0^\infty (1-e^{-wt})\rho(dw)$ is the Laplace exponent. Its generalised inverse is given by
$
\mu^{-1}(x)=\overline\rho(\psi^{-1}(x)/2).
$
The Laplace exponent satisfies $\psi(t)\sim m t$ as $t\to 0$. It therefore follows that $\mu^{-1}$ satisfies Assumption 1 with
$
\ell(x)=(2m)^\sigma \widetilde \ell(x).
$
Ignoring loops, the model is of the form given by Equation \eqref{eq:bernoullipoissonlink} with $f(x)=2m \overline{\rho}^{-1}(x)$. Assumption 2 is therefore satisfied. Regarding the global clustering coefficient, $\int \psi(2w)^2\rho(dw)\leq 4\int w^2\rho(dw)<\infty$ so its limit is finite. For the local clustering coefficient, using dominated convergence and the inequality $\frac{1-e^{-2\overline{\rho
}^{-1}(x)y}}{2\overline{\rho}^{-1}(x)}\leq y$, we obtain
\begin{align*}
\int W(x,y)W(y,z)W(x,z)dydz  &  =\int(1-e^{-2\overline{\rho}^{-1}%
(x)y})(1-e^{-2\overline{\rho}^{-1}(x)z})(1-e^{-2yz})\rho(dy)\rho(dz)\\
&  \sim4\overline{\rho}^{-1}(x)^{2}\int yz(1-e^{-2yz})\rho(dy)\rho(dz)
\end{align*}
Using the fact that $\mu(x)=\psi(2\overline \rho^{-1}(x))\sim 2m\overline \rho^{-1}(x)$ as $x\to\infty$, we obtain the result. Finally, if $\overline\rho$ satisfies \eqref{eq:asymptrho_zero}, then $\psi(t)\sim \Gamma(1-\sigma)\widetilde \ell(t)t^\sigma$ as $t\to\infty$. Using \cite[Proposition 1.5.15]{Bingham1987} $$\psi^{-1}(t)\sim \Gamma(1-\sigma)^{-1/\sigma}\widetilde\ell^{\#1/\sigma}(t^{1/\sigma})t^{1/\sigma}$$
as $t\to\infty$, where $\widetilde{\ell}^{\#}$ is the de Bruijn conjugate of $\widetilde{\ell}%
$. We obtain
$
\psi^{-1}(t)=\ell_3(t^{1/\sigma})t^{\frac{1}{\sigma}}%
$
where $\ell_3$ is a slowly varying function with
$
\ell_3(t^{1/\sigma})\sim\widetilde{\ell}^{\#1/\sigma}(t^{1/\sigma})\Gamma
(1-\sigma)^{-1/\sigma}\text{ as }t\rightarrow\infty.
$
We therefore have
$ \mu^{-1}(t)  \sim c_0 2^{-\tau\sigma}\ell_3(t^{1/\sigma})^{\sigma\tau} t^{\tau}\text{ as }t\to\infty.
$
If $\widetilde\ell(t)=c$, then $\ell_3(t)=(c\Gamma(1-\sigma))^{-1/\sigma}$.

For the CLT for the number of subgraphs $F$ to hold, we need $\int_0^\infty \mu(x)^{2|F|-2}dx<\infty$. As $\mu$ is monotone decreasing and integrable, we only need $\mu(x)^{2|F|-2}=\psi(2\overline\rho^{-1}(x))^{2|F|-2}$ to be integrable in a neighbourhood of 0.
In the dense case, $\psi(t)$ is bounded, and the condition holds. If $\overline\rho$ satisfies \eqref{eq:asymptrho_zero}, then $\psi(t)\sim \Gamma(1-\sigma)\widetilde \ell(t)t^\sigma$ as $t\to\infty$. For $\sigma\in(0,1]$ (sparse regime) the condition holds if $\overline\rho(x)=O(x^{-(2|F|-2)\sigma-\epsilon})$ as $x\to\infty$ for some $\epsilon>0$.

We now check the assumptions for the CLT for the number of nodes. Noting again that $\mu(x)\sim 2m\overline\rho^{-1}(x)$ as $x\to\infty$, we have, using the inequality $1-e^{-x}\leq x$,
\begin{align*}
\nu(x,y)&=\int \left (1-e^{-2\overline\rho^{-1}(x)w}\right )\left (1-e^{-2\overline\rho^{-1}(y)w}\right )\rho(dw)\\
&\leq L(x)L(y)\mu(x)\mu(y)
\end{align*}
where $L(x)=2 \frac{\overline\rho^{-1}(x)}{\mu(x)}\sqrt{\int w^2\rho(dw)}\to \sqrt{\int w^2\rho(dw)}/m$ as $x\to\infty$. Using now the inequality $1-e^{-x}\geq xe^{-x}$, we have
\begin{align*}
\nu(x,y)&\geq  4\overline\rho^{-1}(x)\overline\rho^{-1}(y)\int w^2e^{-2(\overline\rho^{-1}(x)+\overline\rho^{-1}(y))w}\rho(dw)
\end{align*}
As $\int w^2e^{-2(\overline\rho^{-1}(x)+\overline\rho^{-1}(y))w}\rho(dw)\to\int w^2\rho(dw) $ as $\min(x,y)\to\infty$, there is $C_0=2\int w^2\rho(dw)$ and $x_0$ such that for all $x,y>x_0$, $\nu(x,y)\geq C_0\mu(x)\mu(y)$.

More generally, if $\int w^6\rho(dw)<\infty$, then for any $j\leq 6$
\begin{align*}
\int_0^\infty \prod_{i=1}^j W(x_i,y)dy &\leq \prod_{i=1}^j L(x_i)\mu(x_i)
\end{align*}
where $L(x)=2 \frac{\overline\rho^{-1}(x)}{\mu(x)}\max\left (1, \max_{j=1,\ldots,6}\int w^j \rho(dw)\right )\to \max\left (1, \max_{j=1,\ldots,6}\int w^j \rho(dw)\right )/m$ as $x\to\infty$. Note also that $\int L(x)\mu(x)dx= 2\max\left (1, \max_{j=1,\ldots,6}\int w^j \rho(dw)\right ) \int w\rho(dw)<\infty$.

\section{Sparse and dense models with local structure}
\label{sec:localglobal}

In this section, we develop a class of models which allows to control separately the local structure, for example the presence of communities or particular subgraphs, and the global sparsity/power-law properties. The class of models introduced can be used as a way of sparsifying any dense graphon model.

\subsection{Statement of the results}

Due to Kallenberg's representation theorem, any exchangeable point process can be represented by Equation \eqref{eq:Zij}. However, it may be more suitable to use a different formulation where the function $W$ is defined on a general space, not necessarily $\mathbb R_+^2$, as discussed by~\cite{Borgs2018}. Such a construction may lead to more interpretable parameters and easier inference methods. Indeed, a few sparse vertex-exchangeable models, such as the models of \cite{Herlau2016} or \cite{Todeschini2016} are written in a way such that it is not straightforward to express them in the form given by  \eqref{eq:Zij}.

In this section we show that the above results easily extend to models expressed in the following way. Let $F$ be a probability space.  Writing  $\vartheta = (u,v)\in \mathbb R_+\times F$, let $\xi(d\vartheta)=du G(dv)$ where $G$ is some probability distribution on  $F$. Consider models expressed as in \eqref{eq:pointprocess} with
\begin{align}
Z_{ij}\mid (\theta_k,\vartheta_k)_{k=1,2,\ldots}\sim \text{Bernoulli}\{ W\ ( \vartheta_i, \vartheta_j)\}, \quad W : (\mathbb R_+\times F)^2 \rightarrow[0,1]\label{eq:generalW}
\end{align}
where $(\theta_k,\vartheta_k)_{k=1,2,\infty}$ are the points of a Poisson point process with mean measure $d\theta\xi(d \vartheta)$ on $\mathbb R_+\times (\mathbb R_+ \times F)$.  Let us assume additionally that the function $W$ factorizes in the following way
\begin{equation}
W((u_i,v_i),(u_j,v_j))=\omega(v_i,v_j)\eta(u_i,u_j).\label{eq:separableW}
\end{equation}
where $\omega:F\times F\rightarrow [0,1]$ and the function $\eta:\mathbb R_+\times\mathbb R_+\rightarrow [0,1]$ is integrable.
 In this model $\omega$ can capture the local structure, as in the classical dense graphon, and $\eta$ the sparsity behaviour of the graph. Let $\mu_\eta(u)=\int_0^\infty \eta(u,u')du'$, $\mu_\omega (v) = \int_F\omega(v,v')G(dv')$ and $\nu_\eta(x,y)=\int_{\mathbb R_+^2} \eta(x,z)\eta(y,z)dz$. The results presented in Section \ref{sec:asympstats} remain valid when $\mu_\eta$ and $\nu_\eta$ satisfy Assumptions \ref{assumpt:1} and \ref{assumpt:2}. The proof of Proposition~\ref{prop:local-global} is given in Section~\ref{sec:prooflocalglobal}.
\begin{proposition}\label{prop:local-global}
Consider the model defined by Equations \eqref{eq:generalW} and \eqref{eq:separableW} and assume that the functions $\mu_\eta$ and $\nu_\eta$ satisfy Assumptions 1 and 2. Then the conclusions of Proposition \ref{th:numberedges} hold and so do the conclusions of Theorems \ref{th:meannumbernodes} and \ref{th:asnumbernodes} with $\ell(\alpha)$ and $\ell_1(\alpha)$ replaced respectively  by
$$ \tilde \ell(\alpha ) = \ell(\alpha) \int_F \mu_\omega(v)^\sigma G(dv) , \quad \tilde \ell_1(\alpha) = \ell_1(\alpha)   \int_F \mu_\omega(v)^\sigma G(dv).$$

\end{proposition}

Consider for example the following class of models for sparse and dense stochastic block-models.

 \begin{example}[Dense and Sparse stochastic block-models]
Consider $F=[0,1]$ and $G$ the uniform distribution on $[0,1]$. We choose for $\omega$ the graphon function associated to a (dense) stochastic block-model. For some partition $A_1,\ldots,A_p$ of $[0,1]$, and any $v,v'\in[0,1]$, let
\begin{equation}
\omega(v,v')= B_{k,\ell}
\label{eq:SBM}
\end{equation}

with $v\in A_k$, $v'\in A_\ell$ and $B$ is a $p\times p$ matrix where $B_{k,\ell}\in[0,1]$ denotes the probability that a node in community $k$ forms a link with a node in community $\ell$. $\omega$ defines the community structure of the graph, and $\eta$ will tune its sparsity properties. Choosing $\eta(x,y)=\1{x\leq 1}\1{y\leq 1}$ yields the dense, standard stochastic block-model. Choosing $\eta(x,y)=\exp(-x-y)$ yields a sparse stochastic block-model without power-law behaviour, etc. An illustration of this model to obtain sparse stochastic block-models with power-law behaviour, generalizing the model of Section \ref{sec:example:powerlaw}, is given in Figure \ref{fig:sparseSBM}. The function $\omega$ is defined by:
$A_1=[0,0.5),A_2=[0.5,0.8),A_3=[0.8,1],B_{11}=0.7,B_{22}=0.5,B_{33}=0.9,B_{12}=B_{13}=0.1,B_{23}=0.05$
                                       and $\eta(x,y)=(1+x)^{-1/\sigma}(1+y)^{-1/\sigma}$, with $\sigma=0.8$.
\begin{figure}[t]
\begin{center}
\subfigure[Function $\omega$]{\includegraphics[width=.38\textwidth]{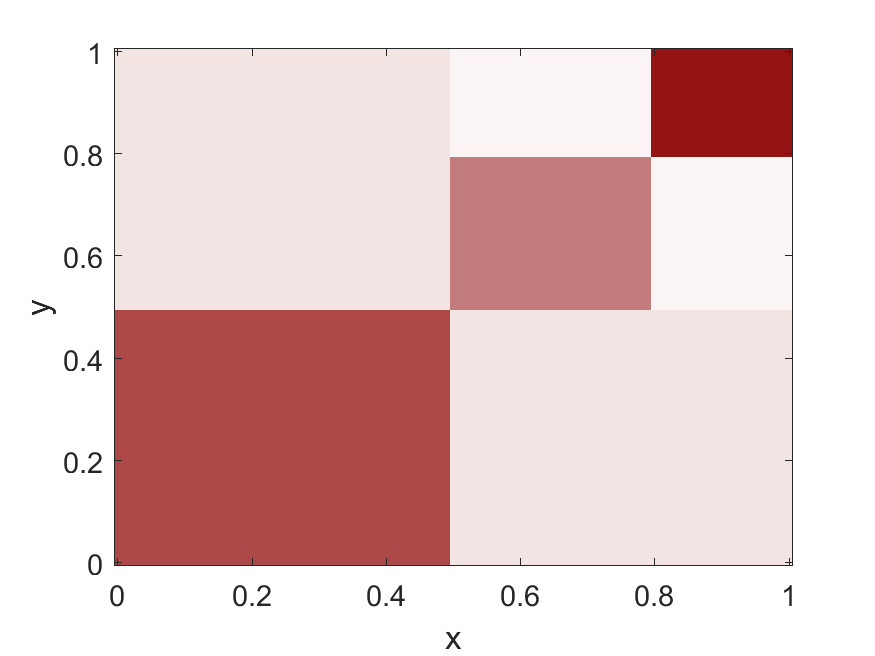}}
\subfigure[Function $\eta$]{\includegraphics[width=.38\textwidth]{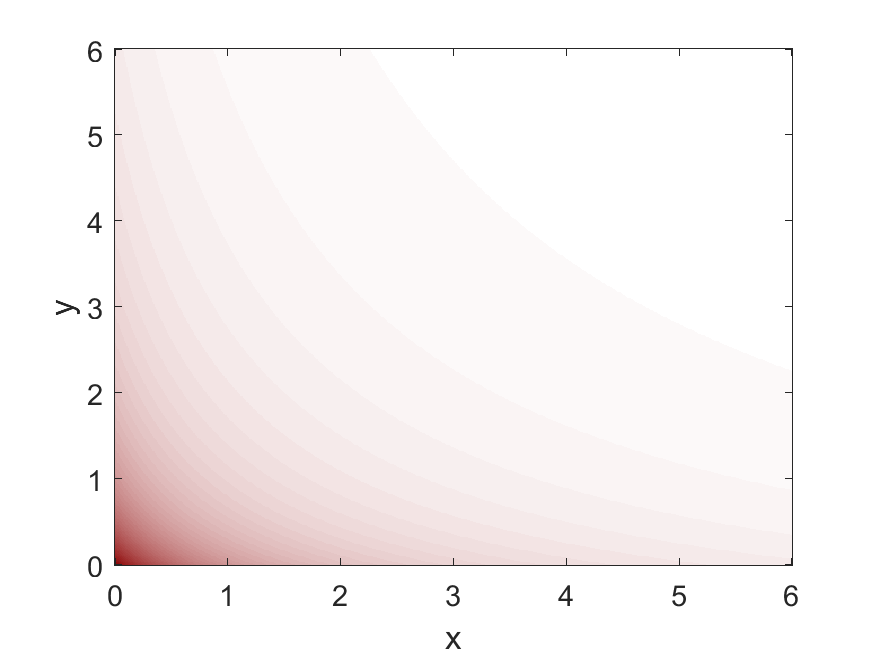}}
\subfigure[Sampled graph]{\includegraphics[width=.38\textwidth]{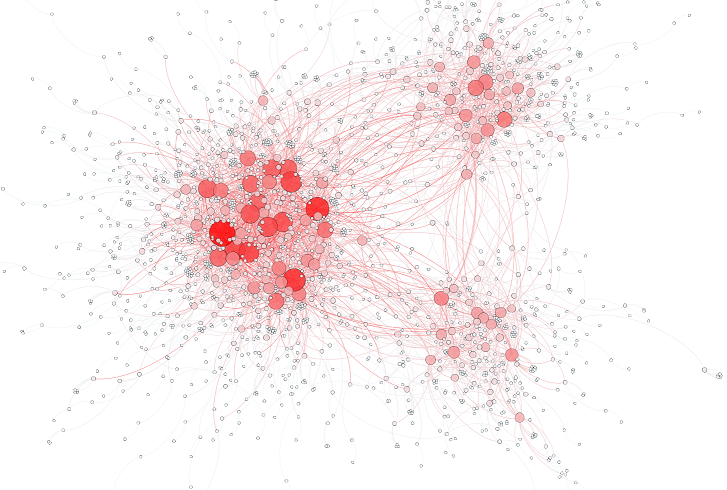}}
\subfigure[Empirical degree distribution of the sampled graph]{\includegraphics[width=.38\textwidth]{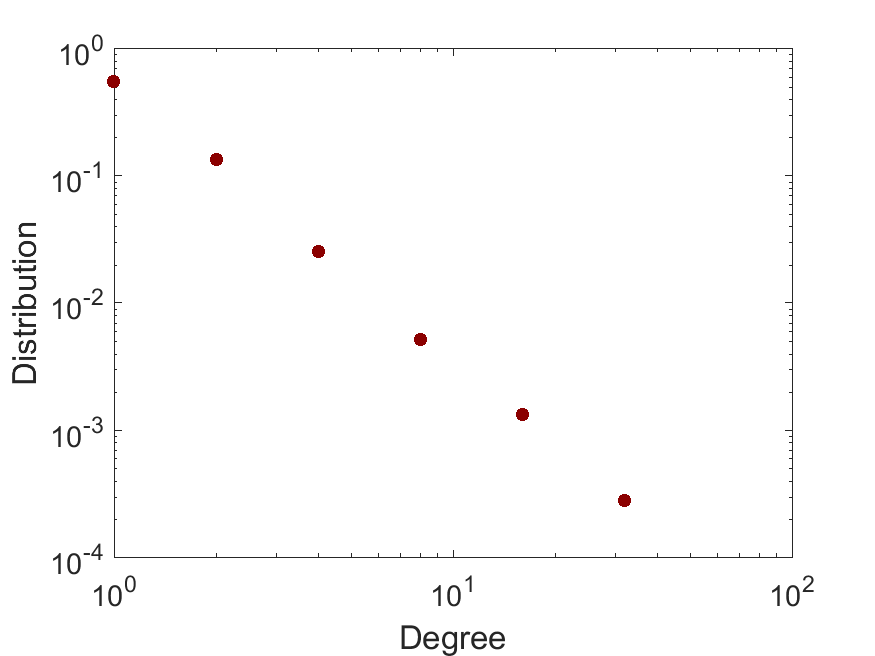}}
\end{center}
\caption{Illustration of a sparse stochastic block-model with 3 communities. (a) Function $\omega$, that controls the local community structure. A darker color represents a higher value. (b) Function $\eta$, that controls the sparsity. (c) Graph sampled from the sparse stochastic block-model using $\alpha=50$. The size of each node are proportional to its degree. (d) Empirical degree distribution of the sampled graph.}
\label{fig:sparseSBM}
\end{figure}
\end{example}

More generally, one can build on the large literature on (dense) graphon/exchangeable graph models, and combine these models with a function $\eta$ satisfying Assumptions 1 and 2, such as those described in the previous section, in order to sparsify a dense graphon and control its sparsity/power-law properties.

\begin{remark}
We can also obtain asymptotic results for those functions $W$ that do not satisfy the separability condition~\eqref{eq:separableW}. Let $\mu(u,v)=\int_{\mathbb R_+\times F} W((u,v),(u',v'))du'dv'$.
Assume that, for each fixed $v$, there exists $u_0(v)>0$ such that for $u>u_0$
\begin{equation}
C_3\tilde\mu_\eta(u)\tilde\mu_\omega(v)\leq \mu(u,v)\leq C_4\tilde\mu_\eta(u)\tilde\mu_\omega(v)\label{eq:inequality}
\end{equation}
where $\tilde\mu_\omega:F\rightarrow \mathbb R_+$, $\tilde\mu_\eta:\mathbb R_+ \rightarrow \mathbb R_+$ with $\tilde\mu_\eta(u)=\int_0^\infty \tilde\eta(u,u')du'$ for some positive function $\tilde\eta$, and $C_3>0$ and $C_4>0$. Assume that $\tilde\mu_\eta$ and $\tilde\nu_\eta$ verify Assumptions 1 and 2. Then the results of Theorems \ref{th:meannumbernodes} and \ref{th:asnumbernodes}, Corollary \ref{th:sparsity1} hold up to a constant. For example, we have for $\sigma\in[0,1]$,
$
\Nedges\asymp  N_\alpha^{2/(1+\sigma)}\ell_\sigma^*(N_\alpha)
$
almost surely as $\alpha$ tends to infinity. In particular, the inequality from \eqref{eq:inequality} is satisfied if
\begin{equation}
W((u_i,v_i),(u_j,v_j))=1-e^{-\tilde \omega(v_i,v_j)\tilde\eta(u_i,u_j)}.\label{eq:modelherlautodes}
\end{equation}
The models developed by \cite{Herlau2016} and \cite{Todeschini2016} for capturing (overlapping) communities fit in this framework. Ignoring loops, both models can be written under the form given by Equation \eqref{eq:modelherlautodes} with
$
\tilde\eta(u,u')=2\overline\rho^{-1}(u)\overline\rho^{-1}(u')
$,
where $\rho$ is a L\'evy measure on $(0,+\infty)$ and $\overline{\rho}(x)=\int_{x}^{\infty}\rho(dw)$ is the tail L\'evy intensity with generalised inverse $\overline{\rho}^{-1}(x)$. When $\tilde \omega$ is given by Equation \eqref{eq:SBM}, it corresponds to the (dense) stochastic blockmodel graphon  of \cite{Herlau2016} and if $\tilde \omega(v_i,v_j)=v_i^{T}v_j$ with $v_i\in \mathbb R_+^p$, it corresponds to the model of \cite{Todeschini2016}. For instance, let $\rho$ be the mean measure from Equation \eqref{eq:GGP} with parameters  $\tau_0>0$ and $\sigma_0\in(-\infty,1)$. Then for $\sigma_0\in(0,1)$, the corresponding sparse regime with power-law for this graph is given by
\begin{align*}
\Nedges\asymp  N_\alpha^{2/(1+\sigma_0)}, \quad
\frac{C_3}{C_4}\frac{\sigma_0 \Gamma(j-\sigma_0)}{j!\Gamma(1-\sigma_0)}\leq \lim_{\alpha\rightarrow\infty}\frac{N_{\alpha,j}}{N_\alpha}\leq  \frac{C_4}{C_3}\frac{\sigma_0 \Gamma(j-\sigma_0)}{j!\Gamma(1-\sigma_0)}, \quad j \geq 1
\end{align*}
For $\sigma_0=0$ (sparse, almost dense regime)
$
\Nedges\asymp  N^2_\alpha/\log(N_\alpha)^2$ and $
N_{\alpha,j}/N_\alpha\rightarrow  0, j\geq 1;
$
for $\sigma_0<0$ (dense regime)
$
\Nedges\asymp  N^2_\alpha$ and
$N_{\alpha,j}/N_\alpha\rightarrow  0,  j \geq 1
$
almost surely as $\alpha$ tends to infinity.
\end{remark}

\subsection{Proof of Proposition \ref{prop:local-global}}
\label{sec:prooflocalglobal}

The proofs of Proposition \ref{th:numberedges} and Theorems \ref{th:meannumbernodes} and \ref{th:asnumbernodes} hold with $x$ replaced by $(u,v) \in \mathbb R_+ \times F$,  $dx = du G(dv)$ and
$\mu(x) = \mu_\eta(u)\mu_\omega(v)$.
We thus need only prove that if $\eta$ verifies Assumptions \ref{assumpt:1} and \ref{assumpt:2} then Lemmas \ref{lemma:abelianvariations}, \ref{lemma:abelianvariationsv2} and \ref{lemma:abelianvariationsv3} in Appendix hold.
Recall that $\mu(x) = \mu_\eta (u) \mu_\omega(v) $, for $x=(u,v)$. Then for all $v$ such that $\mu_\omega(v) >0$ we  apply Lemma \ref{lemma:abelianvariations} to
$$g_0(t)  = \int_0^\infty (1 - e^{-t \mu_\eta(u)})du, \quad g_r(t)  = \int_0^\infty \mu_\eta(u)^r e^{-t \mu_\eta(u)}du, \quad t = \alpha \mu_\omega(v).$$
This leads to, for all $v$ such that $\mu_\omega(v) >0$
\begin{equation*}
\begin{split}
\int_0^\infty (1 - e^{-\alpha \mu_\omega(v) \mu_\eta(u)})du &= \Gamma(1-\sigma) \alpha^\sigma \ell(\alpha) \mu_\omega(v)^\sigma\frac{ \ell\{\alpha \mu_\omega(v) \} }{ \ell(\alpha) } \{1 + o(1) \} \\
& = \Gamma(1-\sigma) \alpha^\sigma \ell(\alpha) \mu_\omega(v)^\sigma \{ 1 + o(1)\}.
\end{split}
\end{equation*}
To prove that there is convergence in $L_1(G)$, note that if $\mu_\omega(v) >0$ and since $\mu_\omega \leq 1$,
\begin{equation*}
\begin{split}
\int_0^\infty (1 - e^{-\alpha \mu_\omega(v) \mu_\eta(u)})du& =  \int_0^\infty  \mu_\eta^{-1}\left\{\frac{ z}{ \alpha \mu_\omega(v) }\right\}e^{-z} dz \leq  \int_0^\infty \mu_\eta^{-1}\left(\frac{ z}{ \alpha  }\right)e^{-z}dz.
\end{split}
\end{equation*}
Moreover
 $$\sup_{\alpha \geq 1} \frac{1}{\alpha^{\sigma}\ell(\alpha)}  \int_0^\infty \mu_\eta^{-1}\left(\frac{ z}{ \alpha  }\right)e^{-z}dz < +\infty , $$
thus the Lebesgue dominated convergence theorem implies
 $$\int_F \int_0^\infty (1 - e^{-\alpha \mu_\omega(v) \mu_\eta(u)})duG(dv) \sim \Gamma(1-\sigma) \alpha^\sigma \ell(\alpha) \int_F\mu_\omega(v)^\sigma G(dv)$$
 when $\sigma <1$ and when $\sigma=1$,
 $$\int_F \int_0^\infty (1 - e^{-\alpha \mu_\omega(v) \mu_\eta(u)})duG(dv) \sim \alpha\ell_1(\alpha) \int_F\mu_\omega(v) G(dv). $$
 The same reasoning is applied to the integrals
 $$\int_F \mu_\omega(v)^r \int_0^\infty \mu_\eta(u)^re^{-\alpha \mu_\omega(v) \mu_\eta(u)}duG(dv).$$
To verify Lemma \ref{lemma:abelianvariationsv2}, note that
\begin{equation*}
\begin{split}
h_0(\alpha) &= \int_F \omega(v,v) \int_0^\infty \eta(u,u)( 1 - e^{-\alpha \mu_\omega(v)\mu_\eta(u) })du G(dv), \\
 h_r(\alpha) &= \int_F \omega(v,v)\mu_\omega(v)^r  \int_0^\infty \eta(u,u)\mu_\eta(u)^r e^{-\alpha \mu_\omega(v)\mu_\eta(u) }du G(dv)
 \end{split}
\end{equation*}
so that the  Lebesgue dominated convergence Theorem also leads to
$$h_0(\alpha ) \sim \int_F \omega(v,v) \int_0^\infty \eta(u,u)du G(dv), \quad h_r(\alpha ) = o(\alpha^{-r} )$$
and the control of the integrals $\int_{\mathbb R_+\times F} \{t\mu(u,v)\}e^{-t \mu(u,v)}duG(dv )$ as in Lemma \ref{lemma:abelianvariationsv3}.

\section{Conclusion}

In this article, we derived a number of properties of graphs based on exchangeable random measures. We relate the sparsity and power-law properties of the graphs to the regular variation properties of the marginal graphon function, identifying four different regimes, from dense to almost extremely sparse.
We derived asymptotic results for the global and local clustering coefficients.
We derived a central limit theorem for the number of nodes $N_\alpha$ in the sparse and dense regimes, and for the number of nodes of degree greater than $j$ in the dense regime. We conjecture that a CLT also holds for $N_{\alpha,j}$ in the sparse regime, under assumptions similar to Assumptions \ref{assumpt:4} and \ref{assumpt:5}, and that a (lengthy) proof similar to that of Theorem~\ref{th:cltsparse} could be used. We leave this for future work.

\section*{Acknowledgment}
The authors thank Zacharie Naulet for helpful feedback and suggestions on an earlier version of this article.
The project leading to this work has received funding from the European Research Council (ERC) under the European Union's Horizon 2020 research and innovation programme
(grant agreement No 834175). At the start of the project, Francesca Panero was funded by the EPSRC and MRC Centre for Doctoral Training in Statistical Science (grant code EP/L016710/1).

\bibliographystyle{chicago}
\bibliography{bnpnetwork}

\begin{thebibliography}{}

\bibitem[\protect\citeauthoryear{Aldous}{Aldous}{1981}]{Aldous1981}
Aldous, D.~J. (1981).
\newblock Representations for partially exchangeable arrays of random
  variables.
\newblock {\em Journal of Multivariate Analysis\/}~{\em 11\/}(4), 581--598.

\bibitem[\protect\citeauthoryear{Ayed, Lee, and Caron}{Ayed
  et~al.}{2019}]{Ayed2019}
Ayed, F., J.~Lee, and F.~Caron (2019).
\newblock Beyond the {C}hinese restaurant and {P}itman-{Y}or processes:
  Statistical models with double power-law behavior.
\newblock In {\em International Conference on Machine Learning}, pp.\
  395--404.

\bibitem[\protect\citeauthoryear{Ayed, Lee, and Caron}{Ayed
  et~al.}{2020}]{Ayed2020}
Ayed, F., J.~Lee, and F.~Caron (2020).
\newblock The normal-generalised gamma-pareto process: A novel pure-jump
  l$\backslash$'evy process with flexible tail and jump-activity properties.
\newblock {\em arXiv preprint arXiv:2006.10968\/}.

\bibitem[\protect\citeauthoryear{Bickel and Chen}{Bickel and
  Chen}{2009}]{BickelChen2009}
Bickel, P.~J. and A.~Chen (2009).
\newblock A nonparametric view of network models and {N}ewman--{G}irvan and
  other modularities.
\newblock {\em Proceedings of the National Academy of Sciences\/}~{\em
  106\/}(50), 21068--21073.

\bibitem[\protect\citeauthoryear{Bickel, Chen, and Levina}{Bickel
  et~al.}{2011}]{BickelChenLevina2011}
Bickel, P.~J., A.~Chen, and E.~Levina (2011).
\newblock The method of moments and degree distributions for network models.
\newblock {\em The Annals of Statistics\/}~{\em 39\/}(5), 2280--2301.

\bibitem[\protect\citeauthoryear{Bingham, Goldie, and Teugels}{Bingham
  et~al.}{1987}]{Bingham1987}
Bingham, N.~H., C.~M. Goldie, and J.~L. Teugels (1987).
\newblock {\em Regular variation}, Volume~27.
\newblock Cambridge university press.

\bibitem[\protect\citeauthoryear{Bollob{\'a}s, Janson, and
  Riordan}{Bollob{\'a}s et~al.}{2007}]{Bollobas2007}
Bollob{\'a}s, B., S.~Janson, and O.~Riordan (2007).
\newblock The phase transition in inhomogeneous random graphs.
\newblock {\em Random Structures \& Algorithms\/}~{\em 31\/}(1), 3--122.

\bibitem[\protect\citeauthoryear{Bollob{\'a}s and Riordan}{Bollob{\'a}s and
  Riordan}{2009}]{BollobasRiordan2009}
Bollob{\'a}s, B. and O.~Riordan (2009).
\newblock Metrics for sparse graphs.
\newblock In S.~Huczynska, J.~Mitchell, and C.~Roney-Dougal (Eds.), {\em
  Surveys in combinatorics}. arXiv:0708.1919: Cambridge University Press.

\bibitem[\protect\citeauthoryear{Borgs, Chayes, Cohn, and Holden}{Borgs
  et~al.}{2018}]{Borgs2018}
Borgs, C., J.~T. Chayes, H.~Cohn, and N.~Holden (2018).
\newblock Sparse exchangeable graphs and their limits via graphon processes.
\newblock {\em Journal of Machine Learning Research\/}~{\em 18}, 1--71.

\bibitem[\protect\citeauthoryear{Borgs, Chayes, Cohn, and Veitch}{Borgs
  et~al.}{2019}]{Borgs2019a}
Borgs, C., J.~T. Chayes, H.~Cohn, and V.~Veitch (2019).
\newblock Sampling perspectives on sparse exchangeable graphs.
\newblock {\em The Annals of Probability\/}~{\em 47\/}(5), 2754--2800.

\bibitem[\protect\citeauthoryear{Borgs, Chayes, Dhara, and Sen}{Borgs
  et~al.}{2019}]{Borgs2019}
Borgs, C., J.~T. Chayes, S.~Dhara, and S.~Sen (2019).
\newblock Limits of sparse configuration models and beyond: Graphexes and
  multi-graphexes.
\newblock {\em arXiv preprint arXiv:1907.01605\/}.

\bibitem[\protect\citeauthoryear{Caron and Fox}{Caron and
  Fox}{2017}]{Caron2017}
Caron, F. and E.~Fox (2017).
\newblock Sparse graphs using exchangeable random measures.
\newblock {\em Journal of the Royal Statistical Society B\/}~{\em 79}, 1--44.
\newblock Part 5.

\bibitem[\protect\citeauthoryear{Caron, Panero, and Rousseau}{Caron
  et~al.}{2020a}]{CaronRousseau2017:main}
Caron, F., F.~Panero, and J.~Rousseau (2020a).
\newblock On sparsity, power-law and clustering properties of graphs based of
  graphex processes.
\newblock Technical report, University of Oxford.

\bibitem[\protect\citeauthoryear{Caron, Panero, and Rousseau}{Caron
  et~al.}{2020b}]{CaronRousseau2017:supplement}
Caron, F., F.~Panero, and J.~Rousseau (2020b).
\newblock On sparsity, power-law and clustering properties of graphs based of
  graphex processes:supplementary material.
\newblock Technical report, University of Oxford.

\bibitem[\protect\citeauthoryear{Chatterjee}{Chatterjee}{2015}]{Chatterjee2015}
Chatterjee, S. (2015).
\newblock Matrix estimation by universal singular value thresholding.
\newblock {\em The Annals of Statistics\/}~{\em 43\/}(1), 177--214.

\bibitem[\protect\citeauthoryear{Diaconis and Janson}{Diaconis and
  Janson}{2008}]{Diaconis2008}
Diaconis, P. and S.~Janson (2008).
\newblock Graph limits and exchangeable random graphs.
\newblock {\em Rendiconti di Matematica e delle sue Applicazioni. Serie VII\/},
  33--61.

\bibitem[\protect\citeauthoryear{Gao, Lu, and Zhou}{Gao et~al.}{2015}]{Gao2015}
Gao, C., Y.~Lu, and H.~Zhou (2015).
\newblock Rate-optimal graphon estimation.
\newblock {\em The Annals of Statistics\/}~{\em 43\/}(6), 2624--2652.

\bibitem[\protect\citeauthoryear{Gnedin, Hansen, and Pitman}{Gnedin
  et~al.}{2007}]{Gnedin2007}
Gnedin, A., B.~Hansen, and J.~Pitman (2007).
\newblock Notes on the occupancy problem with infinitely many boxes: general
  asymptotics and power laws.
\newblock {\em Probability Surveys\/}~{\em 4\/}(146-171), 88.

\bibitem[\protect\citeauthoryear{Herlau, Schmidt, and M{\o}rup}{Herlau
  et~al.}{2016}]{Herlau2016}
Herlau, T., M.~N. Schmidt, and M.~M{\o}rup (2016).
\newblock Completely random measures for modelling block-structured sparse
  networks.
\newblock In {\em Advances in Neural Information Processing Systems 29 (NIPS
  2016)}.

\bibitem[\protect\citeauthoryear{Hoover}{Hoover}{1979}]{Hoover1979}
Hoover, D.~N. (1979).
\newblock Relations on probability spaces and arrays of random variables.
\newblock {\em Preprint, Institute for Advanced Study, Princeton, NJ\/}.

\bibitem[\protect\citeauthoryear{Janson}{Janson}{2016}]{Janson2016}
Janson, S. (2016).
\newblock Graphons and cut metric on sigma-finite measure spaces.
\newblock {\em arXiv:1608.01833\/}.

\bibitem[\protect\citeauthoryear{Janson}{Janson}{2017}]{Janson2017a}
Janson, S. (2017).
\newblock On convergence for graphexes.
\newblock {\em arXiv preprint arXiv:1702.06389\/}.

\bibitem[\protect\citeauthoryear{Kallenberg}{Kallenberg}{1990}]{Kallenberg1990}
Kallenberg, O. (1990).
\newblock Exchangeable random measures in the plane.
\newblock {\em Journal of Theoretical Probability\/}~{\em 3\/}(1), 81--136.

\bibitem[\protect\citeauthoryear{Kolaczyk}{Kolaczyk}{2009}]{Kolaczyk2009}
Kolaczyk, E.~D. (2009).
\newblock {\em Statistical Analysis of Network Data. Methods and models.}
\newblock Springer.

\bibitem[\protect\citeauthoryear{Last, Peccati, and Schulte}{Last
  et~al.}{2016}]{Last2016}
Last, G., G.~Peccati, and M.~Schulte (2016).
\newblock Normal approximation on {P}oisson spaces: {M}ehler's formula, second
  order {P}oincar{\'e} inequalities and stabilization.
\newblock {\em Probability theory and related fields\/}~{\em 165\/}(3-4),
  667--723.

\bibitem[\protect\citeauthoryear{Latouche and Robin}{Latouche and
  Robin}{2016}]{Latouche2016}
Latouche, P. and S.~Robin (2016).
\newblock Variational {B}ayes model averaging for graphon functions and motif
  frequencies inference in {W}-graph models.
\newblock {\em Statistics and Computing\/}~{\em 26\/}(6), 1173--1185.

\bibitem[\protect\citeauthoryear{Lloyd, Orbanz, Ghahramani, and Roy}{Lloyd
  et~al.}{2012}]{Lloyd2012}
Lloyd, J., P.~Orbanz, Z.~Ghahramani, and D.~Roy (2012).
\newblock Random function priors for exchangeable arrays with applications to
  graphs and relational data.
\newblock In {\em Advances in Neural Information Processing Systems 25 (NIPS
  2012)}.

\bibitem[\protect\citeauthoryear{Lo\`eve}{Lo\`eve}{1977}]{Loeve1977}
Lo\`eve, M. (1977).
\newblock {\em Probability Theory I (4th ed.)\/} (4th ed. ed.).
\newblock New York: Springer-Verlag.

\bibitem[\protect\citeauthoryear{Lov{\'a}sz and Szegedy}{Lov{\'a}sz and
  Szegedy}{2006}]{Lovasz2006}
Lov{\'a}sz, L. and B.~Szegedy (2006).
\newblock Limits of dense graph sequences.
\newblock {\em Journal of Combinatorial Theory, Series B\/}~{\em 96\/}(6),
  933--957.

\bibitem[\protect\citeauthoryear{Naulet, Sharma, Veitch, and Roy}{Naulet
  et~al.}{2017}]{Naulet2017}
Naulet, Z., E.~Sharma, V.~Veitch, and D.~M. Roy (2017).
\newblock An estimator for the tail-index of graphex processes.
\newblock {\em arXiv preprint arXiv:1712.01745\/}.

\bibitem[\protect\citeauthoryear{Newman}{Newman}{2010}]{Newman2010}
Newman, M. E.~J. (2010).
\newblock {\em Networks: An Introduction}.
\newblock Oxford University Press.

\bibitem[\protect\citeauthoryear{Nowicki and Snijders}{Nowicki and
  Snijders}{2001}]{Nowicki2001}
Nowicki, K. and T.~Snijders (2001).
\newblock Estimation and prediction for stochastic blockstructures.
\newblock {\em Journal of the American Statistical Association\/}~{\em
  96\/}(455), 1077--1087.

\bibitem[\protect\citeauthoryear{Orbanz and Roy}{Orbanz and
  Roy}{2015}]{Orbanz2015}
Orbanz, P. and D.~M. Roy (2015).
\newblock Bayesian models of graphs, arrays and other exchangeable random
  structures.
\newblock {\em IEEE Transactions on Pattern Analysis and Machine
  Intelligence\/}~{\em 37\/}(2), 437--461.

\bibitem[\protect\citeauthoryear{Palla, Lov{\'a}sz, and Vicsek}{Palla
  et~al.}{2010}]{Palla2010}
Palla, G., L.~Lov{\'a}sz, and T.~Vicsek (2010).
\newblock Multifractal network generator.
\newblock {\em Proceedings of the National Academy of Sciences\/}~{\em
  107\/}(17), 7640--7645.

\bibitem[\protect\citeauthoryear{Penrose}{Penrose}{2003}]{Penrose2003}
Penrose, M. (2003).
\newblock {\em Random geometric graphs}, Volume~5.
\newblock Oxford University Press.

\bibitem[\protect\citeauthoryear{Reitzner and Schulte}{Reitzner and
  Schulte}{2013}]{Reitzner2013}
Reitzner, M. and M.~Schulte (2013, 11).
\newblock Central limit theorems for {$U$}-statistics of {P}oisson point
  processes.
\newblock {\em Ann. Probab.\/}~{\em 41\/}(6), 3879--3909.

\bibitem[\protect\citeauthoryear{Resnick}{Resnick}{1987}]{Resnick1987}
Resnick, S. (1987).
\newblock {\em Extreme values, point processes and regular variation}.
\newblock Springer-Verlag, New York.

\bibitem[\protect\citeauthoryear{Todeschini, Miscouridou, and Caron}{Todeschini
  et~al.}{2020}]{Todeschini2016}
Todeschini, A., X.~Miscouridou, and F.~Caron (2020).
\newblock Exchangeable random measures for sparse and modular graphs with
  overlapping communities.
\newblock {\em Journal of the Royal Statistical Society series B\/}.
\newblock to appear.

\bibitem[\protect\citeauthoryear{Veitch and Roy}{Veitch and
  Roy}{2015}]{Veitch2015}
Veitch, V. and D.~M. Roy (2015).
\newblock The class of random graphs arising from exchangeable random measures.
\newblock {\em arXiv:1512.03099\/}.

\bibitem[\protect\citeauthoryear{Veitch and Roy}{Veitch and
  Roy}{2019}]{Veitch2016}
Veitch, V. and D.~M. Roy (2019).
\newblock Sampling and estimation for (sparse) exchangeable graphs.
\newblock {\em Annals of Statistics\/}~{\em 47\/}(6), 3274--3299.

\bibitem[\protect\citeauthoryear{Willmot}{Willmot}{1990}]{Willmot1990}
Willmot, G.~E. (1990).
\newblock Asymptotic tail behaviour of poisson mixtures by applications.
\newblock {\em Advances in Applied Probability\/}~{\em 22\/}(1), 147--159.

\bibitem[\protect\citeauthoryear{Wolfe and Olhede}{Wolfe and
  Olhede}{2013}]{Wolfe2013}
Wolfe, P.~J. and S.~C. Olhede (2013).
\newblock Nonparametric graphon estimation.
\newblock {\em ArXiv preprint arXiv:1309.5936\/}.

\end{thebibliography}



\setcounter{section}{0}
\renewcommand{\thesection}{\Alph{section}}
\counterwithin{theorem}{section}
\counterwithin{theorem}{section}

\newpage
\begin{appendices}
\section{Proofs of Theorem~\ref{th:meannumbernodes} and Proposition \ref{prop:localclustering}}
\label{sec:proofmain}

Let $g_{\alpha,x}(\theta,\vartheta)$ be defined, for any $\alpha,x,\theta,\vartheta>0$, by
\begin{equation}\label{eq:defg}
g_{\alpha,x}(\theta,\vartheta)=-\log\{1-W(x,\vartheta)\}\1{\theta\leq\alpha}.
\end{equation}

\subsection{Proof of Theorem \ref{th:meannumbernodes}}%
\label{sec:proofmeannodesedges}

The mean number of nodes is \cite[Theorem 5.4]{Veitch2015}
\begin{align*}
E(N_{\alpha} )=\alpha\int_{\mathbb{R}_{+}}\{1-e^{-\alpha\mu(x)}\}dx+\alpha\int
_{\mathbb{R}_{+}}W(x,x)e^{-\alpha\mu(x)}dx.
\end{align*}
By the Lebesgue dominated convergence, we have
$\alpha\int_{\mathbb{R}_{+}}W(x,x)e^{-\alpha\mu(x)}dx=o(\alpha)$.
We have, using Lemma \ref{lemma:abelianvariations}, for $\sigma\in[0,1)$, as $\alpha $ goes to infinity
$
\int_{\mathbb{R}_{+}}(1-e^{-\alpha\mu(x)})dx   \sim\alpha^{\sigma}\ell(\alpha)\Gamma(1-\sigma),
$
and for $\sigma=1$,
$
\int_{\mathbb{R}_{+}}\{1-e^{-\alpha\mu(x)}\}dx   \sim\alpha\ell_1(\alpha).
$ It follows that, as $\alpha $ goes to infinity
\[
E(N_{\alpha})\sim\left \{
\begin{array}{ll}
  \alpha^{\sigma+1}\ell(\alpha)\Gamma(1-\sigma) & \text{if }\sigma\in[0,1) \\
  \alpha^2\ell_1(\alpha) & \text{if }\sigma=1
\end{array}\right .
.
\]

The mean number of nodes of degree $j$ is \cite[Theorem 5.5]{Veitch2015}
\begin{equation}
\begin{split}
E(N_{\alpha,j})   &=\frac{\alpha^{j+1}}{j!}\int_{\mathbb{R}_{+}}(1-W(\vartheta,\vartheta))e^{-\alpha\mu
(\vartheta)}\mu(\vartheta)^{j}d\vartheta  +\frac{\alpha^{j}}{j-1!}\int_{\mathbb{R}_{+}}e^{-\alpha\mu
(\vartheta)}W(\vartheta,\vartheta)\mu(\vartheta)^{j-1} d\vartheta
\end{split}\label{eq:ExpNalphaj}
\end{equation}
 Lemma~\ref{lemma:abelianvariationsv2}, implies that
$$
-\frac{\alpha^{j+1}}{j!}\int_{\mathbb{R}_{+}} W(\vartheta,\vartheta)e^{-\alpha\mu
(\vartheta)}\mu(\vartheta)^{j}d\vartheta
 +\frac{\alpha^{j}}{j-1!}\int_{\mathbb{R}_{+}}e^{-\alpha\mu
(\vartheta)}W(\vartheta,\vartheta)\mu(\vartheta)^{j-1} d\vartheta =o(\alpha)
$$
and from Lemma~\ref{lemma:abelianvariations}, we have, when $\sigma\in[0,1)$
$$
\frac{\alpha^{j+1}}{j!}\int_{\mathbb{R}_{+}}e^{-\alpha\mu
(\vartheta)}\mu(\vartheta)^{j}d\vartheta \sim \frac{\sigma\Gamma(j-\sigma)}{j!}\alpha^{1+\sigma}\ell(\alpha).
$$
If $\sigma=1$,  from Lemma~\ref{lemma:abelianvariations} then $\alpha^2\int_{\mathbb{R}_{+}}e^{-\alpha\mu
(\vartheta)}\mu(\vartheta)d\vartheta \sim \alpha^{2}\ell_1(\alpha)$ and for $j \geq 2$
$$
  \frac{\alpha^{j+1}}{j!}\int_{\mathbb{R}_{+}}e^{-\alpha\mu
(\vartheta)}\mu(\vartheta)^{j}d\vartheta\sim\frac{1}{j(j-1)}\alpha^{2}\ell(\alpha). 
$$
We finally obtain, for $\sigma\in[0,1)$
$
E(N_{\alpha,j})\sim\frac{\sigma\Gamma(j-\sigma)}{j!}\alpha^{1+\sigma
}\ell(\alpha),
$
and for $\sigma=1$,
$
  E(N_{\alpha,1}) \sim \alpha^{2}\ell_1(\alpha)$, and $   E(N_{\alpha,j}) \sim \alpha^{2}/\{j(j-1)\}\ell(\alpha)$, for $ j\geq 2.
$

\subsection{Proof of Proposition \ref{prop:localclustering}}
\label{sec:proofclustering}

For $j\geq 1$, define
\begin{equation}
R_{\alpha j}=\sum_{i}T_{\alpha i}\1{D_{\alpha i=j}}.
\end{equation}
$R_{\alpha j}$ corresponds to the number of triangles having a node of degree $j$ as a vertex, where triangles having $k\leq 3$ degree-$j$ nodes as vertices are counted $k$ times. We therefore have
$$
C_{\alpha,j}^{(\ell)}=\frac{2}{j(j-1)}\frac{R_{\alpha j}}{N_{\alpha,j}}.
$$

The proof for the asymptotic behaviour of the local clustering coefficients $C_{\alpha,j}^{(\ell)}$ is organised as follows. We first derive a convergence result for $E(R_{\alpha j})$. This result is then extended to an almost sure result. The extension requires some additional work as $R_{\alpha j}$ is not monotone, and $\sum_{j\geq k} R_{\alpha k}$ is monotone but not of the same order as $R_{\alpha j}$, hence a proof similar to that for $N_{\alpha j}$ (see Section~\ref{sec:app:asnumbernodes}) cannot be used. The almost sure convergence results for $C_{\alpha,j}^{(\ell)}$ and $\overline C_{\alpha}^{(\ell)}$ then follow from the almost sure convergence result for $R_{\alpha j}$.\medskip

We have%
\[
R_{\alpha j}=\sum_{i}T_{\alpha i}\1{D_{\alpha i=j}}=\frac{1}{2}\sum_{i\neq l\neq k}%
Z_{il}Z_{ik}Z_{lk}\1{\sum_{s}Z_{is}=j \1{\theta_{s}\leq\alpha}}\1{\theta
_{i}\leq\alpha}\1{\theta_{l}\leq\alpha}\1{\theta_{k}\leq\alpha}%
\]
and%
\begin{align*}
  E\left(  R_{\alpha j}\mid M\right) &  =\frac{1}{2}\sum_{i\neq l\neq k}W(\vartheta_{i},\vartheta_{l}%
)W(\vartheta_{i},\vartheta_{k})W(\vartheta_{l},\vartheta_{k})\frac{1}{\left(
j-2\right)  !}\\
&\qquad\qquad\times\sum_{i_{1}\neq i_{2}\ldots\neq i_{j-2}\neq l\neq k}\left[
\prod_{s=1}^{j-2}W(\vartheta_{i},\vartheta_{s})\right]  e^{-\sum_{s\neq
l,k,i_{1},\ldots,i_{j-2}}g_{\alpha,\vartheta_{i}}(\theta_{s},\vartheta_{s})}\\
&  =\frac{1}{2\left(  j-2\right)  !}\sum_{i\neq l\neq k\neq i_{1}\neq
i_{2}\ldots\neq i_{j-2}}W(\vartheta_{i},\vartheta_{l})W(\vartheta
_{i},\vartheta_{k})W(\vartheta_{l},\vartheta_{k})(1-W(\vartheta_i,\vartheta_i))\\
&\qquad\qquad\qquad\qquad\times\left[  \prod_{s=1}%
^{j-2}W(\vartheta_{i},\vartheta_{s})\right]  e^{-\sum_{s\neq l,k,i_{1}%
,\ldots,i_{j-2}}g_{\alpha,\vartheta_{i}}(\theta_{s},\vartheta_{s})}\\
& \qquad +\frac{1}{2(j-3)!}\sum_{i\neq l\neq k\neq i_{1}\neq i_{2}\ldots\neq
i_{j-3}}W(\vartheta_{i},\vartheta_{i})W(\vartheta_{i},\vartheta_{l}%
)W(\vartheta_{i},\vartheta_{k})W(\vartheta_{l},\vartheta_{k})\\
&\qquad\qquad\qquad\qquad\times\prod_{s=1}%
^{j-3}W(\vartheta_{i},\vartheta_{s})e^{-\sum_{s\neq i,l,k,i_{1},\ldots
,i_{j-3}}g_{\alpha,\vartheta_{i}}(\theta_{s},\vartheta_{s})}%
\end{align*}
where $g_{\alpha,x}(\theta,\vartheta)$ is defined in Equation \eqref{eq:defg}.
Applying the Slivnyak-Mecke theorem, we obtain%
\begin{align}
E\left( R_{\alpha j}\right)& =\frac
{\alpha^{j+1}}{2(j-2)!}\int_{\mathbb{R}_{+}^{3}}%
W(x,y)W(x,z)W(y,z)(1-W(x,x))\mu(x)^{j-2}e^{-\alpha\mu(x)}dxdydz\nonumber\\
&  \quad+\frac{\alpha^{j}}{2(j-3)!}\int_{\mathbb{R}_{+}^{3}%
}W(x,y)W(x,z)W(y,z)W(x,x)\mu(x)^{j-3}e^{-\alpha\mu(x)}dxdydz.
\label{eq:Talphaj}
\end{align}
Note that under Assumption \ref{assumpt:1} with $\sigma\in(0,1)$, $\mu(x)>0$ for all $x$. The leading term in the right-handside of Equation~\eqref{eq:Talphaj} is the first term. We have therefore
\begin{align*}
E\left( R_{\alpha j}\right)& \sim \frac
{\alpha^{j+1}}{2(j-2)!}\int_{\mathbb{R}_{+}^{3}}%
L(x)\mu(x)^{j}e^{-\alpha\mu(x)}dxdydz
\end{align*}
where $$L(x)=\frac{(1-W(x,x))\int_{\mathbb{R}_{+}^{2}}W(x,y)W(x,z)W(y,z)dydz}{\mu(x)^{2}}.$$ As $\lim_{x\to\infty}W(x,x)=0$, the condition \eqref{cond:localcluster} implies $\lim_{x\to\infty}L(x)= b$. \smallskip

\paragraph{Case $b>0$.} Assume first that
$b>0$. In this case, $L$ is a slowly varying function by
assumption.\ Therefore, using Lemma~\ref{lemma:abelianvariations2}, we
have, under Assumption \ref{assumpt:1}, for $\sigma\in(0,1)$
\[
\int_{0}^{\infty}L(x)\mu(x)^{j}e^{-\alpha\mu(x)}dx\sim\sigma b\ell
(\alpha)\Gamma(j-\sigma)\alpha^{\sigma-j}.
\]
as $\alpha$ tends to infinity. Hence
\begin{equation}
E\left( R_{\alpha j}\right)  \sim\frac
{b\sigma\Gamma(j-\sigma)}{2(j-2)!}\alpha^{1+\sigma}\ell(\alpha)\label{eq:meanT}
\end{equation}
as $\alpha$ tends to infinity. In order to obtain a convergence in probability, we state the following proposition, whose proof is given in Section~\ref{sec:proofvariancelocalclust} in the Supplementary Material \citep{CaronRousseau2017:supplement} and is similar to that of Proposition~\ref{prop:variancenbnodesj}.
\begin{proposition}
\label{prop:variancelocalclust}
Under  Assumptions \ref{assumpt:1} and \ref{assumpt:2},  with $\sigma\in[0,1]$, slowly varying function $\ell$ and positive scalar $a$ satisfying \eqref{eq:conda},  we have
$$
\var\left(  \sum_{i}T_{\alpha i}\1{D_{\alpha i=j}}\right) =O\{\alpha^{3+2\sigma-2a}\ell_\sigma(\alpha)^2\}\text{ as }\alpha\to\infty,
$$
and for any sequence $\alpha_n$ going to infinity such that $\alpha_{n+1} - \alpha_n = o(\alpha_n)$,
$$
\var\left(  \sum_iT_{\alpha_{n+1} i}\1{D_{\alpha_n i=j}} \1{\sum_{i'}\1{\alpha_n<\theta_{i'}\leq\alpha_{n+1}}Z_{ii'} =1} \right) =O\left(\alpha_n^{3+2\sigma-2a}\ell_\sigma(\alpha_n)^2\right)\text{ as }n\to\infty.
$$
\end{proposition}

We now want to find a subsequence $\alpha_n$ along which the convergence is almost sure. Using Chebyshev's inequality and the first part of Proposition \ref{prop:variancelocalclust}, there exists $n_0\ge 0$ and $C\ge 0$ such that for all $n>n_0$
\begin{align*}
\Pr\left(\left|\frac{R_{\alpha_n j}}{E(R_{\alpha_n j})}-1\right|>\epsilon\right)\le&\frac{C\alpha_n^{3+2\sigma-2a}\ell_\sigma(\alpha_n)^2}{\epsilon^2(\frac
{b\sigma\Gamma(j-\sigma)}{2(j-2)!}\alpha_n^{1+\sigma}\ell(\alpha_n))^2}.
\end{align*}
Now, if Assumption \ref{assumpt:2} is satisfied for a given $a>1/2$, consider the sequence
\begin{equation}\label{eq:alphan}
\alpha_n  = (n \log^2 n)^{1/(2a-1)}
\end{equation}
so that $\sum_n \alpha_n^{1-2a} < +\infty$ and
\begin{align*}
\sum_n\Pr\left(\left|\frac{R_{\alpha_n j}}{E(R_{\alpha_n j})}-1\right|>\epsilon\right)<\infty.
\end{align*}
Therefore, using Borel-Cantelli's lemma we have
\[
R_{\alpha_n j} \sim\frac
{b\sigma\Gamma(j-\sigma)}{2(j-2)!}\alpha_n^{1+\sigma}\ell(\alpha_n)
\]
almost surely as $n\to\infty$.

The goal is now to extend this result to $R_{\alpha j}$, by sandwiching. Let $I_\alpha := \{  i : \, \theta_i \leq \alpha\}$. We have the following upper and lower bounds for $R_{\alpha j}$
\begin{equation}\label{eq:sandwich}
\sum_{i\in I_{\alpha_n} }T_{\alpha_n i}\1{D_{\alpha i=j}} \leq \sum_{i\in I_{\alpha} }T_{\alpha i}\1{D_{\alpha i}=j} \leq \sum_{i\in I_{\alpha_{n+1}} }T_{\alpha_{n+1} i}\1{D_{\alpha i}=j}.
\end{equation}
Considering the upper bound of \eqref{eq:sandwich}:
\begin{align}\label{eq:upper_bound_T}
 \sum_{i\in I_{\alpha_{n+1}} }T_{\alpha_{n+1} i}\1{D_{\alpha i}=j} &\leq  \sum_{i\in I_{\alpha_{n+1}} }T_{\alpha_{n+1} i}\1{D_{\alpha_{n+1} i}=j} +  \sum_{i\in I_{\alpha_{n+1}} }T_{\alpha_{n+1} i}\1{D_{\alpha i}=j}\1{D_{\alpha_{n+1} i}>j}\notag\\
 &\le  R_{\alpha_{n+1}j}  +  \widetilde R_{nj}
 \end{align}
 where
 \begin{equation}\label{eq:Rnj}
\widetilde R_{nj}=\sum_{i\in I_{\alpha_{n+1}} }T_{\alpha_{n+1} i}\1{D_{\alpha_n i}\leq j} \1{ \sum_{i'} \1{\alpha_n < \theta_{i'}\leq \alpha_{n+1}}Z_{ii'}\geq 1}.
 \end{equation}
We can bound the lower bound of \eqref{eq:sandwich} by
  \begin{align}\label{eq:lower_bound_T}
 \sum_{i\in I_{\alpha_n} }T_{\alpha_n i}\1{D_{\alpha i=j}} &\ge \sum_{i\in I_{\alpha_n} }T_{\alpha_n i}\1{D_{\alpha_n i=j}}\1{D_{\alpha i=j}}\notag \\
 &\ge\sum_{i\in I_{\alpha_n} }T_{\alpha_n i}\1{D_{\alpha_n i=j}} - \sum_{i\in I_{\alpha_n} }T_{\alpha_n i}\1{D_{\alpha_n i=j}}\1{D_{\alpha_{n+1} i>j}}\notag\\
 &\geq \sum_{i\in I_{\alpha_n} }T_{\alpha_n i}\1{D_{\alpha_n i=j}} -\sum_{i\in I_{\alpha_{n+1}} }T_{\alpha_{n+1} i}\1{D_{\alpha_n i}\leq j} \1{ \sum_{i'} \1{\alpha_n < \theta_{i'}\leq \alpha_{n+1}}Z_{ii'}\geq 1}\notag\\
 &= R_{\alpha_n j} - \widetilde R_{nj}.
\end{align}
The following Lemma, proved in Section~\ref{sec:app:asymptT} of the Supplementary Material \citep{CaronRousseau2017:supplement}, provides an asymptotic bound for the remainder term $\widetilde R_{nj}$.
\begin{lemma}\label{lemma:asymptT}
Let $\widetilde R_{nj}$ be defined as in Equation~\eqref{eq:Rnj}. If Assumptions \ref{assumpt:1} and \ref{assumpt:2} hold with $\sigma\in(0,1)$ and slowly varying function $\ell$, and condition~\eqref{cond:localcluster} is satisfied with $b>0$, we have
$$\widetilde R_{nj} = o(\alpha_n^{1+\sigma}\ell(\alpha_n))$$ almost surely as $\alpha$ tends to infinity.
\end{lemma}
Combining Lemma~\ref{lemma:asymptT} with the inequalities~\eqref{eq:sandwich},~\eqref{eq:upper_bound_T} and \eqref{eq:lower_bound_T}, and the fact that $R_{\alpha_n j}\sim R_{\alpha_{n+1} j}\asymp \alpha_n^{1+\sigma}\ell(\alpha_n)$ almost surely as $n\to\infty$, we obtain by sandwiching
$$
R_{\alpha j}\sim\frac
{b\sigma\Gamma(j-\sigma)}{2(j-2)!}\alpha^{1+\sigma}\ell(\alpha)\text{ almost surely as $\alpha$ tends to infinity.}
$$
 Recalling that $N_{\alpha,j}\sim\frac{\sigma\Gamma(j-\sigma
)}{j!}\alpha^{1+\sigma}\ell(\alpha)$ almost surely, we have, for any $j\geq1$
\[
C_{\alpha,j}^{(\ell)}=\frac{2R_{\alpha j}}{j(j-1)N_{\alpha,j}}\rightarrow b\text{ almost surely as $\alpha$ tends to infinity.}
\]
 Finally, as $\frac{N_{\alpha,j}%
}{N_{\alpha}-N_{\alpha,1}}$ converges to a constant $\pi_{j}\in(0,1)$ almost surely for
any $j$, we have, using Toeplitz's lemma
\[
\overline{C}_{\alpha}^{(\ell)}=\frac{1}{N_{\alpha}-N_{\alpha,1}}\sum_{j\geq
2}N_{\alpha,j}C_{\alpha,j}^{(\ell)}\rightarrow b
\]
almost surely as $\alpha$ tends to infinity.

\paragraph{Case $b=0$. } In the case $L(x)\rightarrow0$, Lemma \ref{lemma:abelianvariations2} gives $\int%
_{0}^{\infty}L(x)\mu(x)^{j}e^{-\alpha\mu(x)}dx=o(\alpha^{\sigma-j})$ hence, by Markov inequality
\[
R_{\alpha j} =o(\alpha^{1+\sigma
}\ell(\alpha))
\]
and $C_{\alpha j}^{(\ell)}\rightarrow 0$ in probability as $\alpha$ tends to infinity.

\section{Technical Lemma} \label{sec:techlemma}

The proof of the following lemma follows similarly to the proof of Proposition 2 in \citep{Gnedin2007}, and is omitted here.

\begin{lemma}
\label{lemma:almostsure}Let $(X_{t})_{t\geq0}$ be some positive monotone
increasing stochastic process with finite first moment $(E(X_{t}))_{t\geq0}\in
RV_{\gamma}$ where $\gamma\geq0$ (see Definition \ref{def:RV}). Assume%
\[
\var(X_{t})=O\{ t^{-a}E(X_{t})^{2}\}%
\]
for some $a>0$. Then
\[
\frac{X_{t}}{E(X_{t})}\rightarrow1\text{ almost surely as
}t\rightarrow\infty.
\]

\end{lemma}

The following lemma is a compilation of results from Propositions 17, 18 and 19 in~\cite{Gnedin2007}.
\begin{lemma}
\label{lemma:abelianvariations}Let $\mu:\mathbb{R}_{+}\rightarrow\mathbb{R}_{+}$
be a positive, right-continuous and monotone decreasing function  with $\int_0^\infty \mu(x)dx<\infty$ and generalised inverse $\mu^{-1}(x)=\inf\{ y> 0\mid f(y)\leq x\}$ satisfying
\begin{equation}
\mu^{-1}(x)=x^{-\sigma}\ell(1/x)\label{eq:RVinlemma}%
\end{equation}
where $\sigma\in\lbrack0,1]$ and $\ell$ is a slowly varying function. Consider
$$
g_{0}(t)=\int_{0}^{\infty}(1-e^{-t\mu(x)})dx, \quad g_{r}(t)=\int_{0}^{\infty}e^{-t\mu(x)}\mu(x)^{r}dx.\quad r \geq 1.
$$
Then, for any $\sigma\in\lbrack0,1)$%
\[
g_{0}(t)\sim\Gamma(1-\sigma)t^{\sigma}\ell(t)\text{ as }t\rightarrow\infty
\]
and, for $r\geq 1$,
\[
\left\{
\begin{array}
[c]{ll}%
g_{r}(t)\sim t^{\sigma-r}\ell(t)\sigma\Gamma(r-\sigma) & \text{if }\sigma\in(0,1)\\
g_{r}(t)=o\{t^{\sigma-r}\ell(t)\} & \text{if }\sigma=0
\end{array}
\right.
\]
as $t\rightarrow\infty.$ For $\sigma=1$, as $t\rightarrow\infty$,
\begin{equation*}
g_0(t) \sim t\ell_1(t), \quad
g_1(t)\sim \ell_1(t), \quad g_r(t)\sim t^{1-r}\ell(t)\Gamma(r-1)
\end{equation*}
 where $\ell_1(t)=\int_t^\infty x^{-1} \ell(x)dx$. Note that $\ell(t)=o(\ell_1(t))$ hence $g_r(t)=o\{t^{1-r}\ell_1(t)\}$.
\end{lemma}

\begin{lemma}
\label{lemma:abelianvariationsv2}Let $\mu:\mathbb{R}_{+}\rightarrow\mathbb{R}_{+}$
be a positive, monotone decreasing function, and $u:\mathbb R_+\rightarrow [0,1]$ a positive and integrable function with $\int_0^\infty u(x)dx<\infty$. Consider
$
h_{0}(t)=\int_{0}^{\infty}u(x)(1-e^{-t\mu(x)})dx
$
and for $r\geq 1$
$
h_{r}(t)=\int_{0}^{\infty}u(x)e^{-t\mu(x)}\mu(x)^{r}dx.
$
\newline Then, as $t\rightarrow\infty.$
$$
h_{0}(t)\sim \int_0^\infty u(x)dx, \quad
h_{r}(t)=o(t^{-r}), \quad r\geq 1.
$$

\end{lemma}
\begin{proof}
$h_{0}(t)\rightarrow \int_0^\infty u(x)dx$ by dominated convergence. Using Proposition \ref{prop:RVderivative},
\[
\frac{t h_1(t)}{\int_0^\infty u(x)dx}\rightarrow 0
\]
Proceed by induction for the final result.
\end{proof}

\begin{lemma}\label{lemma:abelianvariationsv3}
Let $\mu $ be a  non-negative, non-increasing function on $\mathbb R_+$, with $\int_0^{\infty}\mu(x)dx<\infty$ and such that its generalised inverse
 $ \mu^{-1} $ verifies $ \mu^{-1}(x)\sim x^{-\sigma}\ell(1/x)$ as $x\rightarrow 0$ with $\sigma \in [0,1]$ and $\ell$ a slowly varying function. Then as $t\rightarrow\infty$, for all $r> \sigma$
$$
 \int_{\mathbb R_+} \mu(x)^{r} e^{-t \mu(x) } dx =O\{ t^{\sigma-r} \ell(t)\}
$$
\end{lemma}

\begin{proof}
Let $r>\sigma$. Let $U(y)=\mu^{-1}(1/y)$. $U$ is non-negative, non-decreasing, with $U(y)\sim y^{\sigma}\ell(y)$ as $y\rightarrow\infty$. Consider the change of variable $x=U(y)$, one obtains
\begin{align*}
\int_0^\infty \{ \mu(x) \}^{r} e^{-t \mu(x) } dx=\int_0^\infty y^{-r} e^{-t/y }dU(y)
\end{align*}

We follow part of the proof in \citep[p.37]{Bingham1987}. Note that $y\rightarrow y^{-r}\exp(-t/y)$ is monotone increasing on $[0,t/r]$ and monotone decreasing on $[t/r,\infty)$.
\begin{align*}
\int_{0}^{\infty}y^{-r}e^{-t/y}dU(y)& =\left\{  \int_{0}^{t/r}+\sum_{n=1}^{\infty}\int_{2^{n-1} t/r}^{2^{n}t/r}\right\}  y^{-r}e^{-t/y}dU(y)\\
& \leq  t^{-r}e^{-r}r^r U(t/r)+t^{-r}r^r\sum_{n=1}^{\infty}2^{-r(n-1)}U\left(2^{n}t/r\right)  \\
& \leq  2t^{\sigma-r}e^{-r}r^r \ell(t/r)+2t^{-r}r^r\sum_{n=1}^{\infty}2^{-r(n-1)}(2^{n}t/r)^\sigma \ell\left(
2^{n}t/r\right)   \\
&\leq  2t^{\sigma-r}e^{-r}r^r \ell(t/r)+2^{r+1}t^{\sigma-r}r^{r-\sigma}\sum_{n=1}^{\infty}2^{-n(r-\sigma)} \ell\left(
2^{n}t/r\right)
\end{align*}

for $t$ large, using the regular variation property of $U$. Using Potter's bound~\citep[Theorem 1.5.6]{Bingham1987}, we have, for any $\delta>0$ and for $t$ large
\[
 \ell(2^{n}t/r)\leq 2\ell(t)\max (1,2^{n\delta}/r^\delta).
\]
Hence, for $t$ large,
\[
\int_{0}^{\infty}y^{-r}e^{-t/y}dU(y)\lesssim t^{\sigma-r} \ell(t) \left (1+\sum_{n=1}^{\infty}2^{-n(r-\sigma)}\max (r^\delta,2^{n\delta})\right )
\]
Taking $0<\delta<\frac{r-\sigma}{2}$, the series in the right handside converges.
\end{proof}

The next lemma is a slight variation of Lemma \ref{lemma:abelianvariations}, with the addition of a slowly varying function in the integrals. Note that the case $\sigma=0$ and $\ell$ tends to a constant is not covered.

\begin{lemma}
\label{lemma:abelianvariations2}Let $f:\mathbb{R}_{+}\rightarrow\mathbb{R}_{+}$
be a positive, right-continuous and monotone decreasing function  with $\int_0^\infty f(x)dx<\infty$ and generalised inverse $f^{-1}(x)=\inf\{ y> 0\mid f(y)\leq x\}$ satisfying
\begin{equation}
f^{-1}(x)=x^{-\sigma}\ell(1/x)\label{eq:RVinlemma}%
\end{equation}
where $\sigma\in[0,1]$ and $\ell$ is a slowly varying function, with $\lim_{t\to\infty}\ell(t)=\infty$ if $\sigma=0$. Consider
$$
\widetilde g_{0}(t)=\int_{0}^{\infty}(1-e^{-tf(x)})L(x)dx
$$
and for $r\geq 1$
$$
\widetilde g_{r}(t)=\int_{0}^{\infty}e^{-tf(x)}f(x)^{r}L(x)dx.
$$
where $L:\mathbb R_+\rightarrow (0,\infty)$ is a locally integrable function with $\lim_{t\to\infty}L(t)= b\in[0,\infty)$.
\newline
 Then, for any $\sigma\in\lbrack0,1)$%

\[
\left\{
\begin{array}
[c]{ll}%
\widetilde g_{0}(t)\sim b\Gamma(1-\sigma)t^{\sigma}\ell(t) & \text{if }b>0\\
\widetilde g_{0}(t)=o(t^{\sigma}\ell(t)) & \text{if }b=0
\end{array}
\right.
\]

and, for $r\geq 1$,
\[
\left\{
\begin{array}
[c]{ll}%
\widetilde g_{r}(t)\sim b t^{\sigma-r}\ell(t)\sigma\Gamma(r-\sigma) & \text{if }\sigma\in(0,1),b>0\\
\widetilde g_{r}(t)=o\{t^{\sigma-r}\ell(t)\} & \text{if }\sigma=0\text{ or }b=0
\end{array}
\right.
\]
as $t\rightarrow\infty.$ For $\sigma=1$, $b>0$,  as $t\rightarrow\infty$,
\begin{equation*}
\widetilde g_0(t) \sim bt\ell_1(t), \quad
\widetilde g_1(t)\sim b\ell_1(t), \quad \widetilde g_r(t)\sim bt^{1-r}\ell(t)\Gamma(r-1)
\end{equation*}
and
 where $\ell_1(t)=\int_t^\infty x^{-1} \ell(x)dx$. Note that $\ell(t)=o(\ell_1(t))$ hence $\widetilde g_r(t)=o\{t^{1-r}\ell_1(t)\}$.

\end{lemma}
\begin{proof}
Let $g_{0}(t)=\int_{0}^{\infty}(1-e^{-tf(x)})dx$. Let $\ell_1(t)=\int_t^\infty x^{-1} \ell(x)dx$ and $\ell_\sigma(t)=\Gamma(1-\sigma)\ell(t)$ if $\sigma\in[0,1)$. Using Lemma \ref{lemma:abelianvariations}, we have $g_{0}(t)\sim t^{\sigma}\ell_\sigma(t)$ as $t\to\infty$, and in particular $g_{0}(t)\to\infty$.
By dominated convergence, for any $x_0>0$ $\int_0^{x_0} (1-e^{-tf(x)})L(x)dx\to \int_0^{x_0} L(x)dx<\infty$ hence
$\widetilde g_{0}(t)\sim \int_{x_0}^{\infty} (1-e^{-tf(x)})L(x)dx$ as $t\to\infty$.\\
Let $\epsilon>0$. There is $x_0$ such that for all $x\geq x_0$, $|L(x)-b|\leq  \epsilon$ and so
$$
(b-\epsilon)\int_{x_0}^{\infty} (1-e^{-tf(x)})dx\leq  \int_{x_0}^{\infty} (1-e^{-tf(x)})L(x)dx \leq (b+\epsilon)\int_{x_0}^{\infty} (1-e^{-tf(x)})dx.
$$
Hence by sandwiching
$$
\lim_{t\to\infty}\frac{\widetilde g_0(t)}{t^{\sigma}\ell_\sigma(t)}=\lim_{t\to\infty}\frac{\int_{x_0}^{\infty} (1-e^{-tf(x)})L(x)dx}{t^{\sigma}\ell_\sigma(t)}\in(b-\epsilon,b+\epsilon).
$$
As this is true for any $\epsilon>0$, we obtain $\widetilde g_{0}(t)\sim bt^{\sigma}\ell_\sigma(t)\text{ as }t\rightarrow\infty$ if $b>0$ and $\widetilde g_{0}(t)=o(t^{\sigma}\ell_\sigma(t))$ if $b=0$ . The asymptotic results for $\widetilde g_r(t)$ then follow from Proposition \ref{prop:RVderivative}.
\end{proof}

\medskip
The following is a corollary of \cite[Theorem 2.1]{Willmot1990}.
\begin{corollary}\cite[Theorem 2.1]{Willmot1990}. \label{cor:willmot}
Assume that $$f(x)\sim \ell(x)x^\alpha e^{-\beta x}$$
where $\ell$ is a slowly varying, locally bounded function on $(0,\infty)$, $\beta\geq 0$ and $\alpha\in \mathbb R$, or $\alpha<-1$ and $\beta=0$. Then, as $n\to\infty$
\begin{equation}
\int_0^\infty \frac{(\lambda x)^ne^{-\lambda x}}{n!}f(x)dx\sim \frac{\ell(n)}{(\lambda+\beta)^{\alpha+1}}\left ( \frac{\lambda}{\lambda+\beta} \right )^n n^\alpha
\label{eq:willmot1}
\end{equation}
and
\begin{equation}
\int_0^\infty \frac{(\lambda x)^ne^{-\lambda x}}{n!}u(x)f(x)dx= o\left ( \frac{\ell(n)}{(\lambda+\beta)^{\alpha+1}}\left ( \frac{\lambda}{\lambda+\beta} \right )^n n^\alpha\right )
\label{eq:willmot2}
\end{equation}
for any locally bounded function $u$ vanishing at infinity.
\end{corollary}
\begin{proof}
Equation \eqref{eq:willmot1} is proved in \cite[Theorem 2.1]{Willmot1990}. For any $x_0>0$, we have $\int_0^\infty \frac{(\lambda x)^ne^{-\lambda x}}{n!}u(x)f(x)dx\sim \int_{x_0}^\infty \frac{(\lambda x)^ne^{-\lambda x}}{n!}u(x)f(x)dx$. For any $\epsilon>0$, there is $x_0$ such that $u(x)<\epsilon$ for all $x>x_0$, hence
$$\int_{x_0}^\infty \frac{(\lambda x)^ne^{-\lambda x}}{n!}u(x)f(x)dx\leq \epsilon \int_0^\infty \frac{(\lambda x)^ne^{-\lambda x}}{n!}f(x)dx$$
and \eqref{eq:willmot2} follows from \eqref{eq:willmot1} by sandwiching.
\end{proof}

The following lemma is useful to bound the variance and for the proof of the central limit theorem.
\begin{lemma}\label{lemma:boundintnu}
Assume the functions $\mu$ and $\nu$ satisfy Assumptions~\ref{assumpt:1} and~\ref{assumpt:2}, for some $\sigma\in[0,1]$, slowly varying function $\ell$ and some $a>\min(1/2,\sigma)$ if $\sigma<1$ and $a=1$ if $\sigma=1$. Then
$$
\int_{\mathbb{R}_{+}^{2}}  \nu(x,y)   e^{-\alpha\mu(x)-\alpha\mu(y)+\alpha\nu(x,y)} dxdy=O\left (\alpha^{2\sigma-2a}\ell_\sigma^2(\alpha)\right)
$$
where $\ell_\sigma$ is defined in Equation~\eqref{eq:defellsigma}. If $a=1$ and $\sigma=0$ we have the stronger result
$$
\int_{\mathbb{R}_{+}^{2}}  \nu(x,y)   e^{-\alpha\mu(x)-\alpha\mu(y)+\alpha\nu(x,y)} dxdy=o\left (\alpha^{-2}\ell^2(\alpha)\right).
$$
\end{lemma}
\begin{proof}
Using $\nu(x,y)\leq \sqrt{\mu(x)\mu(y)}\leq (\mu(x)+\mu(y))/2$ and Assumption \ref{assumpt:2},
\begin{align*}
\int_{\mathbb{R}_{+}^{2}}  &\nu(x,y)   e^{-\alpha\mu(x)-\alpha\mu(y)+\alpha\nu(x,y)} dxdy \leq\int_{\mathbb{R}_{+}^{2}}  \nu(x,y)   e^{-\alpha\mu(x)/2-\alpha\mu(y)/2} dxdy  \\
&\leq C_1 \left(\int_{x_0}^\infty \mu(x)^a e^{-\alpha\mu(x)/2} \right )^2  +2\int_0^{x_0}\int_0^\infty \nu(x,y)e^{-\alpha\mu(x)/2-\alpha\mu(y)/2} dxdy
\end{align*}
where $a>\min(1/2,\sigma)$ if $\sigma<1$ and $a=1$ if $\sigma=1$.
Using $\int_0^{x_0}\nu(x,y) dx\leq x_0\mu(y)$, we have if $x_0>0$ (otherwise the bound is trivial)
\begin{align*}
\int_0^{x_0}\int_0^\infty \nu(x,y)e^{-\alpha\mu(x)/2-\alpha\mu(y)/2} dxdy\leq e^{-\alpha\mu(x_0)/2}x_0\int_0^\infty \mu(y)e^{-\alpha\mu(y)/2}dy.
\end{align*}
Since $\mu(x_0)>0$, the RHS is in $o(\alpha^{-p})$ for any $p>0$.
Using Lemma \ref{lemma:abelianvariationsv3} ($\sigma<1$) or  \ref{lemma:abelianvariations} ($\sigma=1$) together with Assumption~\ref{assumpt:1}, we therefore obtain
\begin{align*}
\int_{\mathbb{R}_{+}^{2}}  \nu(x,y)   e^{-\alpha\mu(x)-\alpha\mu(y)+\alpha\nu(x,y)} dxdy &=O \{ \alpha^{2\sigma-2a}\ell^2_\sigma(\alpha)\}.
\end{align*}
In the case $\sigma=0$ and $a=1$, Lemma \ref{lemma:abelianvariations} and Assumption~\ref{assumpt:1} give
$$
\int_{\mathbb{R}_{+}^{2}}  \nu(x,y)   e^{-\alpha\mu(x)-\alpha\mu(y)+\alpha\nu(x,y)} dxdy=o\left (\alpha^{-2}\ell^2(\alpha)\right).
$$
\end{proof}

\section{Background on regular variation and some technical Lemmas about regularly varying functions} \label{sec:regularvariation}

\begin{definition}\label{def:RV}
A measurable function $U:\mathbb{R}_{+}\rightarrow\mathbb{R}_{+}$ is regularly
varying at $\infty$ with index $\rho\in\mathbb{R}$ if for $x>0$,
$
\lim_{t\rightarrow\infty}U(tx)/U(t)=x^{\rho}.%
$
We note $U\in RV_{\rho}$. If $\rho=0$, we call $U$ slowly varying.
\end{definition}

\begin{proposition}
If $U\in RV_{\rho}$, then there exists a slowly varying function $\ell\in
RV_{0}$ such that
\begin{equation}
U(x)=x^{\rho}\ell(x)
\end{equation}

\end{proposition}

\begin{definition}\label{def:debruijn}
The de Bruijn conjugate $\ell^\#$ of the slowly varying function $\ell$, which always exists, is uniquely defined up to asymptotic equivalence~\citep[Theorem 1.5.13]{Bingham1987} by
$$
\ell(x)\ell^\#\{x\ell(x)\}\rightarrow 1,~~~\ell^\#(x)\ell\{x\ell^\#(x)\}\rightarrow 1
$$
as $x\rightarrow\infty$. Then $(\ell^\#)^\# \sim\ell$. For example, $(\log^a x)^\#\sim \log^{-a} x$ for $a\neq 0$ and $\ell^\# (x)\sim 1/c$ if $\ell(x)\sim c$.
\end{definition}

\begin{proposition}
\cite[Proposition 0.8, Chapter 0]{Resnick1987} \label{prop:rv1} If $U\in
RV_{\rho}$, $\rho\in \mathbb R$, and the sequences $(a_{n})$ and $(a_{n}^{\prime})$
satisfy $0<a_{n}\rightarrow\infty$, $0<a_{n}^{\prime}\rightarrow\infty$ and
$a_{n}\sim ca_{n}^{\prime}$ for some $0<c<\infty$, then%
\[
U(a_{n})\sim c^{\rho}U(a_{n}^{\prime})\text{ as }n\rightarrow\infty.
\]

\end{proposition}

\begin{proposition}\citep[Proposition 0.7, p.21]{Resnick1987}
\label{prop:RVderivative}Let
$U:\mathbb{R}_{+}\rightarrow\mathbb{R}_{+}$ absolutely continuous with density
$u$, so that $U(x)=\int_0^x u(t)dt$. If $U\in RV_{\rho}$, $\rho\in\mathbb{R}$ and $u$ is monotone, then
\[
\lim_{x\rightarrow\infty}\frac{xu(x)}{U(x)}=\rho
\]
and if $\rho\neq0$, then $sign(\rho)u(x)\in RV_{\rho-1}$.

\end{proposition}

\end{appendices}

\pagebreak

\begin{frontmatter}
\begin{aug}
\title{On sparsity, power-law and clustering properties of graphex processes:\\Supplementary Material}

\date{\today}
\author{Fran\c cois Caron},
\author{Francesca Panero}
\and
\author{Judith Rousseau}

\end{aug}
\end{frontmatter}

\setcounter{equation}{0}
\setcounter{figure}{0}
\setcounter{table}{0}
\setcounter{page}{1}
\makeatletter
\renewcommand{\theequation}{S\arabic{equation}}
\renewcommand{\thefigure}{S\arabic{figure}}
\renewcommand{\bibnumfmt}[1]{[S#1]}
\renewcommand{\citenumfont}[1]{S#1}
\renewcommand*{\thetheorem}{S\arabic{theorem}}
\renewcommand\thesection{S\arabic{section}}
\renewcommand\thesubsection{\thesection.\arabic{subsection}}
\renewcommand\theequation{S\arabic{equation}}

The supplementary material is organised as follows. Section \ref{sec:proofsecond} contains proofs of asymptotic bounds on the variances of the number of nodes, number of nodes of a given degree, and number of triangles of nodes with a given degree, as well as the proof of a secondary proposition for the local clustering coefficient. Section \ref{sec:proofsecondCLT} contains proofs of secondary propositions for the central limit theorem. For the sake of simplicity, all Sections, Equations, Lemmas, etc., in the Supplementary material here are denoted with a prefix S, to differentiate them from the Sections, Equations, Lemmas, etc., of the main text \citep{CaronRousseau2017:main}. 

\section{Proofs of secondary propositions for the variances and clustering coefficients}
\label{sec:proofsecond}

\subsection{Proof of Proposition \ref{prop:variancenbnodes} on $\var(N_\alpha)$}%
\label{sec:proofvarnnodes}

An application of the Slivnyak-Mecke and Campbell theorems gives
\begin{align*}
&\var(N_{\alpha})    =E(N_{\alpha})+2\alpha^{2}\int_{0}^{\infty}%
\mu(x)(1-W(x,x))e^{-\alpha\mu(x)}dx\\
&  +\alpha^{2}\int_{\mathbb{R}_{+}^{2}}(e^{\alpha\nu(x,y)}%
-1+W(x,y))(1-W(x,x))(1-W(y,y))(1-W(x,y))e^{-\alpha\mu(x)-\alpha\mu(y)}dxdy.
\end{align*}

Using the inequality $e^x-1\leq xe^{x}$,
\begin{align*}
\var(N_{\alpha})
& \leq E(N_{\alpha}) +  2\alpha^2 \int_{\mathbb{R}_{+}}\mu(x)e^{-\alpha\mu(x)}dx  +\alpha^{2}\int_{\mathbb{R}_{+}^{2}} e^{-\alpha\mu(x)-\alpha\mu(y)}\left\{ \alpha \nu(x,y)e^{\alpha \nu(x,y)} +W(x,y) \right\}dxdy
\end{align*}

Now, using Lemmas \ref{lemma:abelianvariations} and~\ref{lemma:boundintnu}
\begin{align*}
\int_{\mathbb{R}_{+}^{2}}W(x,y) e^{-\alpha\mu(x)-\alpha\mu(y)}dxdy \leq  \int_{\mathbb R_+} \mu(x) e^{-\alpha\mu(x) } dx&=O\left ( \alpha^{\sigma-1}\ell_\sigma(\alpha)\right).\\
\int_{\mathbb{R}_{+}^{2}}  \nu(x,y)   e^{-\alpha\mu(x)-\alpha\mu(y)+\alpha\nu(x,y)} dxdy&=O\left (\alpha^{2\sigma-2a}\ell_\sigma^2(\alpha)\right).
\end{align*}

It follows that
$\var(N_{\alpha})  =O(\alpha^{3+2\sigma -2a}\ell_\sigma(\alpha)^2).
$

\label{sec:proofvariance}
Assume Assumption \ref{assumpt:1} and \ref{assumpt:2} are satisfied, with $a=1$. From the first part of Proposition \ref{prop:variancenbnodes}, we have the upper bound
$\var(N_{\alpha})=O\left(\alpha^{1+2\sigma}\ell_\sigma^{2}(\alpha)\right ).
$\\
We now derive a lower bound. If $\sigma=0$, $\var(N_{\alpha})\geq E(N_{\alpha})\gtrsim \alpha\ell_0(\alpha)$, hence $\var(N_{\alpha})\asymp \alpha\ell(\alpha)$. Consider now the case $\sigma>0$. We have
\[
\var(N_{\alpha})\geq\alpha^{2}\int_{\mathbb{R}_{+}^{2}}(e^{\alpha\nu
(x,y)}-1)(1-W(x,x))(1-W(y,y))(1-W(x,y))e^{-\alpha\mu(x)-\alpha\mu(y)}dxdy
\]
and using the inequality $e^{x}-1\geq x$ and Assumption \ref{assumpt:4}
\begin{align*}
\var(N_{\alpha})  &  \geq\alpha^{3}\int_{\mathbb{R}_{+}^{2}}\nu
(x,y)(1-W(x,x))(1-W(y,y))(1-W(x,y))e^{-\alpha\mu(x)-\alpha\mu(y)}dxdy\\
&  \geq C_{0}\alpha^{3}\int_{x_{0}}^{\infty}\int_{x_{0}}^{\infty}\mu
(x)\mu(y)(1-W(x,x))(1-W(y,y))(1-W(x,y))e^{-\alpha\mu(x)-\alpha\mu(y)}dxdy
\end{align*}
Using Lemmas~\ref{lemma:abelianvariations} and \ref{lemma:abelianvariationsv2}, we have
\begin{align*}
&  \int_{x_{0}}^{\infty}\int_{x_{0}}^{\infty}\mu(x)\mu
(y)(1-W(x,x))(1-W(y,y))(1-W(x,y))e^{-\alpha\mu(x)-\alpha\mu(y)}dxdy\\
&  \sim\int_{0}^{\infty}\int_{0}^{\infty}\mu(x)\mu(y)e^{-\alpha
\mu(x)-\alpha\mu(y)}dxdy=\left (\int_{\mathbb{R}_{+}}\mu(x)e^{-\alpha\mu(x)}dx\right )^2\sim\alpha^{2\sigma-2}\ell_\sigma^{2}(\alpha).
\end{align*}
It follows that, for $\sigma>0$, $\var(N_{\alpha})\gtrsim\alpha^{1+2\sigma}\ell_\sigma^{2}(\alpha).$
Combining this with the upper bound gives, for all $\sigma\in[0,1]$
$
\var(N_{\alpha})\asymp\alpha^{1+2\sigma}\ell_\sigma^{2}(\alpha).
$

\subsection{Proof of proposition \ref{prop:variancenbnodesj} on $\var(N_{\alpha,j})$}
\label{sec:proofvarnnodesj}

We have,
\begin{align*}
& E(N_{\alpha,j}^{2}\mid M)-E(N_{\alpha,j}\mid M)\nonumber\\
&=  \sum_{i_{1}\neq i_{2}}\1{\theta_{i_{1}}\leq\alpha}\1{\theta_{i_{2}}\leq\alpha}\Pr\left\{  \sum
_{k}\1{\theta_{k}\leq\alpha}Z_{i_{1}k}=j\text{ and } \sum_{k}\1{\theta_{k}\leq\alpha}Z_{i_{2}%
,k}=j\mid M\right\}. \label{eq:bigexpectation}\\
&=\sum_{b\in\{0,1\}^3} \sum_{j_1=0}^j \sum_{i_{1}\neq i_{2}}\1{\theta_{i_{1}}\leq\alpha}\1{\theta_{i_{2}}\leq\alpha}\\
&\quad\times\Pr\left\{  \sum
_{k}\1{\theta_{k}\leq\alpha}Z_{i_{1}k}=j\text{ and } \sum_{k}\1{\theta_{k}\leq\alpha}Z_{i_{2}%
,k}=j\text{ and } \sum
_{k}\1{\theta_{k}\leq\alpha}Z_{i_{1}k}Z_{i_{2}k}=j-j_1\right .\\
&\quad\quad\quad\quad \left .\text{ and }Z_{i_1i_1}=b_{11},Z_{i_1i_2}=b_{12},Z_{i_2i_2}=b_{22} \mid M\right \}
\end{align*}
where $b=(b_{11},b_{12},b_{22})\in\{0,1\}^3$. Let $A_1,A_2,A_{12}$ be disjoint subsets of $\mathbb N\backslash \{i_1,i_2\}$ such that $|A_{12}|+b_{12}=j-j_1$, $|A_1| + |A_{1,2}|+b_{11}+b_{12} = |A_2| + |A_{1,2}|+b_{22}+b_{12} = j$ respectively corresponding to the indices of nodes only connected to node $i_1$, only to node $i_2$, or to both nodes $(i_1,i_2)$. Let $A=\{i_1,i_2\}\cup A_1\cup A_2\cup A_{12}$. We have
\begin{align*}
&\Pr\left\{  \sum
_{k}\1{\theta_{k}\leq\alpha}Z_{i_{1}k}=j, \sum_{k}\1{\theta_{k}\leq\alpha}Z_{i_{2}%
,k}=j, \sum
_{k}\1{\theta_{k}\leq\alpha}Z_{i_{1}k}Z_{i_{2}k}=j-j_1,(Z_{i_1i_1},Z_{i_1i_2},Z_{i_2i_2})=b \mid M\right \}\\
&=\sum_{A_1,A_2,A_{12}}\frac{ \1{\theta_{i_1} \leq\alpha}\1{\theta_{i_{2}}\leq\alpha}}{(j-j_1-b_{12})!(j_1-b_{11})!(j_1-b_{22})! }   W(\vartheta_{i_{1}}, \vartheta_{i_1})^{b_{11}}W(\vartheta_{i_{2}}, \vartheta_{i_2})^{b_{22}} W(\vartheta_{i_{1}}, \vartheta_{i_2})^{b_{12}}  \\
& \quad\times\{ 1 - W(\vartheta_{i_{1}}, \vartheta_{i_1})\}^{1-b_{11}}\{ 1 - W(\vartheta_{i_{2}}, \vartheta_{i_2})\}^{1-b_{22}} \{ 1 - W(\vartheta_{i_{1}}, \vartheta_{i_2})\}^{1-b_{12}}  \\
& \quad\times\left [ \prod_{k \in A_1} \1{\theta_k \leq \alpha}
W(\vartheta_{i_{1}},\vartheta_{i_{k}})\{1 - W(\vartheta_{i_{2}},\vartheta_{i_{k}})\} \right ] \left [ \prod_{k \in A_2} \1{\theta_k \leq \alpha}
\{1 - W(\vartheta_{i_{1}},\vartheta_{i_{k}})\}W(\vartheta_{i_{2}},\vartheta_{i_{k}}) \right ] \\
& \quad \left [ \prod_{k \in A_{12}} \1{\theta_k \leq \alpha}
W(\vartheta_{i_{1}},\vartheta_{i_{k}})W(\vartheta_{i_{2}},\vartheta_{i_{k}}) \right ] \exp\left[ - \sum_{k \in \mathbb N\backslash A} \{g_{\alpha,\vartheta_{i_1}}(\theta_k, \vartheta_k)+g_{\alpha,\vartheta_{i_2}}(\theta_k, \vartheta_k)\}\right] \\
\end{align*}
Using the extended Slivnyak-Mecke theorem,
 \begin{equation*}
 \begin{split}
& E(N_{\alpha,j}^{2})-E(N_{\alpha,j})\\
&=\sum_{b\in\{0,1\}^3} \sum_{j_1=0}^j \frac{\alpha^{2+j+j_1-b_{11}-b_{12}-b_{22}}}{(j-j_1-b_{12})!(j_1-b_{11})!(j_1-b_{22})! }\1{j_1\geq b_{11}}\1{j_1\geq b_{22}}\1{j_1\leq{j-b_{12}}}\\
&\qquad \times\int_{\R_+^2} \{\mu(x)-\nu(x,y)\}^{j_1-b_{11}}\{\mu(y)-\nu(x,y)\}^{j_1-b_{22}} \nu(x,y)^{j-j_1-b_{12}}e^{- \alpha \mu(x) -\alpha \mu(y) +\alpha \nu(x,y) } \\
 & \qquad  \qquad   \times W(x,x)^{b_{11}}W(y,y)^{b_{22}}W(x,y)^{b_{12}}\{1 - W(x,x)\}^{1-b_{11}}\{1-W(y,y)\}^{1-b_{22}}\{1 - W(x,y)\}^{1-b_{12}}dxdy\\
 &\leq \sum_{b\in\{0,1\}^3} \sum_{j_1=0}^j \frac{\alpha^{2+j+j_1-b_{11}-b_{12}-b_{22}}}{(j-j_1-b_{12})!(j_1-b_{11})!(j_1-b_{22})! }\1{j_1\geq b_{11}}\1{j_1\geq b_{22}}\1{j_1\leq{j-b_{12}}}\\
 &\qquad\times\int_{\R_+^2} \mu(x)^{j_1-b_{11}}\mu(y)^{j_1-b_{22}} \nu(x,y)^{j-j_1-b_{12}}e^{- \alpha \mu(x) -\alpha \mu(y) +\alpha \nu(x,y) }  \\
 & \qquad  \qquad   \times W(x,x)^{b_{11}}W(y,y)^{b_{22}}W(x,y)^{b_{12}}\{1 - W(x,x)\}^{1-b_{11}}\{1-W(y,y)\}^{1-b_{22}}\{1 - W(x,y)\}^{1-b_{12}}dxdy
\end{split}
\end{equation*}
We will need the following lemma.

\begin{lemma}\label{lemma:Ir}
Let  $r\geq 1$,  $j_1, j_2 \geq 0$. Define
 \begin{equation*}
 \begin{split}
  I_r& := \int_{\R^2} [\alpha\mu(x)]^{j_1}[\alpha\mu(y)]^{j_2} (\alpha \nu(x,y))^{r} e^{ - \alpha \mu(x) -\alpha \mu(y) +\alpha \nu(x,y)}dxdy.
\end{split}
\end{equation*}
Under Assumptions 1 and 2, we have
$$I_r  =O( \alpha^{r-2ar+2\sigma }  \ell_\sigma^2(\alpha))$$ for all $r\geq 1$.

\end{lemma}
\begin{proof}
We have, using Assumption 2, that
 \begin{equation*}
 \begin{split}
  I_r& \leq  \alpha^{r}  \int_{\R_+^2} [\alpha\mu(x)]^{j_1}[\alpha\mu(y)]^{j_2} \nu(x,y)^{r} e^{ - \alpha \left\{ \mu(x) + \mu(y) \right\}/2 }dxdy\\
    &\leq C_1^r \alpha^{r-2ar}  \left(\int_{\mathbb R_+} (\alpha \mu(x) )^{j_1+ar}e^{-\alpha \mu(x) /2} dx \right) \left(\int_{\mathbb R_+} (\alpha \mu(x) )^{j_2+ar}e^{-\alpha \mu(x) /2} dx \right)+o(\alpha^{-p}).
\end{split}
\end{equation*}
for any $p>0$. Assumption 1 and Lemmas \ref{lemma:abelianvariations} ($\sigma=1$) and  \ref{lemma:abelianvariationsv3} ($\sigma\in[0,1)$) imply that
$$I_r  =O( \alpha^{r-2ar+2\sigma }  \ell_\sigma^2(\alpha))$$ for all $r\geq 1$.
\end{proof}

It follows
\begin{align*}
 E(N_{\alpha,j}^{2})-E(N_{\alpha,j})&\lesssim \sum_{b\in\{0,1\}^3} \frac{\alpha^{2+2j-b_{11}-b_{22}-2b_{12}}}{(j-b_{12}-b_{11})!(j-b_{12}-b_{22})! }\1{j\geq b_{11}+b_{12}}\1{j\geq b_{22}+b_{12}}\\
 &\qquad\times\int_{\R_+^2} \mu(x)^{j-b_{12}-b_{11}}\mu(y)^{j-b_{12}-b_{22}} e^{- \alpha \mu(x) -\alpha \mu(y) +\alpha \nu(x,y) } \\
 & \qquad  \qquad   \times W(x,x)^{b_{11}}W(y,y)^{b_{22}}W(x,y)^{b_{12}}\\
 &\qquad  \qquad   \times\{1 - W(x,x)\}^{1-b_{11}}\{1-W(y,y)\}^{1-b_{22}}\{1 - W(x,y)\}^{1-b_{12}}dxdy\\
 &\qquad+O\{ \alpha^{2+2\sigma+1-2a}\ell_\sigma^2(\alpha)\}.
\end{align*}
Let $V_{0}$ and $V_1$ respectively denote the sum of terms such that $b_{12}=0$ and $b_{12}=1$ in the above sum. Using the inequality $e^{x}\leq 1+xe^x$,
\bgroup
\allowdisplaybreaks
\begin{align*}
V_{0}&= \sum_{b_{11},b_{22}\in\{0,1\}^2}  \frac{\alpha^{2+2j-b_{11}-b_{22}}}{(j-b_{11})!(j-b_{22})! } \int_{\R_+^2} \mu(x)^{j-b_{11}}\mu(y)^{j-b_{12}} e^{- \alpha \mu(x) -\alpha \mu(y) +\alpha \nu(x,y) }\\
& \qquad  \qquad    \times W(x,x)^{b_{11}}W(y,y)^{b_{22}}\{1 - W(x,x)\}^{1-b_{11}}\{1-W(y,y)\}^{1-b_{22}}\{1 - W(x,y)\}dxdy\\
&= \sum_{b_{11},b_{22}}  \frac{\alpha^{2+2j-b_{11}-b_{22}}}{(j-b_{11})!(j-b_{22})! } \int_{\R_+^2} \mu(x)^{j-b_{11}}\mu(y)^{j-b_{12}} e^{- \alpha \mu(x) -\alpha \mu(y)}\\
& \qquad  \qquad    \qquad \qquad\times W(x,x)^{b_{11}}W(y,y)^{b_{22}}\{1 - W(x,x)\}^{1-b_{11}}\{1-W(y,y)\}^{1-b_{22}}dxdy\\
&\quad+O \left\{\sum_{b_{11},b_{22}}  \frac{\alpha^{3+2j-b_{11}-b_{22}}}{(j-b_{11})!(j-b_{22})! } \int_{\R_+^2} \mu(x)^{j-b_{11}}\mu(y)^{j-b_{12}}\nu(x,y) e^{- \alpha \mu(x) -\alpha \mu(y) +\alpha \nu(x,y) }dxdy \right\} \\
&=\sum_{b_{11},b_{22}}  \left \{ \frac{\alpha^{1+j-b_{11}}}{(j-b_{11})!}\int_{\mathbb R_+}\mu(x)^{j-b_{11}}W(x,x)^{b_{11}}\{1 - W(x,x)\}^{1-b_{11}}e^{-\mu(x)} dx \right \}  \\
&\quad\quad\quad\quad\quad\times \left \{ \frac{\alpha^{1+j-b_{22}}}{(j-b_{22})!}\int_{\mathbb R_+}\mu(y)^{j-b_{22}}W(y,y)^{b_{22}}\{1 - W(y,y)\}^{1-b_{22}}e^{-\mu(y)} dy \right \}\\
&\quad+O\{\alpha^{2+2\sigma+1-2a}\ell_\sigma^2(\alpha)\}= E(N_{\alpha,j})^2  +O\{\alpha^{2+2\sigma+1-2a}\ell_\sigma^2(\alpha)\}
\end{align*}
\egroup
Similarly,
\begin{align*}
V_{1}&\leq  \sum_{b_{11},b_{22}}  \frac{\alpha^{2j-b_{11}-b_{22}}\1{j\geq 1+b_{11}}\1{j\geq 1+b_{22}}}{(j-1-b_{11})!(j-1-b_{22})! } \int_{\R_+^2} \mu(x)^{j-1-b_{11}}\mu(y)^{j-1-b_{12}} e^{\alpha \nu(x,y)- \alpha \mu(x) -\alpha \mu(y) }W(x,y)dxdy\\
&\leq \sum_{b_{11},b_{22}}  \frac{\alpha^{2j-b_{11}-b_{22}}\1{j\geq 1+b_{11}}\1{j\geq 1+b_{22}}}{(j-1-b_{11})!(j-1-b_{22})! } \int_{\R_+^2} \mu(x)^{j-1-b_{11}}\mu(y)^{j-1-b_{12}} e^{- \alpha \mu(x) -\alpha \mu(y) }W(x,y)dxdy\\
&\quad +O\{\alpha^{2+2\sigma+1-2a}\ell_\sigma^2(\alpha)\}
\end{align*}

For $j_1\geq 1$ and $j_2\geq 1$, using Cauchy-Schwarz and Lemma \ref{lemma:abelianvariationsv3},
\begin{align*}
&\int W(x,y)\mu(x)^{j_1} \mu(y)^{j_2} e^{-\alpha \mu(x) -\alpha \mu(y)} dxdy\\
&\leq
\int_{\mathbb R_+}\mu(x)^{j_1}  e^{-\alpha \mu(x)} \left\{ \int W(x,y)  \mu(y)^{2j_2}e^{ -2\alpha \mu(y)} dy\right\}^{1/2} \mu(x)^{1/2} dx \\
& \leq \left\{\int\mu(x)^{j_1+1/2}  e^{-\alpha \mu(x)} dx\right \}  \left\{\int  \mu(y)^{2j_2}e^{ -2\alpha \mu(y)} dy\right\}^{1/2} =O\left \{\alpha^{3\sigma/2-j_1-j_2-1/2}\ell^{3/2}_\sigma(\alpha)\right \}
\end{align*}
and for $j_1\geq 0$
\begin{align*}
\int_{\R_+^2} \mu(x)^{j_1} e^{- \alpha \mu(x) -\alpha \mu(y) }W(x,y)dxdy&\leq \int_{\R_+^2} \mu(x)^{j_1} e^{- \alpha \mu(x) }W(x,y)dxdy\\
&=\int_{\R_+} \mu(x)^{j_1+1} e^{- \alpha \mu(x) }dx=O\left \{\alpha^{\sigma-j_1-1}\ell_\sigma(\alpha)\right \}
\end{align*}
It follows that $V_1=O\{\alpha^{2+3\sigma/2-1/2}\ell^{3/2}_\sigma(\alpha)\}+O\{\alpha^{1+\sigma}\ell_\sigma^2(\alpha)\}+O\{\alpha^{2+2\sigma+1-2a}\ell_\sigma^2(\alpha)\}$. Combining the upper bounds on $V_0$ and $V_1$, we obtain
$
\var(N_{\alpha,j})=O(\alpha^{3-2a+2\sigma}\ell^2_\sigma(\alpha))
$
and this terminates the proof. In the case $\sigma=0$ and $a=1$, one can use Lemma \ref{lemma:abelianvariations} instead of  Lemma \ref{lemma:abelianvariationsv3} and replace big $O$ by little $o$ in the above bounds, together with the fact that $E(N_{\alpha,j})=o(\alpha\ell(\alpha))$ and $\ell(t)=O(\ell^2(t))$ if $\sigma=0$.

\subsection{Proof of Proposition \ref{prop:variancelocalclust}}

\label{sec:proofvariancelocalclust}

We first prove the first equality. The proof is similar to that of Proposition \ref{prop:variancenbnodesj}, given in Section~\ref{sec:proofvarnnodesj}.
For any $j\geq 2$,
\begin{align*}
2R_{\alpha,j}&=2\sum_i T_{\alpha i}\1{D_{\alpha i}=j}\1{\theta_i\leq \alpha}=\sum_{i\neq k\neq l} Z_{ik}Z_{il}Z_{kl}\1{D_{\alpha i}=j}\1{\theta_i\leq \alpha}.
\end{align*}
Let $S_{\alpha j}:=4R_{\alpha j}^2$. We have
\begin{equation}\label{Big:R2}
\begin{split}
S_{\alpha j}&=\left (\sum_{i\neq k\neq l} Z_{ik}Z_{il}Z_{kl}\1{D_{\alpha i}=j}\1{\theta_i\leq \alpha}\right )^2\\
&=\sum_{i_1\neq k_1\neq l_1\neq i_2\neq k_2\neq l_2} Z_{i_1k_1}Z_{i_1l_1}Z_{k_1l_1}Z_{i_2k_2}Z_{i_2l_2}Z_{k_2l_2}  \1{D_{\alpha i_1}=j}\1{D_{\alpha i_2}=j}\1{\theta_{i_1}\leq \alpha}  \1{\theta_{i_2}\leq \alpha}\\
&\quad+2\sum_{i_1\neq k_1\neq l_1\neq i_2\neq k_2} Z_{i_1k_1}Z_{i_1l_1}Z_{k_1l_1}Z_{i_2k_2}Z_{i_2l_1}Z_{k_2l_1}  \1{D_{\alpha i_1}=j}\1{D_{\alpha i_2}=j}\1{\theta_{i_1}\leq \alpha}  \1{\theta_{i_2}\leq \alpha}\\
&\quad+2\sum_{i_1\neq k_1\neq l_1\neq i_2} Z_{i_1k_1}Z_{i_1l_1}Z_{k_1l_1}Z_{i_2k_1}Z_{i_2l_1}  \1{D_{\alpha i_1}=j}\1{D_{\alpha i_2}=j}\1{\theta_{i_1}\leq \alpha}  \1{\theta_{i_2}\leq \alpha}\\
&\quad+\sum_{i_1\neq k_1\neq l_1\neq k_2\neq l_2} Z_{i_1k_1}Z_{i_1l_1}Z_{k_1l_1}Z_{i_1k_2}Z_{i_1l_2}Z_{k_2l_2}  \1{D_{\alpha i_1}=j}\1{\theta_{i_1}\leq \alpha}\\
&\quad+2\sum_{i_1\neq k_1\neq l_1\neq k_2} Z_{i_1k_1}Z_{i_1l_1}Z_{k_1l_1}Z_{i_1k_2}Z_{k_2l_1}  \1{D_{\alpha i_1}=j}\1{\theta_{i_1}\leq \alpha}\\
&\quad+2\sum_{i_1\neq k_1\neq l_1} Z_{i_1k_1}Z_{i_1l_1}Z_{k_1l_1} \1{D_{\alpha i_1}=j}\1{\theta_{i_1}\leq \alpha}
\end{split}
\end{equation}

Note that some of the terms above are equal to 0 if $j\leq 4$. First note that for any $j_0\leq j$:
\begin{equation}
\label{eq:temp2}
\sum_{i\neq k_1\neq \ldots\neq k_{j_0}} \left ( \prod_{l=1}^{j_0} Z_{il}\right )\1{D_{\alpha i}=j}\1{\theta_i\leq \alpha}\leq \binom{j}{j_0} N_{\alpha j}
\end{equation}

Hence the last three terms of the right-handside of \eqref{Big:R2} are upper bounded by  $C_j N_{\alpha,j}$, for some constant $C_j$ that does not depend on $\alpha$. Consider now
\begin{align*}
S_{\alpha,j,1}&=\sum_{i_1\neq k_1\neq l_1\neq i_2\neq k_2\neq l_2} Z_{i_1k_1}Z_{i_1l_1}Z_{k_1l_1}Z_{i_2k_2}Z_{i_2l_2}Z_{k_2l_2}  \1{D_{\alpha i_1}=j}\1{D_{\alpha i_2}=j}\1{\theta_{i_1}\leq \alpha}  \1{\theta_{i_2}\leq \alpha}\\
&=\sum_{j_1=2}^j S_{\alpha,j,1,j_1}
\end{align*}
where, for $j_1=2,\ldots,j$
\begin{align*}
S_{\alpha,j,1,j_1}&= \sum_{\substack{i_1\neq k_1\neq l_1\\\neq i_2\neq k_2\neq l_2}} Z_{i_1k_1}Z_{i_1l_1}Z_{k_1l_1}Z_{i_2k_2}Z_{i_2l_2}Z_{k_2l_2}\1{D_{\alpha i_1}=j}\1{D_{\alpha i_2}=j}\1{\sum_k Z_{i_1k}Z_{i_2k}\1{\theta_k\leq \alpha}=j-j_1  }\1{\theta_{i_1}\leq \alpha}  \1{\theta_{i_2}\leq \alpha}\\
&=\sum_{b\in\{0,1\}^3} \sum_{\substack{i_1\neq k_1\neq l_1\\\neq i_2\neq k_2\neq l_2}} Z_{i_1k_1}Z_{i_1l_1}Z_{k_1l_1}Z_{i_2k_2}Z_{i_2l_2}Z_{k_2l_2}\\
&\quad\times\1{D_{\alpha i_1}=j}\1{D_{\alpha i_2}=j}\1{\sum_k Z_{i_1k}Z_{i_2k}\1{\theta_k\leq \alpha}=j-j_1  }\1{Z_{i_1 i_1}=b_{11}}\1{Z_{i_1 i_2}=b_{12}}\1{Z_{i_2 i_2}=b_{22}}\1{\theta_{i_1}\leq \alpha}  \1{\theta_{i_2}\leq \alpha}
\end{align*}
where we introduce $b=(b_{11},b_{12},b_{22})\in\{0,1\}^3$ as in Section~\ref{sec:proofvarnnodesj}. Using the extended Slivnyak-Mecke theorem, for $j_1=2,\ldots,j$,
 \begin{align}
& E(S_{\alpha,j,1,j_1})\\
&=\sum_{b\in\{0,1\}^3} \frac{\alpha^{2+j+j_1-b_{11}-b_{12}-b_{22}}}{(j-j_1-b_{12})!(j_1-b_{11}-2)!(j_1-b_{22}-2)! }\1{j_1\leq{j-b_{12}}}\1{j_1\geq b_{11}}\1{j_1\geq b_{22}}\nonumber\\
&\qquad \times\int_{\R_+^6} \{\mu(x_1)-\nu(x_1,x_2)\}^{j_1-2-b_{11}}\{\mu(x_2)-\nu(x_1,x_2)\}^{j_1-2-b_{22}} \nu(x_1,x_2)^{j-j_1-b_{12}}e^{- \alpha \mu(x_1) -\alpha \mu(x_2) +\alpha \nu(x_1,x_2) } \nonumber\\
 &\qquad  \qquad   \times W(x_1,y_1)W(x_1,z_1)W(y_1,z_1)W(x_2,y_2)W(x_2,z_2)W(y_2,z_2)\nonumber\\
 & \qquad  \qquad   \times W(x,x)^{b_{11}}W(y,y)^{b_{22}}W(x,y)^{b_{12}}\nonumber\\
 & \qquad  \qquad   \times\{1 - W(x,x)\}^{1-b_{11}}\{1-W(y,y)\}^{1-b_{22}}\{1 - W(x,y)\}^{1-b_{12}}dx_1dy_1dz_1dx_2dy_2dz_2\nonumber\\
 &\leq \sum_{b\in\{0,1\}^3} \frac{\alpha^{2+j+j_1-b_{11}-b_{12}-b_{22}}}{(j-j_1-b_{12})!(j_1-b_{11}-2)!(j_1-b_{22}-2)! }\1{j_1\leq{j-b_{12}}}\1{j_1\geq b_{11}}\1{j_1\geq b_{22}}\nonumber\\
 &\qquad\times\int_{\R_+^6} \mu(x_1)^{j_1-2-b_{11}}\mu(x_2)^{j_1-2-b_{22}} \nu(x_1,x_2)^{j-j_1-b_{12}}e^{- \alpha \mu(x) -\alpha \mu(y) +\alpha \nu(x_1,x_2) }  \nonumber\\
 &\qquad  \qquad   \times W(x_1,y_1)W(x_1,z_1)W(y_1,z_1)W(x_2,y_2)W(x_2,z_2)W(y_2,z_2)\nonumber\\
 & \qquad  \qquad   \times W(x_1,x_1)^{b_{11}}W(x_2,x_2)^{b_{22}}W(x_1,x_2)^{b_{12}}\nonumber\\
 & \qquad  \qquad   \times\{1 - W(x_1,x_1)\}^{1-b_{11}}\{1-W(x_2,x_2)\}^{1-b_{22}}\{1 - W(x_1,x_2)\}^{1-b_{12}}dx_1dy_1dz_1dx_2dy_2dz_2\label{eq:temp1}
\end{align}
For $b_{12}\neq 0$ or $j\neq j_1$, we can bound the terms in the above sum by
\begin{align}
 &\frac{\alpha^{2+j+j_1-b_{11}-b_{12}-b_{22}}}{(j-j_1-b_{12})!(j_1-b_{11}-2)!(j_1-b_{22}-2)! }\1{j_1\leq{j-b_{12}}}\1{j_1\geq b_{11}}\1{j_1\geq b_{22}} \nonumber\\
 &\qquad\times\int_{\R_+^4} \mu(x_1)^{j_1-b_{11}}\mu(x_2)^{j_1-b_{22}} \nu(x_1,x_2)^{j-j_1-b_{12}}e^{- \alpha \mu(x) -\alpha \mu(y) +\alpha \nu(x_1,x_2) } \nonumber \\
 & \qquad  \qquad   \times W(x_1,x_1)^{b_{11}}W(x_2,x_2)^{b_{22}}W(x_1,x_2)^{b_{12}} \nonumber\\
 & \qquad  \qquad   \times\{1 - W(x_1,x_1)\}^{1-b_{11}}\{1-W(x_2,x_2)\}^{1-b_{22}}\{1 - W(x_1,x_2)\}^{1-b_{12}}dx_1dx_2 \nonumber\\
 &=O(\alpha^{3+2\sigma-2a}\ell_\sigma(\alpha)^2)\label{eq:temp1b}
\end{align}
using the intermediate results of the proof in Section \ref{sec:proofvarnnodesj}.

Consider now the sum of terms such that $b_{12}=0$ and $j=j_1$ in \eqref{eq:temp1}. Using the inequality $e^x\leq 1+xe^x$, this sum is upper bounded by
\begin{align*}
 &\sum_{b_{11},b_{12}} \frac{\alpha^{2+2j-b_{11}-b_{22}}}{(j-b_{11}-2)!(j-b_{22}-2)! }\int_{\R_+^6} \mu(x_1)^{j-2-b_{11}}\mu(x_2)^{j-2-b_{22}} e^{- \alpha \mu(x) -\alpha \mu(y) +\alpha \nu(x_1,x_2) }  \\
 &\qquad     \times W(x_1,y_1)W(x_1,z_1)W(y_1,z_1)W(x_2,y_2)W(x_2,z_2)W(y_2,z_2)\\
 & \qquad     \times W(x_1,x_1)^{b_{11}}W(x_2,x_2)^{b_{22}}\{1 - W(x_1,x_1)\}^{1-b_{11}}\{1-W(x_2,x_2)\}^{1-b_{22}}dx_1dy_1dz_1dx_2dy_2dz_2\\
 &\leq  \sum_{b_{11},b_{12}} \frac{\alpha^{2+2j-b_{11}-b_{22}}}{(j-b_{11}-2)!(j-b_{22}-2)! }\int_{\R_+^6} \mu(x_1)^{j-2-b_{11}}\mu(x_2)^{j-2-b_{22}} e^{- \alpha \mu(x) -\alpha \mu(y) }  \\
 & \qquad   \times W(x_1,y_1)W(x_1,z_1)W(y_1,z_1)W(x_2,y_2)W(x_2,z_2)W(y_2,z_2)\\
 &  \qquad   \times W(x_1,x_1)^{b_{11}}W(x_2,x_2)^{b_{22}}\{1 - W(x_1,x_1)\}^{1-b_{11}}\{1-W(x_2,x_2)\}^{1-b_{22}}dx_1dy_1dz_1dx_2dy_2dz_2\\
 &\quad+\sum_{b_{11},b_{12}} \frac{\alpha^{3+2j-b_{11}-b_{22}}}{(j-b_{11}-2)!(j-b_{22}-2)! }\int_{\R_+^6} \mu(x_1)^{j-2-b_{11}}\mu(x_2)^{j-2-b_{22}} \nu(x,y)e^{- \alpha \mu(x) -\alpha \mu(y) +\alpha\nu(x,y)}  \\
 &\qquad    \times W(x_1,y_1)W(x_1,z_1)W(y_1,z_1)W(x_2,y_2)W(x_2,z_2)W(y_2,z_2)\\
 & \qquad   \times W(x_1,x_1)^{b_{11}}W(x_2,x_2)^{b_{22}}  \times\{1 - W(x_1,x_1)\}^{1-b_{11}}\{1-W(x_2,x_2)\}^{1-b_{22}}dx_1dy_1dz_1dx_2dy_2dz_2\\
 &\leq 4E(R_{\alpha,j})^2\\
 &\quad+\sum_{b_{11},b_{12}} \frac{\alpha^{3+2j-b_{11}-b_{22}}}{(j-b_{11}-2)!(j-b_{22}-2)! }\int_{\R_+^2} \mu(x_1)^{j-b_{11}}\mu(x_2)^{j-b_{22}} \nu(x,y)e^{- \alpha \mu(x) -\alpha \mu(y) +\alpha\nu(x,y)}  \\
 & \qquad   \times W(x_1,x_1)^{b_{11}}W(x_2,x_2)^{b_{22}}  \times\{1 - W(x_1,x_1)\}^{1-b_{11}}\{1-W(x_2,x_2)\}^{1-b_{22}}dx_1dx_2\\
 &=4E(R_{\alpha,j})^2+O(\alpha^{3+2\sigma-2a}\ell_\sigma(\alpha)^2)
\end{align*}
using Lemma \ref{lemma:Ir} in Section \ref{sec:proofvarnnodesj}. It follows that
\begin{align}
E(S_{\alpha,j,1})=E(4R_{\alpha,j})^2 +O(\alpha^{3+2\sigma-2a}\ell_\sigma(\alpha)^2)\label{eq:temp3}
\end{align}

Consider now
\begin{align*}
S_{\alpha,j,2}=\sum_{i_1\neq k_1\neq l_1\neq i_2\neq k_2} Z_{i_1k_1}Z_{i_1l_1}Z_{k_1l_1}Z_{i_2k_2}Z_{i_2l_1}Z_{k_2l_1}  \1{D_{\alpha i_1}=j}\1{D_{\alpha i_2}=j}\1{\theta_{i_1}\leq \alpha}  \1{\theta_{i_2}\leq \alpha}
\end{align*}
We have similarly
\begin{align}
E(S_{\alpha,j,2})
 &=O(\alpha^{3+2\sigma-2a}\ell_\sigma(\alpha)^2)\label{eq:temp4}
 \end{align}
 using Lemma \ref{lemma:Ir}.
Similarly, using Lemma \ref{lemma:Ir},
$
E(S_{\alpha,j,3})=O(\alpha^{3+2\sigma-2a}\ell_\sigma(\alpha)^2).
$
\\
Combining the above bound with \eqref{eq:temp3} and  \eqref{eq:temp4}, we obtain
$
\var(R_{\alpha,j})=O(\alpha^{3+2\sigma-2a}\ell_\sigma(\alpha)^2).
$
We now consider the second bound in Proposition~\ref{prop:variancelocalclust}. Consider an increasing sequence  $\alpha_n \rightarrow \infty$ such that $\alpha_{n+1} - \alpha_n = o(\alpha_n)$ as $n\to\infty$. Let $I_{\alpha_n} = \{ i, \, \theta_i \leq \alpha_n\}$ and $I_n^c = I_{\alpha_{n+1}}\backslash I_{\alpha_{n}}$.

For any $j\geq 1$, let
\begin{align*}
\widetilde R_{nj}^{(1)}&:=\sum_{i\in I_{\alpha_{n}} }T_{\alpha_{n+1} i}\1{D_{\alpha_n i}=j} \1{ \sum_{i'\in I_n^c} Z_{ii'}=1}.
\end{align*}

We have, similarly to Equation~\eqref{Big:R2}
\begin{equation*}
\begin{split}
(\widetilde R_{nj}^{(1)})^2 & =
\sum_{i_1, i_2}^{I_{\alpha_n}}\sum_{k_1\neq l_1\neq i_1}^{I_{\alpha_{n+1}}}\sum_{i_2\neq k_2\neq l_2}^{I_{\alpha_{n+1}}} Z_{i_1k_1}Z_{i_1l_1}Z_{k_1l_1}Z_{i_2k_2}Z_{i_2l_2}Z_{k_2l_2}  \\
&\qquad\qquad\qquad\times\1{D_{\alpha_n i_1}=j}\1{D_{\alpha_n i_2}=j}\1{\sum_{i_1^{'}\in I_n^c}Z_{i_1^{'}i_1} =1}\1{\sum_{i_2^{'}\in I_n^c}Z_{i_2^{'}i_2} =1}.
\end{split}
\end{equation*}
Hence using the same decomposition as \eqref{Big:R2}  together with  the fact that $\1{\sum_{i'\in I_n^c}Z_{i'i} =1} \leq 1$, we derive the same bounds as \eqref{eq:temp2}, \eqref{eq:temp1b} and \eqref{eq:temp4} so that
\begin{equation*}
E( (\widetilde R_{nj}^{(1)})^2 ) \lesssim \alpha_n^{\sigma +1}\ell_\sigma(\alpha_n) + \alpha_n^{3 +2\sigma -2a}\ell_\sigma(\alpha_n)^2 + E( \widetilde S_{\alpha_n,j,1,j}  )
\end{equation*}
where, writing $b = (b_{11}, b_{22})$,
\begin{equation*}
\begin{split}
 \widetilde S_{\alpha_n,j,1,j}
&=\sum_{b\in\{0,1\}^2} \sum_{\substack{i_1\neq i_2}}^{\in I_{\alpha_{n}}}  \sum_{\substack{ k_1\neq l_1\\ \neq k_2\neq l_2}}^{\in I_{\alpha_{n+1} } \setminus\{ i_1,i_2\}} Z_{i_1k_1}Z_{i_1l_1}Z_{k_1l_1}Z_{i_2k_2}Z_{i_2l_2}Z_{k_2l_2}\1{D_{\alpha_n i_1}=j}\1{D_{\alpha_n i_2}=j}\\
&\quad\times\1{\sum_{k\in I_{\alpha_{n+1}}} Z_{i_1k}Z_{i_2k}=0  }\1{Z_{i_1 i_1}=b_{11}}\1{Z_{i_2 i_2}=b_{22}} \1{\sum_{i'\in I_n^c}Z_{i'i_1} =1} \1{\sum_{i'\in I_n^c}Z_{i'i_2} =1},
\end{split}
\end{equation*}
so that
$ E(\widetilde S_{\alpha_n,j,1,j})
  = E(\widetilde R_{nj}^{(1)})^2$.
We thus obtain that
$ \var(\widetilde R_{nj}^{(1)}) = O ( \alpha_n^{3+2\sigma-2a} \ell_{\sigma}(\alpha_n)^2) .$

\subsection{Proof of Lemma~\ref{lemma:asymptT}}
\label{sec:app:asymptT}

Let  $I_n^c = I_{\alpha_{n+1}}\backslash I_{\alpha_{n}}=\{i\mid \theta_i\in(\alpha_n,\alpha_{n+1}]\}$. First note that $\widetilde R_{nj}=\sum_{r=1}^j\widetilde R_{nr}^{(1)}+\widetilde R_{nr}^{(2)}+\widetilde R_{nr}^{(3)}$ where
\begin{align*}
\widetilde R_{nr}^{(1)}&=\sum_{i\in I_{\alpha_{n}} }T_{\alpha_{n+1} i}\1{D_{\alpha_n i}=r} \1{ \sum_{i'\in I_n^c} Z_{ii'}=1},\\
\widetilde R_{nr}^{(2)}&=\sum_{i\in I_{\alpha_{n+1}} }T_{\alpha_{n+1} i}\1{D_{\alpha_n i}=r} \1{ \sum_{i'\in I_n^c} Z_{ii'}\geq 2},\\
\widetilde R_{nr}^{(3)}&=\sum_{i\in  I_n^c}T_{\alpha_{n+1} i}\1{D_{\alpha_n i}=r} \1{ \sum_{i'\in I_n^c} Z_{ii'}=1}.
\end{align*}

For any $r\leq j$ 
\begin{equation*}
\begin{split}
E\left(\left.  \widetilde R_{nr}^{(2)} \right| M\right) \leq &\sum_{i\in I_{\alpha_{n+1}}}\sum_{l\neq k}^{\in I_{\alpha_{n}},\neq i} W(\vartheta_{i},\vartheta_{l})W(\vartheta_{i},\vartheta_{k})W(\vartheta_{l},\vartheta_{k})J_n(i, r-2) \Pr\left( \sum_{i'\in I_n^c}Z_{ii'}\geq 2|M\right) \\
& +2 \sum_{i  \in I_{\alpha_{n+1}}}\sum_{ l \neq i }^{\in I_{\alpha_{n}}} \sum_{ k\neq i}^{\in I_n^c} W(\vartheta_{i},\vartheta_{l})W(\vartheta_{i},\vartheta_{k})W(\vartheta_{l},\vartheta_{k})J_n(i, r-1)  \Pr\left( \sum_{i'\in I_n^c}Z_{ii'}\geq 1|M\right) \\
& + \sum_{i \in I_{\alpha_{n+1}}} \sum_{ l\neq k}^{\in I_n^c} W(\vartheta_{i},\vartheta_{l})W(\vartheta_{i},\vartheta_{k})W(\vartheta_{l},\vartheta_{k})J_n(i, r)
\end{split}
\end{equation*}
where, recalling the definition of $g_{\alpha,x}$ in Equation~\eqref{eq:defg},
$$J_n (i, r) = \sum_{i_{1}\neq i_{2}\ldots\neq i_{r}\neq l\neq k}^{\in  I_{\alpha_n}} \left[
\prod_{s=1}^{r}W(\vartheta_{i},\vartheta_{i_s})\right]  e^{-\sum_{s\neq
l,k,i_{1},\ldots,i_{r}}^{I_{\alpha_n}}g_{\alpha, \vartheta_{i}}(\theta_{s},\vartheta_{s})}.$$
Note that
\begin{align*}
 \Pr\left( \sum_{i'\in I_n^c}Z_{ii'}\geq 2|M\right) &= 1 - e^{-\sum_{s\in I_n^c}g_{\alpha_{n+1}-\alpha_n,\vartheta_{i}}(\theta_{s},\vartheta_{s})}- \sum_{i' \in I^c_n}W(\vartheta_i, \vartheta_{i'}) e^{-\sum_{s\neq i'}^{I_n^c}g_{\alpha_{n+1}-\alpha_n,\vartheta_{i}}(\theta_{s},\vartheta_{s})}. \end{align*}
Using the Slivnyak-Mecke theorem, the inequality $1-e^{-y}-ye^{-y}\leq y^2$ for $y\geq 0$, the condition \eqref{cond:localcluster} and Lemma \ref{lemma:abelianvariations2}, we obtain
\begin{align*}
E \left(  \widetilde R_{nr}^{(2)} \right) &\lesssim \alpha_{n+1}\alpha_n^{r}(\alpha_{n+1} - \alpha_n)^2 \int L_2(x) \mu(x)^{r+2}e^{-\alpha_n \mu(x)} dx\lesssim  \alpha_{n+1}\alpha_n^{\sigma-2}(\alpha_{n+1} - \alpha_n)^2\ell(\alpha_n) ,
\end{align*}
where $L_2(x)$ converges to $b\geq 0$ at infinity. Noting that
$
(\alpha_{n+1}-\alpha_n)/\alpha_n=O(1/n),
$
 we obtain
$
E \left(  \widetilde R_{nr}^{(2)} \right)  \lesssim  \alpha_n^{\sigma+1} \ell(\alpha_n) / n^2.
$
 This implies that
$\sum_n E \left(   \widetilde R_{nr}^{(2)} \right)/ (\alpha_n^{\sigma+1} \ell(\alpha_n) ) <+\infty$
so that, by Markov inequality and Borel-Cantelli lemma,
$\widetilde R_{nr}^{(2)}=o( \alpha_n^{\sigma+1} \ell(\alpha_n))$
almost surely as $n$ tends to infinity. \\We now study
$$  \widetilde R_{nr}^{(3)} := \sum_{i\in I_n^c }T_{\alpha_{n+1} i}\1{D_{\alpha_n i}=r} \1{ \sum_{i'\in I_n^c}Z_{ii'}=1}. $$
Similarly to before
\begin{equation*}
\begin{split}
E \left( \widetilde R_{nr}^{(3)} \mid M \right) & \leq\sum_{i\in I_n^c}  \sum_{ l\neq k}^{\in I_{\alpha_{n}}} W(\vartheta_{i},\vartheta_{l})W(\vartheta_{i},\vartheta_{k})W(\vartheta_{l},\vartheta_{k})J_n(i, r-2) \Pr\left( \sum_{i'\in I_n^c}Z_{ii'}=1|M\right) \\
& \quad+2\sum_{l \in I_n^c}^{l\neq i}  \sum_{ k \in I_{\alpha_{n}}} W(\vartheta_{i},\vartheta_{l})W(\vartheta_{i},\vartheta_{k})W(\vartheta_{l},\vartheta_{k})J_n(i, r-1)
\end{split}
\end{equation*}
 so that
 \begin{equation*}
\begin{split}
E\left( \widetilde R_{nr}^{(3)} \right)
& \lesssim (\alpha_{n+1} - \alpha_n)^2 \alpha_n^{r} \int L_3(x) \mu(x)^{r+1} e^{-\alpha_{n} \mu(x) } dx \lesssim   \frac{\alpha_n^{\sigma+1} \ell(\alpha_n)  }{ n^2 },
\end{split}
\end{equation*}
where $L_3(x)$ converges to $b$ and $\widetilde R_{nr}^{(3)}  = o(\alpha_n^{\sigma+1} \ell(\alpha_n)  )$ almost surely as $n$ tends to infinity.
Finally, we have
 \begin{equation*}
\begin{split}
E \left(\widetilde {R}_{nr}^{(1)} |M  \right) & \lesssim   \sum_{i\neq l\neq k}^{\in I_{\alpha_{n}}} W(\vartheta_{i},\vartheta_{l})W(\vartheta_{i},\vartheta_{k})W(\vartheta_{l},\vartheta_{k})J_n(i, r-2)\sum_{i' \in I_n^c} W(\vartheta_i, \vartheta_{i'}) \\
& \quad+2\sum_{i \neq l }^{\in I_{\alpha_{n}}} \sum_{ k \in I_n^c } W(\vartheta_{i},\vartheta_{l})W(\vartheta_{i},\vartheta_{k})W(\vartheta_{l},\vartheta_{k})J_n(i, r-1)
\end{split}
\end{equation*}
which implies that
 \begin{equation*}
\begin{split}
E \left(\frac{ \widetilde {R}^{(1)}_{nr}}{ \alpha_n^{\sigma+1} \ell_\sigma(\alpha_n)}  \right) & \quad \lesssim   \frac{ \alpha_n^{r+1}(\alpha_{n+1}-\alpha_n)}{ \alpha_n^{\sigma+1} \ell_\sigma(\alpha_n)} \int_{\mathbb R_+^{4}} L_1(x)\mu(x)^{ r+1}e^{-\alpha_n \mu(x)} dx \\
 & \quad = O\left(\frac{ \alpha_{n+1}-\alpha_n}{ \alpha_n}  \right) = o(1).
\end{split}
\end{equation*}
where $L_1(x)$ converges to $b$. Moreover, from Proposition \ref{prop:variancelocalclust} 
 $$ \var \left(\frac{ \widetilde {R}_{n r}^{(1)}}{ \alpha_n^{\sigma+1} \ell_\sigma(\alpha_n)}  \right) = O\left(  \alpha_n^{ 1-2a}  \right)
 $$
 so that, $\widetilde {R}_{nr}^{(1)}=o(\alpha_n^{1+\sigma}\ell(\alpha_n))$ almost surely. It finally follows that, for any $j\geq 1$,   $\widetilde R_{nj} = o(\alpha_n^{1+\sigma}\ell(\alpha_n))$ almost surely as $n$ tends to infinity.

\section{Proof of secondary propositions for the Central Limit Theorem}
\label{sec:proofsecondCLT}

\subsection{Proof of Proposition \ref{prop:normpart1}}%

Let
\[
Z_\alpha:=N_{\alpha}-E(N_{\alpha}\mid M)=\sum_{i}\1{\theta_{i}\leq\alpha
}(\1{D_{\alpha,i}\geq1}-(1-e^{-M(g_{\alpha,\vartheta_{i}})}))
\]

where we recall that $g_{\alpha,x}(\theta,\vartheta)    =-\log(1-W(x,\vartheta))\1{\theta\leq
\alpha}$ and
\begin{equation*}\\
e^{-M(g_{\alpha,\vartheta_{i}})}    =e^{-\sum_{j}-\log(1-W(\vartheta
_{i},\vartheta_{j}))\1{\theta_{j}\leq\alpha}}
  =\prod_{j}(1-W(\vartheta_{i},\vartheta_{j}))^{\1{\theta_{j}\leq\alpha}}%
\end{equation*}
We have $E(Z_{\alpha}\mid M)=0$ hence $\var(Z_\alpha)=E(Z_\alpha^2)$. Note that
\begin{align*}
Z_{\alpha}^{2}  &  =Z_{\alpha}+\sum_{i_{1}\neq i_{2}}\1{\theta_{i_{1}}%
\leq\alpha}(\1{D_{\alpha,i_{1}}\geq1}-(1-e^{-M(g_{\alpha,\vartheta_{i_{1}}}%
)}))\1{\theta_{i_{2}}\leq\alpha}(\1{D_{\alpha,i_{2}}\geq1}-(1-e^{-M(g_{\alpha
,\vartheta_{i_{2}}})}))\\
&  =Z_{\alpha}+\sum_{i_{1}\neq i_{2}}\1{\theta_{i_{1}}\leq\alpha}%
\1{D_{\alpha,i_{1}}\geq1}\1{\theta_{i_{2}}\leq\alpha}\1{D_{\alpha,i_{2}}\geq
1} -\sum_{i_{1}\neq i_{2}}\1{\theta_{i_{1}}\leq\alpha}\1{D_{\alpha,i_{1}}%
\geq1}\1{\theta_{i_{2}}\leq\alpha}(1-e^{-M(g_{\alpha,\vartheta_{i_{2}}})})\\
&  -\sum_{i_{1}\neq i_{2}}\1{\theta_{i_{1}}\leq\alpha}(1-e^{-M(g_{\alpha
,\vartheta_{i_{1}}})}))\1{\theta_{i_{2}}\leq\alpha}\1{D_{\alpha,i_{2}}\geq1}+\sum_{i_{1}\neq i_{2}}\1{\theta_{i_{1}}\leq\alpha}(1-e^{-M(g_{\alpha
,\vartheta_{i_{1}}})})\1{\theta_{i_{2}}\leq\alpha}(1-e^{-M(g_{\alpha
,\vartheta_{i_{2}}})})
\end{align*}
We have
\begin{align*}
E\left(  Z_{\alpha}^{2}\mid M\right)
&  =\sum_{i_{1}\neq i_{2}}\1{\theta_{i_{1}}\leq\alpha}\1{\theta_{i_{2}}%
\leq\alpha}e^{-M(g_{\alpha,\vartheta_{i_{1}}})-M(g_{\alpha,\vartheta_{i_{2}}%
})}(e^{g_{\alpha,\vartheta_{i_{1}}}(\theta_{i_{2}},\vartheta_{i_{2}})}-1)
\end{align*}
Applying the extended Slivnyak-Mecke theorem%
\begin{align*}
E\left(  Z_{\alpha}^{2}\right)   &  =\alpha^{2}\int_{\mathbb{R}%
_{+}^{2}}(1-W(x,x))(1-W(y,y))(1-W(x,y))e^{-\alpha\mu(x)-\alpha\mu(y)+\alpha
\nu(x,y)}(1-(1-W(x,y)))dxdy\\
&  \leq\alpha^{2}\int_{\mathbb{R}_{+}^{2}}W(x,y)e^{-\alpha\mu(x)-\alpha
\mu(y)+\alpha\nu(x,y)}dxdy.
\end{align*}
Using Cauchy-Schwarz inequality, $\nu(x,y)\leq\sqrt{\mu(x)\mu(y)}\leq
\frac{1}{2}(\mu(x)+\mu(y))$, and Lemma \ref{lemma:abelianvariations}, we obtain
\begin{align*}
E\left(  Z_{\alpha}^{2}\right)
&  =\alpha^{2}\int_{0}^{\infty}\mu(x)e^{-\alpha/2\mu(x)}dx\asymp
\alpha^{1+\sigma}\ell(\alpha)
=\left \{
\begin{array}{ll}
  O(\alpha^{1+\sigma}\ell_\sigma(\alpha)) & \sigma\in[0,1)\\
  o(\alpha\ell(\alpha)) & \sigma=0
\end{array}\right. .
\end{align*}

It follows from Markov's inequality that, in probability
$$
Z_{\alpha}=\left \{
\begin{array}{ll}
  O(\alpha^{1/2+\sigma/2}\ell^{1/2}_\sigma(\alpha)) & \sigma\in[0,1)\\
  o(\alpha^{1/2}\ell^{1/2}(\alpha)) & \sigma=0
\end{array}\right. .
$$

\subsection{Proof of Proposition \ref{prop:normpart2}}

Define
$
M(h_{\alpha})=\sum_{i}\widetilde{Z}_{i}%
$
where
$
\widetilde{Z}_{i}=h_{\alpha}(\theta_{i},\vartheta_{i})=\1{\theta_{i}\leq
\alpha}\left[  1-(1-W(\vartheta_{i},\vartheta_{i}))e^{-\alpha\mu(\vartheta_{i}%
)}\right]  .
$
Using Campbell's formula%
\begin{align*}
E\left(  \sum_{i}\widetilde{Z}_{i}\right)   &  =\alpha\int%
_{0}^{\infty}(1-(1-W(x,x))e^{-\alpha\mu(x)})dx=E(N_{\alpha})\\
\var\left(  \sum_{i}\widetilde{Z}_{i}\right)   &  =\alpha\int_{0}^{\infty
}\left[  1-(1-W(x,x))e^{-\alpha\mu(x)}\right]  ^{2}dx\leq E(N_{\alpha
})
\end{align*}
Noting that $E(N_{\alpha
})\sim\alpha^{1+\sigma}\Gamma(1-\sigma)\ell(\alpha)$, it follows from Chebyshev's inequality that, in probability,
\begin{align*}
\sum_{i}\widetilde{Z}_{i}-E(N_{\alpha})  &  =O\left(
\sqrt{\var\left(  \sum_{i}\widetilde{Z}_{i}\right)  }\right) =O\left(  \alpha^{1/2+\sigma/2}\ell_\sigma^{1/2}(\alpha)\right)
\end{align*}
If $\mu$ has an unbounded support then, under Assumption \ref{assumpt:1}, either $\sigma>0$ or $\sigma=0$ and $\ell(t)\to\infty$. In both cases, in probability,
$\sum_{i}\widetilde{Z}_{i}-E(N_{\alpha}) =o\left(  \alpha^{1/2+\sigma}\ell_\sigma(\alpha)\right).$

\subsection{Proof of Proposition \ref{prop:normpart3}}
\label{sec:proofCLTsparse}

Let
\begin{align*}
f_{\alpha}(M)
&  =\sum_{i}\1{\theta_{i}\leq\alpha}\left[  (1-W(\vartheta_{i},\vartheta
_{i}))e^{-\alpha\mu(\vartheta_{i})}-e^{-M(g_{\alpha,\vartheta_{i}})}\right]
\end{align*}

The idea is to use Theorem 1.1 from \cite{Last2016}. To do so, define
\begin{equation}
F_\alpha=\frac{f_{\alpha}(M)}{\sqrt{v_{\alpha}}}
\end{equation} where $v_{\alpha}= \var(f_{\alpha}(M))\sim \var(N_{\alpha} )\asymp\alpha^{1+2\sigma}\ell_\sigma^{2}(\alpha)$. Note that $E(F_\alpha)=0$ and $\var(F_\alpha)=1$. Consider the difference operator $D_{z}F_\alpha$ defined by%
\[
D_{z}F_\alpha=\frac{1}{\sqrt{v_\alpha}}( f_\alpha(M+\delta_{z})-f_\alpha(M))
\]
Also%
\begin{align*}
D_{z_1,z_2}^{2}F_\alpha  &  =D_{z_2}(D_{z_1}F_\alpha)=D_{z_2}\left(\frac{1}{\sqrt{v_\alpha}} (f_\alpha(M+\delta_{z_1})-f_\alpha(M))\right )\\
&  =\frac{1}{\sqrt{v_\alpha}}\left( f_\alpha(M+\delta_{z_1}+\delta_{z_2})- f_\alpha(M+\delta_{z_1})- f_\alpha(M+\delta_{z_2}
)+ f_\alpha(M)\right ).
\end{align*}
Define
\begin{align*}
\gamma_{\alpha,1}  &  :=2\left(  \int_{\mathbb{R}_{+}^{6}}\sqrt{\mathbb{E}(D_{z_{1}%
}F_\alpha)^{2}(D_{z_{2}}F_\alpha)^{2}}\sqrt{\mathbb{E}(D_{z_{1},z_{3}}^{2}F_\alpha)^{2}%
(D_{z_{2},z_{3}}^{2}F_\alpha)^{2}}dz_{1}dz_{2}dz_{3}\right)  ^{1/2}\\
\gamma_{\alpha,2}  &  :=\left(  \int_{\mathbb{R}_{+}^{6}}\mathbb{E}\left [(D_{z_{1},z_{3}%
}^{2}F_\alpha)^{2}(D_{z_{2},z_{3}}^{2}F_\alpha)^{2}\right ]dz_{1}dz_{2}dz_{3}\right)
^{1/2}\\
\gamma_{\alpha,3}  &  :=\int_{\mathbb{R}_{+}^{2}}\mathbb{E}|D_{z}F_\alpha|^{3}dz
\end{align*}

We state a corollary of Theorem 1.1 from \cite{Last2016}.
\begin{corollary}{\cite[Theorem 1.1]{Last2016}}
If $\gamma_{\alpha,1},\gamma_{\alpha,2},\gamma_{\alpha,3}\rightarrow0,$ then
$$F_\alpha=\frac{f_{\alpha}(M)}{\sqrt{v_{\alpha}}}\to\mathcal N(0,1).$$
\end{corollary}
The rest of the proof aims to show that $\gamma_{\alpha,1},\gamma_{\alpha,2},\gamma_{\alpha,3}\rightarrow0$. The proof is rather lengthy and therefore split in different subsections. We first state a few notations and lemmas that will be useful in the following.

\subsubsection{Definitions and lemmas}

The following lemma, obtained with H\"older's inequality, will be used multiple times.
\begin{lemma}\label{lemma:mu}
For any $d\geq 1$ and any $z_1,\ldots,z_d>0$,
\begin{align*}
E\left ( \prod_{k=1}^d e^{-M(g_{\alpha,z_k})}  \right )&\leq e^{-\frac{\alpha}{d}\sum_{k=1}^d \mu(z_k)}.
\end{align*}
\end{lemma}
\begin{proof}
Using H\"older's inequality, for any $d\geq 1$
\begin{align*}
E\left ( \prod_{k=1}^d e^{-M(g_{\alpha,z_k})}  \right )&\leq \prod_{k=1}^d E\left (  e^{-d M(g_{\alpha,z_k})}  \right )^{1/d}= \prod_{k=1}^d e^{-\frac{\alpha}{d}\int_0^\infty (1-[1-W(z_k, y)]^d) dy}\\
&\leq \prod_{k=1}^d e^{-\frac{\alpha}{d}\int_0^\infty (1-[1-W(z_k, y)]) dy}=e^{-\frac{\alpha}{d}\sum_{k=1}^d \mu(z_k)}
\end{align*}
\end{proof}

For $i,j\geq 0$, let
\begin{align}
H_{i,j}(x_1,x_2)&=\int_{\mathbb R_+^2}W(x_1,y)^iW(x_2,y)^je^{-\frac{\alpha}{4}\mu(y)}dy , \quad H_{i}(x)  = H_{i,0}(x,x). \label{eq:H1}
\end{align}
The following lemma compiles various useful bounds.
\begin{lemma}\label{lemma:bound} Assume Assumptions 1 and 5. Then
\begin{itemize}
\item For all $j\geq 1$ and all $x_1,\ldots,x_{j-1}>0$, $y_1>0$ and $p_1,\ldots,p_{j}\geq 1$, $\alpha>0$
\begin{align*}
&\int_0^\infty \left (\int_0^\infty W(y_1,x_j) ^{p_j}\left [\prod_{k=1}^{j-1} W(y_1,x_k)^{p_k}\right ]e^{-\alpha\mu(y_1) } dy_1 \right )^{1/2}W(y_2,x_j)dx_j\\
&\leq L(y_2)\mu(y_2) \left (\int L(y_1)^{p_j}\mu(y_1)^{p_j} \left [\prod_{k=1}^{j-1} W(y_1,x_k)^{p_k}\right ]e^{-\alpha\mu(y_1) } dy_1 \right )^{1/2}
\end{align*}
\item For any $x_2>0$,
\begin{align}
\int H_{1,1}(x_1,x_2)^2d x_1
&\leq \int L(y_1)L(y_2)\mu(y_1)\mu(y_2)W(x_2,y_1)W(x_2,y_2)e^{-\frac{\alpha}{4}(\mu(y_1)+\mu(y_2))}dy_1dy_2\label{eq:boundH2square}
\end{align}
\item For any $q=1,2,\ldots$ \begin{align}
\int_0^\infty H_1(x)^{q}dx&=\int \prod_{i=1}^q \int W(x,y_i) e^{-\frac{\alpha}{4} \mu(y_i)}dy_i dx
&=O(\alpha^{q\sigma-q}\ell_\sigma(\alpha)^q)\label{eq:boundintH1}
\end{align}
\item   For any $q\geq 1$, $p\geq 1$,
\begin{align}
\int (\alpha H_1(x_1))^{q/2}L(x_1)\mu(x_1)^{p/2} e^{-\frac{\alpha}{4}\mu(x_1)}dx_1
&=O(\alpha^{(q+1)\sigma/2 -p/2 }\ell_\sigma(\alpha)^{(q+1)/2})\label{eq:boundH1a}
\end{align}
\item If $q\leq 3$,
\begin{align}
\int L(x_1)\mu(x_1)^p W(x_1,x_2)^q e^{-\frac{\alpha}{4}\mu(x_1)}dx_1&\leq \mu(x_2)^q L(x_2)^q\left (\int L(x_1)^2\mu(x_1)^{2p}  e^{-\frac{\alpha}{2}\mu(x_1)}dx_1\right)^2 .\label{eq:boundH1c}
\end{align}
\end{itemize}
\end{lemma}

\begin{proof}
The first inequality comes from H\"older's inequality, together with Assumptions \ref{assumpt:5} and \ref{assumpt:1}.
Also \eqref{eq:boundH2square} is a consequence of
$$ \int H_{1,1}(x_1,x_2)^2d x_1 =\int \nu(y_1,y_2)W(x_2,y_1)W(x_2,y_2)e^{-\frac{\alpha}{4}(\mu(y_1)+\mu(y_2))}dy_1dy_2$$

Under Assumption \ref{assumpt:5},
For any $q=1,2,\ldots$, using Assumption \ref{assumpt:1}, \ref{assumpt:5}, and Lemma \ref{lemma:abelianvariations2},
\begin{align*}
\int_0^\infty H_1(x)^{q}dx&=\int \prod_{i=1}^q \int W(x,y_i) e^{-\frac{\alpha}{4} \mu(y_i)}dy_i dx\nonumber \\
&\leq \prod_{i=1}^q \int L(y_i)\mu(y_i) e^{-\frac{\alpha}{4} \mu(y_i)}dy_i\nonumber=O(\alpha^{q\sigma-q}\ell_\sigma(\alpha)^q).
\end{align*}
Using H\"older's inequality,
\begin{align*}
\int (\alpha H_1(x_1))^{q/2}L(x_1)\mu(x_1)^{p/2} e^{-\frac{\alpha}{4}\mu(x_1)}dx_1&\leq \left (\int (\alpha H_1(x_1))^{q}dx_1 \int L(x_1)^2\mu(x_1)^{p} e^{-\frac{\alpha}{2}\mu(x_1)}dx_1    \right )^{1/2}\nonumber\\
&=O(\alpha^{(q+1)\sigma/2 -p/2 }\ell_\sigma(\alpha)^{(q+1)/2})
\end{align*}
which proves \eqref{eq:boundH1a}. Finally
recall that from Lemma \ref{lemma:abelianvariations2} that for any $p\geq 1$,
\begin{align*}
\int L(x_1)\mu(x_1)^pe^{-\frac{\alpha}{4}\mu(x_1)}dx_1&=O(\alpha^{\sigma -p }\ell_\sigma(\alpha)).
\end{align*}
so that using H\"older and Assumption \ref{assumpt:5}, for any $q\leq 3$
\begin{align*}
\int L(x_1)\mu(x_1)^p W(x_1,x_2)^q e^{-\frac{\alpha}{4}\mu(x_1)}dx_1&\leq \mu(x_2)^q L(x_2)^q\left (\int L(x_1)^2\mu(x_1)^{2p}  e^{-\frac{\alpha}{2}\mu(x_1)}dx_1\right)^2
\end{align*}

\end{proof}

\subsubsection{General bounds}

Let
$z=(t,x)$. Recall that $g_{\alpha,x}(\theta,\vartheta)=-\log(1-W(x,\vartheta))\1{\theta%
\leq\alpha}$.%
\begin{align*}
\sqrt{v_{\alpha}}\times D_{z}F_\alpha
 &  =\1{t\leq\alpha}(1-W(x,x))\left[  e^{-\alpha\mu(x)}-e^{-M(g_{\alpha,x})}\right]   +\1{t\leq\alpha}\sum_{i}\1{\theta_{i}\leq\alpha}W(\vartheta_{i},x)e^{-M(g_{\alpha ,\vartheta_{i}})}.%
\end{align*}
We have
\begin{align}
\sqrt{v_\alpha}\left\vert D_{z}F_\alpha\right\vert
& \leq \1{t\leq\alpha}\left( \left|e^{-\alpha\mu(x)}-e^{-M(g_{\alpha,x})}\right| +\sum_{i}\1{\theta_{i}\leq\alpha}W(\vartheta_{i},x)e^{-M(g_{\alpha,\vartheta_{i}})}  \right)\label{eq:Dz}
\end{align}

Similarly,
\begin{align*}
\sqrt{v_\alpha}D^2_{z_1,z_2}(F_{\alpha})
&  =\1{t_1, t_2\leq\alpha} W(x_1,x_2)\left[  (1-W(x_1,x_1)) e^{-M(g_{\alpha,x_1})}
 + (1-W(x_2,x_2))e^{-M(g_{\alpha,x_2})}\right] \\
&  -\1{t_1, t_2\leq\alpha}\sum_{i}\1{\theta_{i}\leq\alpha} W(\vartheta_i,x_1)W(\vartheta_i,x_2) e^{-M(g_{\alpha,\vartheta_{i}}) }
\end{align*}

Note that the above is equal to 0 if $t_1>\alpha$ or $t_2>\alpha$. For $t_1,t_2\leq \alpha$
\begin{align*}
|\sqrt{v_{\alpha}}\times D^2_{z_1,z_2}F_\alpha|^2&\leq 2
W(x_1,x_2)^2(e^{-M(g_{\alpha,x_1})}+e^{-M(g_{\alpha,x_2})})^2 \\
&  +2\left (\sum_{i}\1{\theta_{i}\leq\alpha} W(\vartheta_i,x_1)W(\vartheta_i,x_2) e^{-M(g_{\alpha,\vartheta_{i}})}\right )^2\\
&\leq 4 W(x_1,x_2)^2(e^{-M(g_{\alpha,x_1})}+e^{-M(g_{\alpha,x_2})})+2\sum_{i}\1{\theta_{i}\leq\alpha} W(\vartheta_i,x_1)^2 W(\vartheta_i,x_2)^2 e^{-2M(g_{\alpha,\vartheta_{i}})}\\
&+2\sum_{i\neq j}\1{\theta_{i}\leq\alpha}\1{\theta_{j}\leq\alpha} W(\vartheta_i,x_1)W(\vartheta_i,x_2)W(\vartheta_j,x_1)W(\vartheta_j,x_2) e^{-M(g_{\alpha,\vartheta_{i}})-M(g_{\alpha,\vartheta_{j}})}
\end{align*}
\begin{align*}
&v_\alpha^2(D_{z_{1},z_{3}%
}^{2}F_\alpha)^{2}(D_{z_{2},z_{3}}^{2}F_\alpha)^{2}\nonumber\\
&\leq 16 W(x_1,x_3)^2W(x_2,x_3)^2(e^{-M(g_{\alpha,x_1})}+e^{-M(g_{\alpha,x_3})})(e^{-M(g_{\alpha,x_2})}+e^{-M(g_{\alpha,x_3})})\nonumber\\
&\quad+8W(x_1,x_3)^2(e^{-M(g_{\alpha,x_1})}+e^{-M(g_{\alpha,x_3})})\left (\sum_{i}\1{\theta_{i}\leq\alpha} W(\vartheta_i,x_2)W(\vartheta_i,x_3) e^{-M(g_{\alpha,\vartheta_{i}})}\right )^2\nonumber\\
&\quad+8W(x_2,x_3)^2(e^{-M(g_{\alpha,x_2})}+e^{-M(g_{\alpha,x_3})})\left (\sum_{i}\1{\theta_{i}\leq\alpha} W(\vartheta_i,x_1)W(\vartheta_i,x_3) e^{-M(g_{\alpha,\vartheta_{i}})}\right )^2\nonumber\\
&\quad+4\sum_{i_1,i_2,i_3,i_4}\1{\theta_{i_1}\leq\alpha}\1{\theta_{i_2}\leq\alpha}\1{\theta_{i_3}\leq\alpha}\1{\theta_{i_4}\leq\alpha} W(\vartheta_{i_1},x_1)W(\vartheta_{i_1},x_3)W(\vartheta_{i_2},x_1)W(\vartheta_{i_2},x_3)\nonumber\\
&\qquad \quad\quad\times W(\vartheta_{i_3},x_2)W(\vartheta_{i_3},x_3)W(\vartheta_{i_4},x_2)W(\vartheta_{i_4},x_3)  e^{-\sum_{k=1}^4 M(g_{\alpha,\vartheta_{i_k}})}
\end{align*}

We obtain, using the inequality~\eqref{eq:boundH2square}
\begin{align}
&E\left (v_\alpha^2(D_{z_{1},z_{3}%
}^{2}F)^{2}(D_{z_{2},z_{3}}^{2}F)^{2}\right ) \nonumber\\
&\leq C\times\left ( W(x_1,x_3)^2W(x_2,x_3)^2(e^{-\alpha/2\mu(x_1))}+e^{-\alpha/2\mu(x_3)})(e^{-\alpha/2\mu(x_2)}+e^{-\alpha/2 \mu(x_3)})\right .\nonumber\\
&+(\alpha^2H_{1,1}(x_2,x_3)^2+\alpha H_{2,2}(x_2,x_3)) W(x_1,x_3)^2(e^{-\alpha/3 \mu(x_1)}+e^{-\alpha/3 \mu(x_3))})\nonumber\\
&+(\alpha^2H_{1,1}(x_1,x_3)^2+\alpha H_{2,2}(x_1,x_3)) W(x_2,x_3)^2(e^{-\alpha/3 \mu(x_2)}+e^{-\alpha/3 \mu(x_3))})\nonumber\\
&\left .+A_2(x_1,x_2,x_3)\right )\label{eq:longinequality}
\end{align}
for some constant $C>0$, where
\bgroup
\allowdisplaybreaks
\begin{align*}
&A_2(x_1,x_2,x_3)\\
&=E\left (\sum_{i_1,i_2,i_3,i_4}\prod_{\ell = 1}^4W(\vartheta_{i_\ell },x_3)\1{\theta_{i_\ell}\leq\alpha} W(\vartheta_{i_1},x_1)W(\vartheta_{i_2},x_1)
 W(\vartheta_{i_3},x_2)W(\vartheta_{i_4},x_2) e^{-\sum_{k=1}^4 M(g_{\alpha,\vartheta_{i_k}})}\right )\\
&=\alpha^4 \int \prod_{\ell = 1}^4W(y_\ell,x_3)  W(y_1,x_1)W(y_2,x_1)W(y_2,x_3)W(y_3,x_2)W(y_4,x_2) e^{-\alpha/4 (\sum_{i=1}^4\mu(y_i))}dy_{1:4} \\
&\quad+\alpha^3 \int  (W(y_1,x_1)^2W(y_1,x_3)^2W(y_2,x_2)W(y_2,x_3)W(y_3,x_2)W(y_3,x_3)\\
&\qquad \quad\quad +W(y_1,x_2)^2W(y_1,x_3)^2W(y_2,x_1)W(y_2,x_3)W(y_3,x_1)W(y_3,x_3)   ) e^{-\alpha/3 (\sum_{i=1}^3\mu(y_i))}dy_{1:3}\\
&\quad+\alpha^2 \int  (W(y_1,x_1)^2W(y_1,x_3)^2W(y_2,x_2)^2W(y_2,x_3)^2\\
&\qquad \quad\quad +W(y_1,x_2)W(y_1,x_1)W(y_1,x_3)^2W(y_2,x_1)W(y_2,x_2)W(y_2,x_3)^2   ) e^{-\alpha/2 (\sum_{i=1}^2\mu(y_i))}dy_{1:2}\\
&\quad\left .+\alpha \int  W(y_1,x_1)^2W(y_1,x_2)^2W(y_1,x_3)^4  e^{-\alpha \mu(y_1)}dy_1\right )
\end{align*}
\egroup

\subsubsection{Proof that $\gamma_{\alpha,3}\rightarrow 0$}
We show here that $\gamma_{\alpha,3}\rightarrow 0$, or equivalently $\int E \left\vert \sqrt{v_\alpha} D_{z}F\right\vert^3dz=o(\alpha^{3/2+3\sigma}\ell_\sigma^{3}(\alpha))$. From Equation~\eqref{eq:Dz} and using the inequality $(a+b)^3\leq 4(a^3+b^3)$ for any $a,b\geq 0$, a sufficient condition is
$$
\int_0^\infty  E \left [\left |e^{-\alpha\mu(x)}-e^{-M(g_{\alpha,x})}\right |^3 + \left (\sum_{i}\1{\theta_{i}\leq\alpha}W(\vartheta_{i},x)e^{-M(g_{\alpha,\vartheta_{i}})}\right )^3\right ]dx=o(\alpha^{1/2+3\sigma}\ell_\sigma^{3}(\alpha)).
$$

We have
\begin{align*}
\int_0^\infty \mathbb E \left [\left |e^{-\alpha\mu(x)}-e^{-M(g_{\alpha,x})}\right |^3 \right] dx
 &\leq 2
 \int E\left (\left (e^{-\alpha\mu(x)}-e^{-M(g_{\alpha,x})}\right )^2\right)dx \\
&\leq \int (1-e^{-2\alpha\mu(x)})dx=O(\alpha^\sigma\ell_\sigma(\alpha))\\
\end{align*}
Also  under Assumptions \ref{assumpt:1} and \ref{assumpt:5}, using Lemma \ref{lemma:mu}

\begin{align*}
&\int_0^\infty  E\left [ \left(\sum_{i}\1{\theta_{i}\leq\alpha}W(\vartheta_{i},x)e^{-M(g_{\alpha,\vartheta_{i}})}\right)^3 \right ]dx
  \leq
 \int_0^\infty  E \left[ \sum_{i_1,i_2,i_3}  \prod_{\ell=1}^3 \1{\theta_{i_\ell}\leq\alpha}W(\vartheta_{i_\ell},x)e^{-M(g_{\alpha,\vartheta_{i_\ell}}) } \right] dx \\
& \leq \alpha^3 \left( \int L(y) \mu(y) e^{-\alpha \mu(y)/3} dy \right)^3 + 3 \alpha^2 \left( \int L(y)^2 \mu(y)^2 e^{-\alpha \mu(y)/3} dy \right)\left( \int L(y) \mu(y) e^{-\alpha \mu(y)/3} dy \right)\\
 & +\alpha \int L(y)^3 \mu(y)^3 e^{-\alpha \mu(y)/3} dy  = O( \alpha^{ 3\sigma} \ell_\sigma(\alpha)^3 ) .
\end{align*}

It follows that $\gamma_{\alpha,3}\to 0$ as $\alpha\to\infty$.

\subsubsection{Proof that $\gamma_{\alpha,2}\rightarrow 0$}
We now need to show that the integral of the right hand-side of Equation \eqref{eq:longinequality} with respect to $x_1,x_2,x_3$ is $o(\alpha^{-3}v_\alpha^2)=o(\alpha^{-1+4\sigma}\ell_\sigma^2(\alpha))$.
For the first term in the right hand-side of the inequality \eqref{eq:longinequality}, we have
\begin{align*}
&\int W(x_1,x_3)^2W(x_2,x_3)^2(e^{-\frac{\alpha}{2}\mu(x_1))}+e^{-\frac{\alpha}{2}\mu(x_3)})(e^{-\frac{\alpha}{2}\mu(x_2)}+e^{-\frac{\alpha}{2} \mu(x_3)})dx_1dx_2dx_3\\
&\leq 3\int W(x_1,x_3)W(x_2,x_3)(e^{-\frac{\alpha}{2}(\mu(x_1)+\mu(x_2))}+e^{-\frac{\alpha}{2}\mu(x_3)})dx_1dx_2dx_3\\
&\leq 3\int \nu(x_1,x_2)e^{-\frac{\alpha}{2}(\mu(x_1)+\mu(x_2))}dx_1dx_2+3\int \mu(x_3)^2 e^{-\frac{\alpha}{2}\mu(x_3)}dx_3 \\
&=O(\alpha^{2\sigma-2}\ell_\sigma^2(\alpha))+O(\alpha^{\sigma-2}\ell_\sigma(\alpha))\end{align*}

For the second line (and similarly for the third line) in the RHS of Equation \eqref{eq:longinequality}, we have, noting that $H_{2,2}(x_2,x_3)\leq H_{1,1}(x_2,x_3)$
\begin{align*}
&\int (\alpha^2H_{1,1}(x_2,x_3)^2+\alpha H_{1,1}(x_2,x_3)) W(x_1,x_3)^2(e^{-\alpha/3 \mu(x_1)}+e^{-\alpha/3 \mu(x_3))})dx_1dx_2dx_3\\
&\leq \alpha^2 \int W(x_1,x_3)^2 W(y_1,x_2)W(y_1,x_3)W(y_2,x_2)W(y_2,x_3)\\
&\times(e^{-\alpha/3(\mu(y_1)+\mu(y_2)+\mu(x_1))}+e^{-\alpha/3(\mu(y_1)+\mu(y_2)+\mu(x_3))})dx_1dx_2dx_3dy_1dy_2\\
&+\alpha \int W(x_1,x_3)^2 W(y_1,x_2)W(y_1,x_3)(e^{-\alpha/3(\mu(y_1)+\mu(x_1))}+e^{-\alpha/3(\mu(y_1)+\mu(x_3))})dx_1dx_2dx_3dy_1\\
&= \alpha^2 \int W(x_1,x_3)^2 \nu(y_1,y_2)W(y_1,x_3)W(y_2,x_3)\\
&\times(e^{-\alpha/3(\mu(y_1)+\mu(y_2)+\mu(x_1))}+e^{-\alpha/3(\mu(y_1)+\mu(y_2)+\mu(x_3))})dx_1dx_3dy_1dy_2\\
&+\alpha \int W(x_1,x_3)^2 \mu(y_1)W(y_1,x_3)(e^{-\alpha/3(\mu(y_1)+\mu(x_1))}+e^{-\alpha/3(\mu(y_1)+\mu(x_3))})dx_1dx_3dy_1\\
&\leq 2\alpha^2 \int  L(x_1)^2 L(y_1)L(y_2)\mu(x_1)^2\mu(y_1)\mu(y_2)e^{-\alpha/3(\mu(y_1)+\mu(y_2)+\mu(x_1))}dx_1dy_1dy_2\\
&+\alpha \int  \mu(y_1)^2 L(y_1)L(x_1)^2\mu(x_1)^2e^{-\alpha/3(\mu(y_1)+\mu(x_1))}dx_1dy_1+\alpha \int \mu(x_3)^2 L(x_3)^2 \mu(y_1)e^{-\alpha/3(\mu(y_1)+\mu(x_3))}dx_3dy_1\\
&=O(\alpha^{3\sigma-2}\ell_\sigma(\alpha)^3)
\end{align*}
using Assumption \ref{assumpt:5} and Lemma \ref{lemma:abelianvariations2}. For the third term in the right-handside of Equation \eqref{eq:longinequality}, we obtain
\begin{align*}
&\int A_2(x_1,x_2,x_3)dx_1dx_2dx_3\\ &\leq \alpha^4 \left(\int  L(y)^2\mu(y)^2 e^{-\alpha/4\mu(y)}dy\right)^4 +\alpha^3 \int  L(y_1)^4\mu(y_1)^4 e^{-\alpha/3\mu(y_1)}dy_1 \left(\int  L(y)^2\mu(y)^2 e^{-\alpha/3\mu(y)}dy\right)^2\\
&\quad+\alpha^2\left(\int  L(y)^4\mu(y)^4 e^{-\alpha/2\mu(y)}dy\right)^2+\alpha \int  L(y)^8\mu(y)^8 e^{-\alpha\mu(y)}dy=O(\alpha^{4\sigma-4})\ell_\sigma^4(\alpha)
\end{align*}

It follows that $\gamma_{\alpha,2}\to 0$ as $\alpha\to\infty$.

\subsubsection{Proof that $\gamma_{\alpha,1}\rightarrow 0$}

For any $x>0$ and any unit-rate Poisson point measure $M$ on $\mathbb R_+^2$, denote
\begin{align}
r_\alpha(x,M)=e^{-\alpha\mu(x)}+e^{-M(g_{\alpha,x})}
\end{align}

For any $z_1=(t_1,x_1),z_2=(t_2,x_2)$, if $t_1>\alpha$ or $t_2>\alpha$,  then $\left\vert D_{z_1}(F_{\alpha})\right\vert^2\left\vert D_{z_2}(F_{\alpha})\right\vert^2=0$. Otherwise, if $t_1,t_2\leq\alpha$, we have from Equation~\eqref{eq:Dz},
\begin{align*}
&v_\alpha\left\vert D_{z_1}(F_{\alpha})\right\vert^2\left\vert D_{z_2}(F_{\alpha})\right\vert^2\\  &  \leq
\left (4r_\alpha(x_1,M)+2 \sum_{i,j}\1{\theta_{i}\leq\alpha}\1{\theta_{j}\leq\alpha}W(\vartheta_{i},x_1)W(\vartheta_{j},x_1)e^{-M(g_{\alpha,\vartheta_{i}})-M(g_{\alpha,\vartheta_{j}})}\right )
\\
&\quad \times\left (4r_\alpha(x_2,M)+2\sum_{i,j}\1{\theta_{i}\leq\alpha}\1{\theta_{j}\leq\alpha}W(\vartheta_{i},x_2)W(\vartheta_{j},x_2)e^{-M(g_{\alpha,\vartheta_{i}})-M(g_{\alpha,\vartheta_{j}})}\right )\\
&\leq 16r_\alpha(x_1,M)r_\alpha(x_2,M) \\
&+8\sum_{i,j}\1{\theta_{i},\theta_{j}\leq\alpha}\left \{ W(\vartheta_{i},x_1)W(\vartheta_{j},x_1)r_\alpha(x_2,M)+W(\vartheta_{i},x_2)W(\vartheta_{j},x_2)r_\alpha(x_1,M)\right) e^{-M(g_{\alpha,\vartheta_{i}})-M(g_{\alpha,\vartheta_{j}})}\\
&+4 \sum_{i_1,i_2,i_3,i_4} W(\vartheta_{i_1},x_1)W(\vartheta_{i_2},x_1)W(\vartheta_{i_3},x_2)W(\vartheta_{i_4},x_2)\prod_{k=1}^4 \left ( \1{\theta_{i_k}\leq\alpha}e^{-M(g_{\alpha,\vartheta_{i_k}})} \right ).
\end{align*}
Note that, using Campbell theorem,  together with Lemma \ref{lemma:mu}
\begin{equation*}
\begin{split}
E(r_\alpha(x_1,M)r_\alpha(x_2,M)) \leq e^{-\alpha (\mu(x_1)+\mu(x_2))/2}
\end{split}
\end{equation*}
 It follows that, using the extended Slivnyak-Mecke theorem
\begin{align*}
&v_\alpha^2 E\left (\left\vert D_{z_1}(F_{\alpha})\right\vert^2\left\vert D_{z_2}(F_{\alpha})\right\vert^2\right)\\&\leq C\left ( e^{-\frac{\alpha}{2}(\mu(x_1)+\mu(x_2))}+\alpha\int_0^\infty (W(y,x_1)e^{-\frac{\alpha}{2}\mu(x_2)}+W(y,x_2)e^{-\frac{\alpha}{2}\mu(x_1)})e^{-\frac{\alpha}{2}\mu(y)}dy\right .\\
&+\alpha^2\int_{\mathbb R_+^2}\Big \{(W(y_1,x_1)W(y_2,x_1)e^{-\frac{\alpha}{3}\mu(x_2)}+W(y_1,x_2)W(y_2,x_2)e^{-\frac{\alpha}{3}\mu(x_1)}\Big \}e^{-\alpha\mu(y_1)/3-\alpha\mu(y_2)/3}dy_1dy_2\\
&+\alpha^3\int_{\mathbb R_+^3}(W(y_1,x_1)^2W(y_2,x_2)W(y_3,x_2)+W(y_1,x_2)^2W(y_2,x_1)W(y_3,x_1))e^{-\alpha\sum_{k=1}^3\mu(y_k)/3}dy_1dy_2dy_3\\
&\left .+\alpha^4 \int_{\mathbb R_+^4} W(y_1,x_1)W(y_2,x_1)W(y_3,x_2)W(y_4,x_2)e^{-\alpha\sum_{k=1}^4\mu(y_k)/4}dy_1dy_2dy_3dy_4 \right )\\
&\leq C\left ( e^{-\frac{\alpha}{3}(\mu(x_1)+\mu(x_2))}+\alpha(H_1(x_1)e^{-\frac{\alpha}{3}\mu(x_2)}+H_1(x_2)e^{-\frac{\alpha}{3}\mu(x_1)})+\alpha^2 (H_1(x_1)^2e^{-\frac{\alpha}{3}\mu(x_2)}+H_1(x_2)^2e^{-\frac{\alpha}{3}\mu(x_1)}) \right. \\
 & \left. \quad +\alpha^3(H_1(x_1)^2H_{2,0}(x_2)+H_{2,0}(x_1)H_1(x_2)^2) +\alpha^4H_1(x_1)^2H_1(x_2)^2\right )
\end{align*}
where $H_{i,j}$ are defined in Equation \eqref{eq:H1}; therefore, for any $t_1,t_2\leq\alpha$,  using the fact that $H_{2,0} \leq H_1$ and that $\sqrt{\sum_{i=1}^p a_i} \leq \sqrt{p} \sum_{i=1}^p\sqrt{a_i}$
\begin{align*}
v_\alpha\sqrt{\mathbb E\left\vert D_{z_1}F_\alpha\right\vert^2\left\vert D_{z_2}F_\alpha\right\vert^2}&\leq \sqrt{6C}\sum_{q_1=0}^2\sum_{q_2=0}^2  (\alpha H_1(x_1))^{q_1/2}(\alpha H_1(x_2))^{q_2/2} e^{-\frac{\alpha}{6}(\mu(x_1)\1{q_1=0}+\mu(x_2)\1{q_2=0})}
\end{align*}

Additionally, from Equation \eqref{eq:longinequality}, we have
\begin{align*}
&v_\alpha\int \sqrt{E\left ((D_{z_{1},z_{3}%
}^{2}F_\alpha)^{2}(D_{z_{2},z_{3}}^{2}F_\alpha)^{2}\right )}dx_3 \nonumber\\
&\leq C\times\Bigg ( \underbrace{\int W(x_1,x_3)W(x_2,x_3)(e^{-\frac{\alpha}{4}(\mu(x_1)+\mu(x_2))}+e^{-\alpha\mu(x_3)/4})dx_3}_{B_{\alpha,1}(x_1,x_2)}\Bigg .\nonumber\\
&\quad\quad+\underbrace{\int(\alpha H_{1,1}(x_2,x_3)+\sqrt{\alpha H_{2,2}(x_2,x_3)}) W(x_1,x_3)(e^{-\alpha/6 \mu(x_1)}+e^{-\alpha/6 \mu(x_3))})dx_3}_{B_{\alpha,2}(x_1,x_2)}\nonumber\\
&\qquad+\int(\alpha H_{1,1}(x_1,x_3)+\sqrt{\alpha H_{2,2}(x_1,x_3)}) W(x_2,x_3)(e^{-\alpha/6 \mu(x_2)}+e^{-\alpha/6 \mu(x_3))})dx_3\nonumber\\
&\qquad\Bigg .+\underbrace{\int\sqrt{A_2(x_1,x_2,x_3)}dx_3}_{B_{\alpha,3}(x_1,x_2)}\Bigg )
\end{align*}
for some constant $C$. To show that $\gamma_{\alpha,1}\to0$, we aim to show that, for any $q_1,q_2\in\{0,1,2\}$, and any $k=1,2,3$
\begin{align*}
I_{\alpha,k}(q_1,q_2)&:=\int (\alpha H_1(x_1))^{q_1/2}(\alpha H_1(x_2))^{q_2/2} e^{-\frac{\alpha}{6}(\mu(x_1)\1{q_1=0}+\mu(x_2)\1{q_2=0})}  B_{\alpha,k}(x_1,x_2)dx_1dx_2\\
&=o(\alpha^{4\sigma-1}\ell_\sigma(\alpha)^4)
\end{align*}
Consider first
\begin{align*}
I_{\alpha,1}(q_1,q_2)&=\int (\alpha H_1(x_1))^{q_1/2}(\alpha H_1(x_2))^{q_2/2} e^{-\frac{\alpha}{6}(\mu(x_1)\1{q_1=0}+\mu(x_2)\1{q_2=0})} \\
&\quad\quad\times W(x_1,x_3)W(x_2,x_3)(e^{-\frac{\alpha}{4}(\mu(x_1)+\mu(x_2))}+e^{-\alpha\mu(x_3)/4})dx_1dx_2dx_3\\
&\leq \int (\alpha H_1(x_1))^{q_1/2}L(x_1)\mu(x_1)e^{-\frac{\alpha}{4}\mu(x_1)}dx_1
\int (\alpha H_1(x_2))^{q_2/2}L(x_2)\mu(x_2)e^{-\frac{\alpha}{4}\mu(x_2)}dx_2
\\
&\quad+\int (\alpha H_1(x_1))^{q_1/2}(\alpha H_1(x_2))^{q_2/2}W(x_1,x_3)W(x_2,x_3) e^{-\alpha\mu(x_3)/4}dx_1dx_2dx_3
\end{align*}

For $q_1,q_2\geq 1$, using H\"older's inequality and Assumptions \ref{assumpt:1} and \ref{assumpt:5},
\begin{align*}
&\int (\alpha H_1(x_1))^{q_1/2}(\alpha H_1(x_2))^{q_2/2}W(x_1,x_3)W(x_2,x_3) e^{-\alpha/4\mu(x_3)}dx_1dx_2dx_3\\
&\leq \left (\int(\alpha H_1(x_1))^{q_1}dx_1 \int L(x_3)^2\mu(x_3)^2 e^{-\alpha\mu(x_3)/4}dx_3 \right . \\
& \quad\quad\times\left . \int(\alpha H_1(x_2))^{q_2}dx_2 \int L(x_3)^2\mu(x_3)^2 e^{-\alpha\mu(x_3)/4}dx_3   \right )^{1/2}\\
&=O(\alpha^{(q_1/2+q_2/2+1)\sigma -2 }\ell_\sigma(\alpha)^{q_1/2+q_2/2+1})
\end{align*}
Similarly,
\begin{align*}
\int (\alpha H_1(x_1))^{q_1/2}W(x_1,x_3)W(x_2,x_3) e^{-\alpha\mu(x_3)/4}dx_1dx_2dx_3&=O(\alpha^{(q_1/2+1/2)\sigma -2 }\ell_\sigma(\alpha)^{q_1/2+1/2})\\
\int W(x_1,x_3)W(x_2,x_3) e^{-\alpha\mu(x_3)/4}dx_1dx_2dx_3&=O(\alpha^{\sigma -2 }\ell_\sigma(\alpha))
\end{align*}
and it follows that for any $q_1,q_2\in\{0,1,2\}$,
$
I_{\alpha,1}(q_1,q_2)=O(\alpha^{3\sigma-2}\ell_\sigma(\alpha)^3)=o(\alpha^{-1+4\sigma}\ell_\sigma(\alpha)^2)\label{eq:I1}
$ as required. Consider now
\begin{align*}
I_{\alpha,2}(q_1,q_2)&=\int (\alpha H_1(x_1))^{q_1/2}(\alpha H_1(x_2))^{q_2/2} e^{-\frac{\alpha}{6}(\mu(x_1)\1{q_1=0}+\mu(x_2)\1{q_2=0})}\\
&\quad\quad\times(\alpha H_{1,1}(x_2,x_3)+\sqrt{\alpha H_{2,2}(x_2,x_3)}) W(x_1,x_3)(e^{-\alpha/6 \mu(x_1)}+e^{-\alpha/6 \mu(x_3))})dx_1dx_2dx_3.
\end{align*}
We have, using Lemma \ref{lemma:bound}
\begin{align*}
I_{\alpha,2,1}(q_1,q_2):=&\int (\alpha H_1(x_1))^{q_1/2}(\alpha H_1(x_2))^{q_2/2} e^{-\frac{\alpha}{6}(\mu(x_1)\1{q_1=0}+\mu(x_2)\1{q_2=0})}\\
&\quad\quad\times \alpha H_{1,1}(x_2,x_3) W(x_1,x_3)e^{-\alpha/6 \mu(x_1)}dx_1dx_2dx_3\\
&\leq \alpha\left(\int (\alpha H_1(x_1))^{q_1/2} e^{-\frac{\alpha}{6}\mu(x_1)\1{q_1=0}}L(x_1)\mu(x_1)e^{-\alpha/6 \mu(x_1)}dx_1 \right)\\
&\quad\times \left( \int (\alpha H_1(x_2))^{q_2/2} e^{-\frac{\alpha}{6}\mu(x_2)\1{q_2=0}}W(x_2,y)L(y)\mu(y) e^{-\frac{\alpha}{4}\mu(y)}dx_2dy \right)
\end{align*}
If $q_1=0$,
$$\int e^{-\frac{\alpha}{3}\mu(x_1)}L(x_1)\mu(x_1)dx_1 =O(\alpha^{\sigma-1}\ell_\sigma(\alpha))$$
and if $q_1\geq 1$, using the bound \eqref{eq:boundH1a},
\begin{equation}
\int (\alpha H_1(x_1))^{q_1/2} L(x_1)\mu(x_1) e^{-\frac{\alpha}{6}\mu(x_1)}dx_1 =O(\alpha^{(q_1+1)\sigma/2-1}\ell_\alpha(\alpha)^{q_1/2+1/2})\label{eq:boundH1d}
\end{equation}
If $q_2=0$,
\begin{align*}
&\int  e^{-\frac{\alpha}{6}\mu(x_2)}W(x_2,y)L(y)\mu(y) e^{-\frac{\alpha}{4}\mu(y)}dx_2dy\leq \int  L(y)\mu(y)^2 e^{-\frac{\alpha}{4}\mu(y)}dy=O(\alpha^{\sigma-2}\ell_\sigma(\alpha))
\end{align*}
If $q_2\geq 1$, using H\"older's inequality, the bound \eqref{eq:boundintH1} and Assumptions \ref{assumpt:5} and \ref{assumpt:1},
\begin{align*}
&\int (\alpha H_1(x_2))^{q_2/2}W(x_2,y)L(y)\mu(y) e^{-\frac{\alpha}{4}\mu(y)}dx_2dy\\
&\leq \left(\int (\alpha H_1(x_2))^{q_2} \right )^{1/2} \int \left (\int W(x_2,y)^2dx_2 \right)^{1/2}L(y)\mu(y) e^{-\frac{\alpha}{4}\mu(y)}dy\\
&\leq \left(\int (\alpha H_1(x_2))^{q_2} \right )^{1/2} \int L(y)^2\mu(y)^2 e^{-\frac{\alpha}{4}\mu(y)}dy
=O(\alpha^{(q_2/2+1)\sigma-2}\ell_\sigma(\alpha^{q_2/2+1}))
\end{align*}
Hence $I_{\alpha,2,1}(q_1,q_2)=o(\alpha^{4\sigma-1} \ell_\sigma^2(\alpha))$. Consider now
\begin{align*}
I_{\alpha,2,2}(q_1,q_2)&:=\int (\alpha H_1(x_1))^{q_1/2}(\alpha H_1(x_2))^{q_2/2} e^{-\frac{\alpha}{6}(\mu(x_1)\1{q_1=0}+\mu(x_2)\1{q_2=0})}\\
&\quad\quad\times\alpha H_{1,1}(x_2,x_3) W(x_1,x_3)e^{-\alpha/6 \mu(x_3)}dx_1dx_2dx_3
\end{align*}
If $q_1=0$, we obtain the following bound, using the same computations as above,
\begin{align*}
I_{\alpha,2,2}(q_1,q_2)&\leq \alpha\int  e^{-\frac{\alpha}{6}\mu(x_1)}L(x_1)\mu(x_1)dx_1  \int (\alpha H_1(x_2))^{q_2/2} e^{-\frac{\alpha}{6}\mu(x_2)\1{q_2=0}}W(x_2,y)L(y)\mu(y) e^{-\frac{\alpha}{4}\mu(y)}dx_2dy\\
&=O(\alpha^{4\sigma-2}\ell_\sigma(\alpha)^{4\sigma})
\end{align*}
If $q_1>0$, we have
\begin{align*}
I_{\alpha,2,2}(q_1,q_2)&\leq \int \left(\int (\alpha H_1(x_1))^{q_1}dx_1\right )^{1/2}\left(\int W(x_1,x_3)^2 dx_1\right )^{1/2} (\alpha H_1(x_2))^{q_2/2} e^{-\frac{\alpha}{6}\mu(x_2)\1{q_2=0}}\\
&\quad\quad\times\alpha H_{1,1}(x_2,x_3) e^{-\alpha/6 \mu(x_3)}dx_2dx_3\\
&\leq \left(\int (\alpha H_1(x_1))^{q_1}dx_1\right )^{1/2}\int L(x_3)\mu(x_3)(\alpha H_1(x_2))^{q_2/2} e^{-\frac{\alpha}{6}\mu(x_2)\1{q_2=0}}\\
&\quad\quad\times\alpha H_{1,1}(x_2,x_3) e^{-\alpha/6 \mu(x_3)}dx_2dx_3\\
\end{align*}
If $q_2=0$, noting that $H_{1,1}(x_2,x_3)\leq L(x_2)\mu(x_2)L(x_3)\mu(x_3)$, we obtain
\begin{align*}
I_{\alpha,2,2}(q_1,q_2)&\leq \alpha\left(\int (\alpha H_1(x_1))^{q_1}dx_1\right )^{1/2}\int L(x_2)\mu(x_2)L(x_3)^2\mu(x_3)^2 e^{-\frac{\alpha}{6}\mu(x_2)}e^{-\alpha/6 \mu(x_3)}dx_2dx_3\\
&=O(\alpha^{(q_1/2+2)\sigma-2}\ell_\sigma(\alpha)^{q_1/2+2})
\end{align*}
If $q_2>0$, using H\"older's inequality and the bound \eqref{eq:boundH1c},
\begin{align*}
I_{\alpha,2,2}
&\leq  \alpha\left(\int (\alpha H_1(x_1))^{q_1}dx_1\right )^{1/2}\left(\int (\alpha H_1(x_2))^{q_2}dx_2\right )^{1/2}\\
&\quad\quad\times \left (\int L(x_3)^2\mu(x_3)^2e^{-\alpha/3 \mu(x_3)}dx_3 \right )^{1/2} \int L(y)^2\mu(y)^2e^{-\frac{\alpha}{4}\mu(y)} dy=O(\alpha^{7\sigma/2-3}\ell(\alpha)^{7\sigma/2})
\end{align*}
Now, using H\"older and the fact that $H_{2,2} \leq H_{1,1}$,  together with  \eqref{eq:boundH1a},
\begin{align*}
I_{\alpha,2,3}(q_1,q_2)&:=\int (\alpha H_1(x_1))^{q_1/2}(\alpha H_1(x_2))^{q_2/2} e^{-\frac{\alpha}{6}(\mu(x_1)\1{q_1=0}+\mu(x_2)\1{q_2=0})}\\
&\quad\quad\times(\sqrt{\alpha H_{2,2}(x_2,x_3)}) W(x_1,x_3)e^{-\alpha/6 \mu(x_1)}dx_1dx_2dx_3\\
&\leq \sqrt{\alpha}\int (\alpha H_1(x_1))^{q_1/2} L(x_1)\mu(x_1)e^{-\alpha \mu(x_1)/6}dx_1\\
&\quad\quad\times \int (\alpha H_1(x_2))^{q_2/2}\left (\int \mu(y)W(x_2,y)e^{-\frac{\alpha}{4}\mu(y)}dy \right)^{1/2}e^{-\frac{\alpha}{6}\mu(x_2)\1{q_2=0}} dx_2\\
&=O(\alpha^{(3+1/4)\sigma-3/2}\ell_\sigma(\alpha)^{3+1/4})
\end{align*}

Consider
\begin{align*}
I_{\alpha,2,4}(q_1,q_2)&:=\int (\alpha H_1(x_1))^{q_1/2}(\alpha H_1(x_2))^{q_2/2} e^{-\frac{\alpha}{6}(\mu(x_1)\1{q_1=0}+\mu(x_2)\1{q_2=0})}\\
&\quad\quad\times\sqrt{\alpha H_{2,2}(x_2,x_3)} W(x_1,x_3)e^{-\alpha/6 \mu(x_3)}dx_1dx_2dx_3\\
\end{align*}
If $q_1=0$, then, using the above computations,
\begin{align*}
I_{\alpha,2,4}(q_1,q_2)&\leq \int (\alpha H_1(x_2))^{q_2/2} e^{-\frac{\alpha}{6}(\mu(x_1)+\mu(x_2)\1{q_2=0})}\\
&\quad\quad\times\sqrt{\alpha H_{1,1}(x_2,x_3)} W(x_1,x_3)dx_1dx_2dx_3\\
&=O(\alpha^{(3+1/4)\sigma-3/2}\ell_\sigma(\alpha)^{3+1/4})
\end{align*}
If $q_1>0$ and $q_2=0$, noting that $H_{2,2}(x_2,x_3)\leq L(x_2)^2\mu(x_2)^2L(x_3)^2\mu(x_3)^2$ and using H\"older's inequality and Assumptions \ref{assumpt:5} and \ref{assumpt:1},
\begin{align*}
I_{\alpha,2,4}(q_1,q_2)&\leq \sqrt{\alpha}\left (\int (\alpha H_1(x_1))^{q_1}dx_1\right )^{1/2} \int L(x_3)\mu(x_3) e^{-\frac{\alpha}{6}\mu(x_2)}\sqrt{ H_{2,2}(x_2,x_3)}e^{-\alpha/6 \mu(x_3)}dx_2dx_3\\
&\leq \sqrt{\alpha}\left (\int (\alpha H_1(x_1))^{q_1}dx_1\right )^{1/2}\int L(x_3)^{2}\mu(x_3)^{2}e^{-\alpha/6 \mu(x_3)}dx_3\int  e^{-\frac{\alpha}{6}\mu(x_2)}\mu(x_2)L(x_2)dx_2\\
&=O(\alpha^{3\sigma -3}\ell_\sigma(\alpha)^{3})
\end{align*}

If $q_1,q_2>0$, noting that $\int H_{1,1}(x_2,x_3)=\int \mu(y)^2e^{-\frac{\alpha}{4}\mu(y)}dy=O(\alpha^{\sigma-2}\ell_\sigma(\alpha))$,
\begin{align*}
&I_{\alpha,2,4}(q_1,q_2)\\
&\leq \sqrt{\alpha}\left (\int (\alpha H_1(x_1))^{q_1}dx_1\right )^{1/2}\times \int L(x_3)\mu(x_3) (\alpha H_1(x_2))^{q_2/2}\sqrt{ H_{1,1}(x_2,x_3)}e^{-\alpha/6 \mu(x_3)}dx_2dx_3\\
&\leq \sqrt{\alpha}\left (\int (\alpha H_1(x_1))^{q_1}dx_1\right )^{1/2}\left (\int (\alpha H_1(x_2))^{q_2}dx_2\right )^{1/2}\\
&\quad\quad\times \left(\int L(x_3)^2\mu(x_3)^2e^{-\alpha/6 \mu(x_3)}dx_3\right)^{1/2} \left (\int H_{1,1}(x_2,x_3)dx_3dx_2 \right)^{1/2}=O(\alpha^{3\sigma -3/2}\ell_\sigma(\alpha^{3\sigma}))
\end{align*}
It follows that
\begin{align}
I_{\alpha,2}(q_1,q_2)=o(\alpha^{-3}v_\alpha^2).\label{eq:I2}
\end{align}

Finally, consider
\begin{align*}
I_{\alpha,3}(q_1,q_2)=\int(\alpha H_1(x_1))^{q_1/2}(\alpha H_1(x_2))^{q_2/2} e^{-\frac{\alpha}{6}(\mu(x_1)\1{q_1=0}+\mu(x_2)\1{q_2=0})}\sqrt{A_2(x_1,x_2,x_3)}dx_1dx_2dx_3
\end{align*}

where
\begin{align*}
&\sqrt{A_2(x_1,x_2,x_3)}\\&\leq \alpha^2 H_{1,1}(x_1,x_3)H_{1,1}(x_2,x_3) + \alpha^{3/2} \left ( H_{1,1}(x_2,x_3) \sqrt{ H_{2,2}(x_1,x_3) }  +H_{1,1}(x_1,x_3)\sqrt{ H_{2,2}(x_2,x_3) } \right)\\
&+\alpha \left (\sqrt{ H_{2,2}(x_1,x_3) }\sqrt{ H_{2,2}(x_2,x_3) } \right.\\
&\quad \quad\left . + \sqrt{W(y_1,x_2)W(y_1,x_1)W(y_1,x_3)^2W(y_2,x_1)W(y_2,x_2)W(y_2,x_3)^2   ) e^{-\alpha/2 (\sum_{i=1}^2\mu(y_i))}dy_{1:2}} \right)\\
&+\sqrt{\alpha} \left ( \int  W(y_1,x_1)^2W(y_1,x_2)^2W(y_1,x_3)^4  e^{-\alpha \mu(y_1)}dy_1\right )^{1/2}
\end{align*}
Using Assumption \ref{assumpt:5},
\begin{align*}
&\int  H_{1,1}(x_1,x_3))H_{1,1}(x_2,x_3)dx_3
&\leq \int W(x_1,y_1)L(y_1)\mu(y_1)e^{-\frac{\alpha}{4}\mu(y_1)}dy_1\int W(x_2,y_2)L(y_2)\mu(y_2)e^{-\frac{\alpha}{4}\mu(y_2)}dy_2
\end{align*}

Therefore, using the asymptotic bounds \eqref{eq:boundH1a} and Assumption \ref{assumpt:1},
\begin{align*}
&I_{\alpha,3,1}(q_1,q_2)\\
&:= \alpha^2\int(\alpha H_1(x_1))^{q_1/2}(\alpha H_1(x_2))^{q_2/2} e^{-\frac{\alpha}{6}(\mu(x_1)\1{q_1=0}+\mu(x_2)\1{q_2=0})}  H_{1,1}(x_1,x_3)H_{1,1}(x_2,x_3)dx_1dx_2dx_3\\
&\leq \alpha^2\int(\alpha H_1(x_1))^{q_1/2}(\alpha H_1(x_2))^{q_2/2}   H_{1,1}(x_1,x_3)H_{1,1}(x_2,x_3)dx_1dx_2dx_3\\
&\leq \alpha^2\int (\alpha H_1(x_1))^{q_1/2} L(y_1)\mu(y_1)W(x_1,y_1)e^{-\frac{\alpha}{4}\mu(y_1)}dy_1dx_1\\
&\quad\times\int (\alpha H_1(x_2))^{q_2/2} L(y_2)\mu(y_2)W(x_2,y_2) e^{-\frac{\alpha}{4}\mu(y_2)}dy_2dx_2\\
\end{align*}
for any $q_1,q_2\leq 2$. If $q=0$,
\begin{align*}
\int (\alpha H_1(x))^{q/2} L(y)\mu(y)W(x,y)e^{-\frac{\alpha}{4}\mu(y)}dydx=\int  L(y)\mu(y)^2e^{-\frac{\alpha}{4}\mu(y)}dy=O(\alpha^{\sigma-2}\ell_\sigma(\alpha))
\end{align*}
If $q>0$, noting that, using H\"older's inequality and Assumption 3,
\begin{align*}
\int (\alpha H_1(x))^{q/2}W(x,y)dx     \leq L(y)\mu(y)\left (\int (\alpha H_1(x))^q dx\right )^{1/2}
\end{align*}
we have
\begin{align*}
\int (\alpha H_1(x))^{q/2} L(y)\mu(y)W(x,y)e^{-\frac{\alpha}{4}\mu(y)}dydx&\leq \left (\int (\alpha H_1(x))^q dx\right )^{1/2}\left( \int  L(y)^2\mu(y)^2e^{-\frac{\alpha}{4}\mu(y)}dy \right)^{1/2}\\
=O(\alpha^{(q/2+1)\sigma-2}\ell_\sigma(\alpha)^{q/2+1}).
\end{align*}
It follows that $I_{\alpha,3,1}(q_1,q_2)=O(\alpha^{4\sigma-2}\ell(\alpha)^{4\sigma})$ for any $q_1,q_2\in\{0,1,2\}$. Consider now
\begin{align*}
I_{\alpha,3,2}(q_1,q_2):=&  \alpha^{3/2} \int(\alpha H_1(x_1))^{q_1/2}(\alpha H_1(x_2))^{q_2/2}e^{-\frac{\alpha}{6}(\mu(x_1)\1{q_1=0}+\mu(x_2)\1{q_2=0})} H_{1,1}(x_2,x_3) \\
&\quad\times\left (\int  W(y_1,x_1)^2W(y_1,x_3)^2 e^{-\alpha/3 \mu(y_1)} dy_1\right )^{1/2}dx_1dx_2dx_3\\
&\leq  \alpha^{3/2}  \int(\alpha H_1(x_1))^{q_1/2}(\alpha H_1(x_2))^{q_2/2} H_{1,1}(x_2,x_3) \\
&\quad\times\left (\int  W(y_1,x_1)^2W(y_1,x_3)^2 e^{-\alpha/3 \mu(y_1)} dy_1\right )^{1/2}dx_1dx_2dx_3
\end{align*}

Using Lemma \ref{lemma:bound},
\begin{align*}
&\int H_{1,1}(x_2,x_3) \left (\int  W(y_1,x_1)^2W(y_1,x_3)^2 e^{-\alpha/3 \mu(y_1)} dy_1\right )^{1/2}dx_3\\
&\leq \int W(x_2,y_2)L(y_2)\mu(y_2)e^{-\frac{\alpha}{4}\mu(y_2)}dy_2 \times \left (\int L(y_1)^{2}\mu(y_1)^{2} W(y_1,x_1)^{2} e^{-\frac{\alpha}{3}\mu(y_1) } dy_1 \right )^{1/2}
\end{align*}
Therefore
\begin{align*}
I_{\alpha,3,2}(q_1,q_2)&\leq \alpha^{3/2}  \int  (\alpha H_1(x_2))^{q_1/2} W(x_2,y_2)L(y_2)\mu(y_2)e^{-\frac{\alpha}{4}\mu(y_2)}dy_2dx_2 \\
&\quad\times \int (\alpha H_1(x_1))^{q_2/2} \left (\int L(y_1)^{2}\mu(y_1)^{2} W(y_1,x_1)^{2} e^{-\frac{\alpha}{3}\mu(y_1) } dy_1 \right )^{1/2}dx_1
\end{align*}
For $q_1=0$,
$$\int  L(y_2)\mu(y_2)^2e^{-\frac{\alpha}{4}\mu(y_2)}dy_2=O(\alpha^{\sigma-2}\ell_\sigma(\alpha)),$$
while for $q_1\geq 1$,
\begin{align*}
&\int  (\alpha H_1(x_2))^{q_1/2} W(x_2,y_2)L(y_2)\mu(y_2)e^{-\frac{\alpha}{4}\mu(y_2)}dy_2dx_2\\
&\leq \left (\int (\alpha H_1(x_2))^{q_1} dx_2 \times \int W(x_2,y_2)^2L(y_2)^2\mu(y_2)^2e^{-\frac{\alpha}{4}\mu(y_2)}dy_2dx_2   \right )^{1/2}\\
&\leq \left (\int (\alpha H_1(x_2))^{q_1} dx_2 \times \int L(y_2)^4\mu(y_2)^4e^{-\frac{\alpha}{4}\mu(y_2)}dy_2dx_2   \right )^{1/2}=O(\alpha^{(q_1/2+1/2)\sigma-2}\ell_\sigma(\alpha)^{q_1+1/2}).
\end{align*}
Additionally, for $q_2=0$, using H\"older's inequality and Assumptions \ref{assumpt:1} and \ref{assumpt:5},
\begin{align*}
&\int \left (\int L(y_1)^{2}\mu(y_1)^{2} W(y_1,x_1)^{2} e^{-\frac{\alpha}{3}\mu(y_1) } dy_1 \right )^{1/2}dx_1\\
&\leq \int \left ( \int L(y_1)^{4}\mu(y_1)^{4}  e^{-\frac{2\alpha}{3}\mu(y_1) } dy_1 \int W(y_1,x_1)^{4} dy_1 \right )^{1/4}dx_1\\
&\leq \int L(x_1)\mu(x_1) dx_1 \left ( \int L(y_1)^{4}\mu(y_1)^{4}  e^{-\frac{2\alpha}{3}\mu(y_1) } dy_1 \right )^{1/4}=O(\alpha^{\sigma/4-1}\ell_\sigma(\alpha)^{1/4})
\end{align*}
while for $q_2\geq 1$
\begin{align*}
&\int (\alpha H_1(x_1))^{q_2/2} \left (\int L(y_1)^{2}\mu(y_1)^{2} W(y_1,x_1)^{2} e^{-\frac{\alpha}{3}\mu(y_1) } dy_1 \right )^{1/2}dx_1\\
&\leq \left ( \int(\alpha H_1(x_1))^{q_2}dx_1\int L(y_1)^{2}\mu(y_1)^{2} W(y_1,x_1)^{2} e^{-\frac{\alpha}{3}\mu(y_1) } dy_1dx_1    \right )^{1/2}\\
&\leq \left ( \int(\alpha H_1(x_1))^{q_2}dx_1\int L(y_1)^{4}\mu(y_1)^{4}e^{-\frac{\alpha}{3}\mu(y_1) } dy_1    \right )^{1/2}=O(\alpha^{(q_2+1/2)\sigma-2}\ell_\sigma(\alpha)^{q_2/2+1/2}).
\end{align*}
Hence, for any $q_1,q_2\leq 2$, $I_{\alpha,3,2}=O(\alpha^{3\sigma - 2}\ell_\sigma(\alpha)^3)=o(\alpha^{-3}v_\alpha^2).$ Consider now
\begin{align*}
I_{\alpha,3,3}(q_1,q_2)&:= \alpha\int(\alpha H_1(x_1))^{q_1/2}(\alpha H_1(x_2))^{q_2/2}e^{-\frac{\alpha}{6}(\mu(x_1)\1{q_1=0}+\mu(x_2)\1{q_2=0})}   \\
&\quad\times \left [\left ( \int  W(y_1,x_1)^2W(y_1,x_3)^2W(y_2,x_2)^2W(y_2,x_3)^2e^{-\alpha/2 (\sum_{i=1}^2\mu(y_i))}dy_{1:2}\right )^{1/2}\right .\\
&\quad+\left .\left ( \int W(y_1,x_2)W(y_1,x_1)W(y_1,x_3)^2W(y_2,x_1)W(y_2,x_2)W(y_2,x_3)^2    e^{-\frac{\alpha}{2 (\sum_{i=1}^2\mu(y_i))}}dy_{1:2}\right )^{1/2}\right ]\\
&\quad\quad dx_1dx_2dx_3
\end{align*}
Using H\"older's inequality,
\begin{align*}
 &\int \left (\int  W(y_1,x_1)^2W(y_1,x_3)^2W(y_2,x_2)^2W(y_2,x_3)^2e^{-\alpha/2 (\sum_{i=1}^2\mu(y_i))}dy_{1:2}\right)^{1/2}dx_3\\
 &\leq \left (\int  W(y_1,x_1)^2W(y_1,x_3)^2e^{-\frac{\alpha}{2}\mu(y_1)}dy_{1}dx_3\right)^{1/2}\left (\int  W(y_2,x_2)^2W(y_2,x_3)^2e^{-\frac{\alpha}{2}\mu(y_2)}dy_{2}dx_3\right)^{1/2}\\
\end{align*}
For $q=0$, using Assumption \ref{assumpt:5} and \ref{assumpt:1}
\begin{align*}
\int e^{-\frac{\alpha}{6}\mu(x)}&\left (\int H_{2,2}(x,x_3) dx_3\right)^{1/2}dx\leq \int e^{-\frac{\alpha}{6}\mu(x)}\left (\int  W(y,x)^2W(y,x_3)dydx_3\right)^{1/2}dx\\
&=O(\alpha^{\sigma-1}\ell_\sigma(\alpha))
\end{align*}
Additionally,
\begin{align*}
 & \left (\int  W(y_1,x_1)^2W(y_1,x_3)^2e^{-\frac{\alpha}{2}\mu(y_1)}dy_{1}dx_3\right)^{1/2}\left (\int  W(y_2,x_2)^2W(y_2,x_3)^2e^{-\frac{\alpha}{2}\mu(y_2)}dy_{2}dx_3\right)^{1/2}\\
 &\leq \left (\int  L(y_1)^2\mu(y_1)^2 W(y_1,x_1)^2 e^{-\frac{\alpha}{2}\mu(y_1)}dy_{1}\right)^{1/2}\left (\int  L(y_2)^2\mu(y_2)^2W(y_2,x_2)^2e^{-\frac{\alpha}{2}\mu(y_2)}dy_{2}\right)^{1/2}
\end{align*}
It follows that, for $q\geq 1$, using H\"older's inequality and Assumptions \ref{assumpt:5} and \ref{assumpt:1}
\begin{align*}
&\int (\alpha H_1(x))^{q/2}\left (\int  L(y)^2\mu(y)^2 W(y,x)^2 e^{-\frac{\alpha}{2}\mu(y)}dy\right)^{1/2}dx\\
&\leq \left ( \int (\alpha H_1(x))^{q}dx \right )^{1/2}\left(\int  L(y)^4\mu(y)^4 e^{-\frac{\alpha}{2}\mu(y)}dy \right)^{1/2}=O\left (\alpha^{(q/2+1/2)\sigma-2}\ell_\sigma(\alpha)^{(q/2+1/2)}\right )
\end{align*}
Similarly, for the second term of $I_{\alpha,3,3}(q_1,q_2)$
\bgroup
\allowdisplaybreaks
\begin{align*}
&\int\left ( \int W(y_1,x_2)W(y_1,x_1)W(y_1,x_3)^2W(y_2,x_1)W(y_2,x_2)W(y_2,x_3)^2    e^{-\alpha/2 (\sum_{i=1}^2\mu(y_i))}dy_{1:2}\right )^{1/2}dx_3\\
&\leq \left (\int  W(y_1,x_1)W(y_1,x_2)W(y_1,x_3)^2e^{-\frac{\alpha}{2}\mu(y_1)}dy_{1}dx_3\int  W(y_2,x_1)W(y_2,x_2)W(y_2,x_3)^2e^{-\frac{\alpha}{2}\mu(y_2)}dy_{2}dx_3\right)^{1/2}
\end{align*}
\egroup
and, using H\"older's inequality and \eqref{eq:boundH1a},
\bgroup
\allowdisplaybreaks
\begin{align*}
&\int(\alpha H_1(x_1))^{q_1/2}(\alpha H_1(x_2))^{q_2/2}e^{-\frac{\alpha}{6}(\mu(x_1)\1{q_1=0}+\mu(x_2)\1{q_2=0})}   \\
&\left (\int  W(y_1,x_1)W(y_1,x_2)W(y_1,x_3)^2e^{-\frac{\alpha}{2}\mu(y_1)}dy_{1}dx_3\right)^{1/2}\\
&\quad\times\left (\int  W(y_2,x_1)W(y_2,x_2)W(y_2,x_3)^2e^{-\frac{\alpha}{2}\mu(y_2)}dy_{2}dx_3\right)^{1/2}dx_1dx_2\\
&\leq \left (\int(\alpha H_1(x_1))^{q_1}e^{-\frac{\alpha}{6}\mu(x_1)\1{q_1=0}}W(y_1,x_1)\mu(y_1)^3L(y_1)^2e^{-\frac{\alpha}{2}\mu(y_1)}dy_{1}dx_1  \right )^{1/2}  \\
&\quad\times \left (\int(\alpha H_1(x_2))^{q_2}e^{-\frac{\alpha}{6}\mu(x_2)\1{q_2=0}}W(y_2,x_2)\mu(y_2)^3L(y_2)^2e^{-\frac{\alpha}{2}\mu(y_2)}dy_{2}dx_2  \right )^{1/2}
\end{align*}
\egroup
For $q=0$,
\begin{align*}
&\left (\int e^{-\frac{\alpha}{6}\mu(x)}W(y,x)\mu(y)^3L(y)^2e^{-\frac{\alpha}{2}\mu(y)}dydx  \right )^{1/2}
=O(\alpha^{\sigma/2-2}\ell_\sigma(\alpha)^{1/2})
\end{align*}
while for $q\geq 1$,
\begin{align*}
&\left (\int (\alpha H_1(x))^qW(y,x)\mu(y)^3L(y)^2e^{-\frac{\alpha}{2}\mu(y)}dydx  \right )^{1/2}
&=O(\alpha^{(q/2+1/2)\sigma-3/2}\ell_\sigma(\alpha)^{q/2+1/2})
\end{align*}
Combining the above results, we obtain, for any $q_1,q_2\in\{0,1,2\}$,   $I_{\alpha,3,3}(q_1,q_2)=o(\alpha^{-3}v_\alpha^2).$ Consider finally
\begin{align*}
I_{\alpha,3,4}(q_1,q_2):=&\int(\alpha H_1(x_1))^{q_1/2}(\alpha H_1(x_2))^{q_2/2} e^{-\frac{\alpha}{6}(\mu(x_1)\1{q_1=0}+\mu(x_2)\1{q_2=0})} \\
&\times\sqrt{\alpha} \left ( \int  W(y_1,x_1)^2W(y_1,x_2)^2W(y_1,x_3)^4  e^{-\alpha \mu(y_1)}dy_1\right )^{1/2} dx_1dx_2dx_3
\end{align*}
We have
\begin{align*}
&\int \left ( \int  W(y_1,x_1)^2W(y_1,x_2)^2W(y_1,x_3)^4  e^{-\alpha \mu(y_1)}dy_1\right )^{1/2} dx_3\\
&\leq \left (\int L(x_3)\mu(x_3)dx_3\right )  \left (\int  W(y_1,x_1)^6  e^{-\alpha \mu(y_1)}dy_1\right)^{1/6} \left ( \int  W(y_1,x_2)^6  e^{-\alpha \mu(y_1)}dy_1  \right )^{1/6}
\end{align*}
and, for $q_1\geq1$, using H\"older's inequality,
\begin{align*}
&\int(\alpha H_1(x_1))^{q_1}  \left (\int  W(y_1,x_1)^6  e^{-\alpha \mu(y_1)}dy_1\right)^{1/6}dx_1=O(\alpha^{(q_1+1/6)\sigma-1}\ell_\sigma(\alpha)^{q_1+1/6}).
\end{align*}
For $q_1=0$,
\begin{align*}
\int  e^{-\frac{\alpha}{6}\mu(x_1)} \left (\int  W(y_1,x_1)^6  e^{-\alpha \mu(y_1)}dy_1\right)^{1/6}dx_1
&=O(\alpha^{\sigma-1}\ell_\sigma(\alpha))
\end{align*}
It follows that, for $q_1,q_2\leq 2$
\begin{align}
I_{\alpha,3,4}(q_1,q_2)=O(\alpha^{(4+1/3)\sigma-3/2 }\ell_\sigma(\alpha)^{4+1/3})=o(\alpha^{-3}v_\alpha^2)\label{eq:I34}\\
I_{\alpha,3}(q_1,q_2)=O(\alpha^{(4+1/3)\sigma-3/2 }\ell_\sigma(\alpha)^{4+1/3})=o(\alpha^{-3}v_\alpha^2)\label{eq:I3}
\end{align}
and, combining the bounds \eqref{eq:I1}, \eqref{eq:I2} and \eqref{eq:I3}, we obtain
$\gamma_{\alpha,1}\to0$.

\end{document}